\documentclass[jsl]{asl_2}

\usepackage{graphicx}
\usepackage{stmaryrd}

\usepackage{amscd}
\usepackage{mathptmx}
\usepackage{mathrsfs}

\usepackage{xcolor}
\usepackage{xspace}

\theoremstyle{plain}
\newtheorem{theorem}{Theorem}[section] 
\newtheorem{lem}[theorem]{Lemma}
\newtheorem{prop}[theorem]{Proposition}
\newtheorem{defn}[theorem]{Definition}
\newtheorem{cor}[theorem]{Corollary}
\newtheorem{exmp}[theorem]{Example}
\newtheorem{fact}[theorem]{Fact}
\newtheorem{rmk}[theorem]{Remark}

\usepackage[cal=boondox,scr=boondoxo]{mathalfa}

\usepackage{hyperref}

\usepackage{float}

\usepackage{bussproofs}

\usepackage{multirow}
\usepackage{cleveref}

\usepackage{enumerate}
\usepackage{enumitem}

\makeatletter
\newcommand{\leqnomode}{\tagsleft@true}
\newcommand{\reqnomode}{\tagsleft@false}
\makeatother

\usepackage{chngpage}

\usepackage{array}
\newcolumntype{C}[1]{>{\centering\arraybackslash}p{#1}}
\newcolumntype{L}[1]{>{\arraybackslash}p{#1}}
\newcolumntype{R}[1]{>\hfill p{#1}}

\DeclareSymbolFont{letters}{OML}{cmm}{m}{it}

\DeclareMathAlphabet{\mathcal}{OMS}{cmsy}{m}{n}

\usepackage{relsize}

\newcommand{\FOT}{\ensuremath{\mathbf{FOT}}\xspace}

\newcommand{\FOTd}{\ensuremath{\mathbf{FOT}^\downarrow}\xspace}

\newcommand{\D}{\ensuremath{\mathbf{D}}\xspace}

\newcommand{\FO}{\ensuremath{\mathbf{FO}}\xspace}

\newcommand{\Ind}{\ensuremath{\mathbf{Ind}}\xspace}
\newcommand{\ESO}{\ensuremath{\mathbf{ESO}}\xspace}

\newcommand{\FOp}{\ensuremath{\mathbf{FO}(\forallo)}\xspace}

\newcommand{\LL}{\ensuremath{\mathcal{L}}\xspace}

\newcommand{\dep}{\ensuremath{\mathop{=\!}}\xspace}
\newcommand{\con}{\ensuremath{\mathop{\dep}}\xspace}

\newcommand{\existso}{\ensuremath{\exists^1}\xspace}
\newcommand{\forallo}{\ensuremath{\forall^1}\xspace}
\newcommand{\wcn}{\ensuremath{\mathop{\dot\sim}}\xspace}
\newcommand{\cimp}{\ensuremath{\rightarrowtriangle}\xspace}
\newcommand{\cbimp}{\ensuremath{\leftrightarrowtriangle}\xspace}

\newcommand{\dblsetminus}{\mathbin{{\setminus}\mspace{-5mu}{\setminus}}}
\newcommand{\vvee}{\raisebox{1pt}{\ensuremath{\,\mathop{\mathsmaller{\mathsmaller{\dblsetminus\hspace{-0.23ex}/}}}\,}}}

\newcommand{\foralloi}{\ensuremath{$\forallo$\!\!\textsf{I}}\xspace}
\newcommand{\foralloe}{\ensuremath{$\forallo$\!\!\textsf{E}}\xspace}
\newcommand{\existsoi}{\ensuremath{$\existso$\!\!\textsf{I}}\xspace}
\newcommand{\existsoe}{\ensuremath{$\existso$\!\!\textsf{E}}\xspace}
\newcommand{\vveei}{\ensuremath{$\vvee$\!\textsf{I}}\xspace}
\newcommand{\vveee}{\ensuremath{$\vvee$\!\textsf{E}}\xspace}
\newcommand{\incpro}{\ensuremath{\subseteq\!\textsf{Pro}}\xspace}
\newcommand{\inccmp}{\ensuremath{\subseteq\!\textsf{Cmp}}\xspace}
\newcommand{\inctr}{\ensuremath{\subseteq\!\textsf{Tr}}\xspace}
\newcommand{\incid}{\ensuremath{\subseteq\!\textsf{Id}}\xspace}
\newcommand{\consi}{\ensuremath{\mathsf{con}\textsf{I}}\xspace}
\newcommand{\incwi}{\ensuremath{\subseteq\!\textsf{wI}_R}\xspace}
\newcommand{\incwe}{\ensuremath{\cn\!\subseteq\!\textsf{E}}\xspace}
\newcommand{\foawe}{\ensuremath{\cn\!\lambda\textsf{E}}\xspace}
\newcommand{\incwcon}{\ensuremath{\subseteq\!\textsf{W}_{\textsf{con}}}\xspace}
\newcommand{\incweo}{\ensuremath{\subseteq\!\textsf{W}_{\existso}}\xspace}
\newcommand{\raa}{\ensuremath{\textsf{RAA}}\xspace}

\def\confname{}
\def\copyrightyear{2000}

\title{Complete logics for elementary team properties}

\author{Juha Kontinen}  
\revauthor{Kontinen, Juha}

\address{PL 68 (Pietari Kalmin katu 5)\\
00014 University of Helsinki, Finland
} 
\email{juha.kontinen@helsinki.fi}

\author{Fan Yang}
\revauthor{Yang, Fan}

\address{PL 68 (Pietari Kalmin katu 5)\\
00014 University of Helsinki, Finland
} 
\email{fan.yang.c@gmail.com}

\thanks{This document and all that is needed to prepare
documents for ASL publications are posted in the ASL Typesetting
Office Website, \texttt{http:\weakslash /www.math\weakdot
ucla\weakdot edu/\urltilde asl/asltex}. August 20, 2000.}

\let\cn=\pkg

\begin{document}

\maketitle

\begin{abstract}
In this paper, we introduce a logic based on team semantics, called \FOT, whose expressive power is elementary, i.e., coincides with first-order logic both on the level of sentences and (possibly open) formulas, and we also show that a sublogic of \FOT, called \FOTd, captures exactly downward closed elementary (or first-order) team properties. We axiomatize completely the logic \FOT, and also extend the known partial axiomatization of dependence logic to dependence logic enriched with the logical constants in $\FOTd$.


\end{abstract}

\section{Introduction}

In this paper, we introduce logics based on team semantics for characterizing elementary (or first-order) team properties, and we also study the axiomatization problem of these logics.   

Team semantics is a semantical framework originally introduced by Hodges \cite{Hodges1997a}, 
  and later systematically developed by V\"a\"an\"anen with the introduction of {\em dependence logic} \cite{Van07dl}, which extends first-order logic with {\em dependence atoms}. Other
notable logics based on team semantics 
include  {\em independence logic} introduced by Gr\"adel and V\"a\"an\"anen \cite{D_Ind_GV} (which is first-order logic extended with independence atoms), and {\em inclusion logic} introduced by Galliani \cite{Pietro_I/E} (which is first-order logic extended with inclusion atoms).
In team semantics formulas are evaluated in a model over {\em sets} of assignments for the free variables (called {\em teams}) rather than single assignments as in the usual first-order logic. Teams $X$ with the domain $\{v_1,\dots,v_k\}$  are essentially $k$-ary relations $rel(X)=\{(s(v_1),\dots,s(v_k))\mid s\in X\}$, and thus open formulas define team properties. In general, knowing the expressive power of a logic for sentences (with no free variables) does not automatically give a characterization for the expressive  power of open formulas of the same logic.
Such a peculiar phenomenon has sparked several studies on the expressive power of logics based on team semantics. In particular, while it follows straightforwardly from the earlier known results of  Henkin, Enderton, Walkoe, and Hodges \cite{henkin61,Enderton1970,Walkoe1970,Hodges1997b} that dependence logic (\D) and independence logic (\Ind) are both equivalent to existential second-order logic ($\ESO$) on the level of sentences, it turns out that open formulas of \D have different expressive power from open formulas of \Ind: The latter characterize all $\ESO$ team properties \cite{Pietro_I/E}, whereas the former characterize only downward closed $\ESO$ team properties \cite{KontVan09}. Along the same line, a later breakthrough showed that inclusion logic corresponds, over sentences, to {\em positive greatest fixed-point logic} \cite{inclusion_logic_GH}, which is  strictly more expressive than first-order logic as well.  In this paper we define  a team-based logic, called \FOT,  whose expressive power coincides with first-order logic (\FO) both on the level of sentences and  open formulas, in the sense that \FOT-formulas characterize (modulo the empty team) exactly team properties definable by first-order sentences with an extra relation symbol $R$. 
To the best of the authors' knowledge, no such logic has been defined previously. 

In related previous work, it was shown in \cite{Pietro_thesis,Galliani2016,Luck18} that first-order logic  extended with {\em constancy atoms} $\dep(x)$ and \FO extended with classical negation $\sim$ are both equivalent to \FO over sentences, whereas on the level of formulas they are both strictly less expressive than \FO, and thus fail to capture all elementary (or first-order) team properties. It was also illustrated in \cite{Coherence_JK} that a certain simple disjunction  of dependence atoms already defines an NP-complete team property. Therefore, any logic based on team semantics having  the disjunction $\vee$ lifted 
from first-order logic and in which dependence atoms are expressible will be able to express NP-complete team properties, indicating that $\vee$ is a too expressive connective to be added to  \FOT. 
The logic \FOT  we define in this paper has weaker version of disjunction $\vvee$ and classical negation $\wcn$ as well as weaker quantifiers $\forallo,\existso$. We prove, in  \Cref{sec:fot-fotd-expr}, that our logic \FOT captures elementary team properties (modulo the empty team)
 and we also show, as by applying Lyndon's Interpolation Theorem of first-order logic,  that a sublogic of \FOT, denoted as \FOTd, captures exactly  downward closed elementary team properties (modulo the empty team).  
 
 From these results it follows immediately that \FOT and \FOTd are compact with respect to sentences (with no free variables). In fact, it also follows from \cite[Theorem 6.4]{Van07dl} that any team-based first-order logic  that is expressively less than or equal to \ESO is compact with respect to sentences.  In \Cref{compact} of  \Cref{sec:fot-fotd-expr} we provide a detailed proof of the compactness of such logics with respect to arbitrary (possibly open) formulas, 
 which is missing in the literature. For team-based first-order logics the compactness with respect to formulas is not a trivial consequence of the compactness with respect to sentences, as free variables in team-based logics are interpreted over teams, which are  essentially relations instead of single values. Our proof assumes that the set of all free variables occurring in the formulas in question is countable. It is unclear how to treat the uncountable case.  


In the second part of this paper we study the axiomatization problem of our logics \FOT and \FOTd. In  \Cref{sec:FOT} we 
  introduce a sound and complete system of natural deduction for \FOT that on one hand behaves like the system of \FO to a certain extent (in the sense of Lemma \ref{FO2FOT_eqiv}), while on the other hand incorporates natural and interesting rules for inclusion atoms and their interaction with the weak logical constants. 
It is also worth mentioning that inquisitive first-order logic (\textsf{InqBQ}) \cite{Ciardelli_PhD} adopts a similar (though technically different) type of team semantics. Our weak quantifiers $\forallo$ and $\existso$ are essentially the same as the quantifiers in  \textsf{InqBQ}, and our rules for these quantifiers  are also closely related to those ones in the deduction system of \textsf{InqBQ}, which is shown in \cite{Grilletti2021} to be complete for the classical antecedent fragment of \textsf{InqBQ}.
  

In \Cref{sec:fotd}, we apply our results to the problem of finding  axiomatizations for larger and larger fragments of  dependence logic and its variants by extending  the known partial axiomatization of dependence logic  to $\D$ enriched with the logical constants in $\FOTd$ (denoted as $\D\oplus\FOTd$), which, by our result in the first part of this paper, is expressively equivalent to \D. While \D is not effectively axiomatizable (for it is equivalent to \ESO), a complete axiomatization for first-order consequences of \D-sentences has been given in \cite{Axiom_fo_d_KV}. More precisely,   a system of natural deduction for dependence logic was introduced in \cite{Axiom_fo_d_KV}
for which the completeness theorem\vspace{-2pt}
\begin{equation}\label{cmp}
\Gamma\vdash\theta\iff \Gamma\models \theta\vspace{-2pt}
\end{equation}
holds whenever $\Gamma$ is a set of \D-sentences (with no free variables) and $\theta$ is an \FO-sentence (with no free variables). This result has been, subsequently, generalized to, e.g.,  allow also open formulas  \cite{Kontinen15foc}, 
independence logic \cite{Hannula_fo_ind_13} and inclusion logic \cite{YangInc20}, 
and extensions of dependence logic by generalized quantifiers  \cite{DBLP:journals/jcss/EngstromKV17}.
A recent new generalization given in  \cite{Yang_neg} extends the known systems for \D and \Ind to cover the case when $\theta$ in (\ref{cmp}) is not necessarily an \FO-formula but merely a formula defining a first-order team property.
However, since the problem of whether a \D- or \Ind-formula defines a first-order team property is undecidable, the extension 
of \cite{Yang_neg}  is not  effectively represented. 
Motivated by  \cite{Yang_neg}, we give an  effective extension of \eqref{cmp}, in which $\Gamma\cup\{\theta\}$ is a set of $\D\oplus\FOTd$-formulas and $\theta$ belongs to certain compositionally defined fragment of $\D\oplus\FOTd$. We also conclude, as a corollary, that the system for \D introduced in \cite{Axiom_fo_d_KV} is also already complete with respect to first-order consequences over arbitrary (possibly open) formulas. This answers  an open problem in \cite{Axiom_fo_d_KV} in the affirmative, which is contrary to the folklore belief that the system of \cite{Axiom_fo_d_KV} is too weak to be complete over arbitrary formulas and additional rules have to be added (as done in, e.g., \cite{Kontinen15foc,Yang_neg}).



Apart from theoretical significance, 
our results also provide new logical tools for applications of team-based logics in other related areas; such applications have been studied in recent years, e.g., in database theory \cite{HannulaK16}, formal semantics of natural language \cite{Ciardelli2015,CiardelliIemhoffYang19}, Bayesian statistics \cite{JELIA,CORANDER2019}, social choice \cite{PacuitYang2016}, and quantum information theory \cite{PV15}. 
In Section \ref{sec:app}, we show some  applications of our systems for \FOT and $\D\oplus\FOTd$, including derivations of Armstrong's Axioms  \cite{Armstrong_Axioms} from database theory, and certain dependency interactions involved in Arrow's Theorem  \cite{Arrow50} in social choice.

 An earlier version of this paper has already appeared as \cite{KontinenYang19}. 

\section{Preliminaries}\label{sec:prel}


We consider first-order vocabularies $\LL$ with the equality symbol $=$. Fix an infinite set \textsf{Var} of first-order variables, and denote its elements by $u,v,x,y,z,\dots$ (with or without subscripts). 
An $\LL$-term $t$ is defined inductively as usual, 
and well-formed formulas of {\em First-order Logic} (\FO)  are defined by the grammar:
\[
\alpha::= t_1=t_2\mid Rt_1\dots t_k\mid \neg\alpha \mid  (\alpha\wedge\alpha)\mid(\alpha\vee\alpha)\mid \exists x \alpha\mid \forall x\alpha.
\]
Throughout the paper, we reserve the first Greek letters $\alpha,\beta,\gamma,\dots$ for first-order formulas. As usual, define $\alpha\to\beta:=\neg\alpha\vee\beta$, $\top:=\forall x(x=x)$ and $\bot:= \exists x(x\neq x)$. We use the letters  $\mathsf{v},\mathsf{x},\mathsf{y},\mathsf{z},\dots$ in sans-serif face  to stand for sequences of variables (of certain length), and sequences of terms are denoted as $\mathsf{t},\mathsf{t}',\dots$ A block of universal quantifiers $\forall x_1\dots\forall x_n$ is sometimes abbreviated as $\forall \mathsf{x}$; similarly for $\exists \mathsf{x}$. We write  $\textsf{Fv}(\alpha)$ for the set of free variables of $\alpha$. Write $\alpha(\mathsf{x})$ to indicate that the free variables of $\alpha$ are among $\mathsf{x}=\langle x_1,\dots,x_n\rangle$.
A formula with no free variables is called a {\em sentence}. We write $\alpha(\mathsf{t}/\mathsf{x})$ for the formula obtained by substituting uniformly a sequence $\mathsf{t}$ of terms  for $\mathsf{x}$ in $\alpha$, where we always assume that each $t_i$ is free for $x_i$.  

For any $\LL$-model $M$, we use the same notation $M$ also to denote its domain, which is assumed to always have at least two elements.
We write $\LL(R)$ for the vocabulary  expanded  from $\LL$ by adding a fresh relation symbol $R$, and write $(M,R^M)$ for the $\LL(R)$-expansion of $M$ in which the $k$-ary relation symbol $R$ is interpreted as $R^M\subseteq M^k$. 
We sometimes write $\alpha(R)$ to emphasize that the formula $\alpha$ is in the vocabulary  $\LL(R)$ for some $\LL$.

 An assignment of an $\LL$-model $M$ for a set $V\subseteq \textsf{Var}$ of variables is a function $s:V\to M$. For any element $a\in M$, $s(a/x)$ is the assignment defined as $s(a/x)(y)=a$ if $y=x$, and $s(a/x)(y)=s(y)$ otherwise.
%
We also write $s(\mathsf{a}/\mathsf{x})$ for $s(a_1/x_1)\dots(a_n/x_n)$.
The interpretation of an \LL-term $t$ under an assignment $s$ of $M$, denoted by $s(t^M)$,  is defined as usual. For a sequence $\mathsf{t}=\langle t_1,\dots,t_n\rangle$ of terms, we write $s(\mathsf{t}^M)$ for $\langle s(t_1^M),\dots,s(t_n^M)\rangle$; similarly for $s(\mathsf{x})$ when $\mathsf{x}$ is a sequence of variables. 


We assume that the reader is familiar with the usual Tarskian semantics of first-order logic. In this paper, we consider  logics with {\em team semantics}. A {\em team} $X$ of $M$ over a set $V$ 
of variables is a set of assignments $s:V\to M$, where $V$ is called the domain of $X$, denoted  $\textsf{dom}(X)$. In particular, the empty set $\emptyset$ is a team, and the singleton $\{\emptyset\}$ of the empty assignment $\emptyset$ is a team.
%
%
%
%
%
Given a first-order formula $\alpha$,  any $\LL$-model $M$ and team $X$ over $V\supseteq \textsf{Fv}(\alpha)$,  the {\em satisfaction} relation $M\models_X\alpha$ is defined inductively as follows:
\begin{itemize}
\item $M\models_X \lambda$ for $\lambda$ a first-order atom ~~ iff ~~ for all $s\in X$, $M\models_s\lambda$ in the usual  sense.
\item $M\models_X \neg\alpha$  ~~iff~~ for all $s\in X$, $M\not\models_{\{s\}}\alpha$. 
  \item $M\models_X\alpha\wedge\beta$ ~~iff~~ $M\models_X\alpha$ and $M\models_X\beta$.
   \item $M\models_X\alpha\vee\beta$ ~~iff~~ there are $Y,Z\subseteq X$ such that $X=Y\cup Z$, $M\models_Y\alpha$ and $M\models_Z\beta$.
       \item $M\models_X\exists x\alpha$ ~~iff~~ $M\models_{X(F/x)}\alpha$ for some  $F:X\to\wp^+ (M)$, where  $\wp^+(M)=\wp(M)\setminus\{\emptyset\}$ and
    \(X(F/x)=\{s(a/x)\mid s\in X,~a\in F(s)\}.\)
  \item $M\models_X\forall x\alpha$ ~~iff~~ $M\models_{X(M/x)}\alpha$, where 
  \(X(M/x)=\{s(a/x)\mid s\in X,~a\in M\}.\)
\end{itemize}
If $\phi$ is a sentence, we write $M\models\phi$  if $M\models_{\{\emptyset\}}\phi$, where (and hereafter) we always assume that $M$ is in the appropriate vocabulary. For any set $\Gamma$ of formulas, we write $M\models_X\Gamma$ if $M\models_X\phi$ holds for all $\phi\in \Gamma$, where (and hereafter) we always assume that the team $X$ over $M$ has domain $\textsf{dom}(X)\supseteq \textsf{Fv}(\phi)$.  We write $\Gamma\models\phi$ if for all models $M$ and teams $X$, $M\models_X\Gamma$ implies $M\models_X\phi$. If $\emptyset\models\phi$, we write $\models\phi$ and say that $\phi$ is {\em valid}. Write simply $\phi\models\psi$ for $\{\phi\}\models\psi$. If both $\phi\models\psi$ and $\psi\models\phi$, write $\phi\equiv\psi$ and say that $\phi$ and $\psi$ are {\em(semantically or logically) equivalent}.

It is easy to verify that first-order formulas have the following properties: Let $M$ be a model, $X,Y$ teams of $M$, $\{X_i\mid i\in I\}$ a nonempty collection of teams of $M$, and $X\upharpoonright V:=\{s\upharpoonright V\mid s\in X\}$ for any $V\subseteq \textsf{dom}(X)$.
\begin{description}
\item[Locality] If $X\upharpoonright \textsf{Fv}(\phi)=Y\upharpoonright \textsf{Fv}(\phi)$, then 
\(M\models_X\phi \iff M\models_Y\phi.\)
\item[Empty team property] $M\models_\emptyset \phi$.
\item[Downward closure]
\([\,M\models_X\phi\text{ and }Y\subseteq X\,]\Longrightarrow M\models_Y\phi.\)
\item[Union closure]
\(M\models_{X_i}\phi\text{ for all }i\in I\Longrightarrow M\models_{\bigcup_{i\in I}X_i}\phi.\)
\end{description}
The empty team property, downward closure property  and union closure property together are equivalent to 
\begin{description}
\item[Flatness property]
\(M\models_X\phi\iff M\models_{\{s\}}\phi\) for all $s\in X$.
\end{description}

Most prominent team-based logics obtained by extending first-order logic with {\em atoms of dependencies} do not have the flatness property. In particular, {\em dependence atoms} $\dep(t_1\dots t_n,t_1'\dots t_m')$ and {\em exclusion atoms} $t_1\dots t_n\mid t_1'\dots t_n'$ are downward closed but not flat, {\em inclusion atoms} $t_1,\dots,t_n\subseteq t_1',\dots,t_n'$ are closed under unions but not flat, and {\em independence atoms} $t_1,\dots,t_n\perp t_1',\dots,t_m'$ are neither downward nor union closed.  We now recall the semantics of these atoms of dependencies, for which free variables are defined as free variables in the terms occurring in the atoms: 
\begin{itemize}
\item $M\models_X\dep(\mathsf{t},\mathsf{t}')$ ~~iff~~  for all $s,s'\in X$, $s(\mathsf{t}^M)=s'(\mathsf{t}^M)$ implies $s({\mathsf{t}'}^M)=s'({\mathsf{t}'}^M)$.
\item $M\models_X\mathsf{t}\mid\mathsf{t}'$ ~~iff~~  for all $s,s'\in X$, $s(\mathsf{t}^M)\neq s'({\mathsf{t}'}^M)$.
\item $M\models_X \mathsf{t}\subseteq \mathsf{t}'$ ~~iff~~ for all $s\in X$, there is $s'\in X$ such that $s(\mathsf{t}^M)=s'({\mathsf{t}'}^M)$
\item $M\models_X\mathsf{t}\perp\mathsf{t}'$ ~~iff~~  for all $s,s'\in X$, there exists $s''\in X$ such that $s''(\mathsf{t}^M)=s(\mathsf{t}^M)$ and $s''({\mathsf{t}'}^M)=s'({\mathsf{t}'}^M)$.
\end{itemize}
Clearly, $\dep(\mathsf{t},t'_1\dots t'_n)\equiv\dep(\mathsf{t},t'_1)\wedge\dots\wedge\dep(\mathsf{t},t'_n)$. The dependence atom $\dep(\langle\,\rangle,\mathsf{t})$ with the first argument being the empty sequence $\langle\,\rangle$ is abbreviated as $\dep(\mathsf{t})$. Such an atom is called  the {\em constancy atom},  and its semantics reduces to
\begin{itemize}
\item $M\models_X\dep(\mathsf{t})$ ~~iff~~  for all $s,s'\in X$, $s(\mathsf{t}^M)=s'(\mathsf{t}^M)$.
\end{itemize}

In this paper, we study two (non-flat) logics based on team semantics, called \FOT and \FOTd, whose formulas are built from a different (yet similar) set of connectives and quantifiers than those in (the team-based) first-order logic as follows:
\leqnomode
\begin{align*}
\tag{\FOT}\phi::&=\lambda\mid \mathsf{x}\subseteq \mathsf{y}\mid\wcn\phi\mid(\phi\wedge\phi)\mid(\phi\vvee\phi)\mid \existso x \phi\mid \forallo x\phi\\
\tag{\FOTd}\phi::&=\lambda\mid\neg\delta\mid (\phi\wedge\phi)\mid(\phi\vvee\phi)\mid \existso x \phi\mid \forallo x\phi
\end{align*}\reqnomode
where $\lambda$ is an arbitrary first-order atomic formula,  $\mathsf{x}$ and $\mathsf{y}$ are two sequences of variables of the same length, and $\delta$ is a  quantifier-free and disjunction-free  formula (i.e., $\delta::=\lambda\mid \neg\delta\mid\delta\wedge\delta$).  We call the logical constants $\wcn,\vvee,\existso,\forallo$ {\em weak classical negation}, {\em weak disjunction}, {\em weak existential quantifier} and {\em weak universal quantifier}, respectively. 
They were introduced earlier in  \cite{AbVan09,KontVan09,Yang_neg} 
with the team semantics:
\begin{itemize}


      \item $M\models_X\wcn\phi$ ~~iff~~ $X=\emptyset$ or $M\not\models_X\phi$.
   \item $M\models_X\phi\vvee\psi$ ~~iff~~ $M\models_X\phi$ or $M\models_X\psi$.
       \item $M\models_X\existso x\phi$ ~iff~  $M\models_{X(a/x)}\phi$ for some $a\in M$, where
    \(X(a/x)\!=\!\{s(a/x)\mid s\!\in\! X\}.\)
  \item $M\models_X\forallo x\phi$ ~~iff~~ $M\models_{X(a/x)}\phi$ for all $a\in M$.
\end{itemize}
The logical constants $\vvee,\existso,\forallo$ have essentially first-order expressive power (in the sense of Theorem \ref{fot2fo} below), and are thus weaker than the corresponding usual ones $\vee,\exists,\forall$ (which have essentially existential second-order expressive power, in the sense of \cite[Theorem 6.2]{Van07dl}). The weak classical negation is considered weaker in comparison to the (strong) classical negation $\sim$ in the literature, defined as $M\models_X\,\sim\!\phi$ iff $M\not\models_X\phi$.
It is easy to verify that formulas of the logics \FOT and \FOTd have the locality property and the empty team property, and \FOTd-formulas are, in addition, downward closed.

In \FOT we  write, as usual, $\phi\cimp\psi$ for $\wcn\phi\vvee\psi$, and $\phi\cbimp\psi$ for $(\phi\cimp\psi)\wedge(\psi\cimp\phi)$. The weak classical negation $\wcn$ admits the law of excluded middle and the double negation elimination law, namely, $\phi\vvee\wcn\phi$ and $\wcn\wcn\phi\cbimp\phi$ are both valid formulas. Analogously to first-order logic, in \FOT conjunction and weak disjunction are inter-definable, so do  the two weak quantifiers, that is,
$\phi\vvee\psi\equiv \wcn(\wcn\phi\wedge\wcn\psi)$, $\forallo x\phi\equiv\wcn\existso x\wcn\phi$, etc.  

While the weak connectives and quantifiers are clearly different from the corresponding (strong) ones in \FO, the two sets of logical constants are closely related. For instance, $\existso$ and $\vvee$ are definable in terms of $\exists$ and $\vee$ together with constancy atoms (assuming that every model has at least two elements):
\[\existso x\phi\equiv\exists x(\dep(x)\wedge\phi)\text{ and }\phi\vvee\psi\equiv\exists x\exists y\big(\dep(x)\wedge \dep(y)\wedge \big((x=y\wedge \phi)\vee (x\neq y\wedge \psi)\big)\big),\]
where $x,y$ are fresh variables (see Proposition \ref{existso_vvee_definable}).
For any \FO-formula $\alpha$, we write $\alpha^\ast$ for the \FOT-formula obtained from $\alpha$ by replacing the first-order logical constants by their weak versions, namely, replacing $\neg$ by $\wcn$, $\vee$ by $\vvee$, $\exists$ by $\existso$, and $\forall$ by $\forallo$. The \FO-formula $\alpha$ and its \FOT-counterpart $\alpha^\ast$ behave similarly over singleton teams or single assignments, as we show in the next lemma.

\begin{lem}\label{fo=fot}
$M\models_{\{s\}}\alpha \iff M\models_{s}\alpha \iff M\models_{\{s\}}\alpha^\ast$. 
\end{lem}
\begin{proof}
The second ``$\iff$" follows from a straightforward inductive argument that uses directly the similarity of the team semantics of \FOT-connectives and quantifiers and the Tarskian semantics of their counterparts in \FO. We now prove the first ``$\iff$"  by induction on $\alpha$. We only give the detailed proof for the non-trivial cases. 

If $\alpha=\beta\vee\gamma$, then 
\begin{align*}
M\models_{\{s\}}\beta\vee\gamma&\iff \exists Y,Z\subseteq \{s\}\text{ s.t. }X=Y\cup Z,~ M\models_Y\beta\text{ and }M\models_Z\gamma\\
&\iff M\models_{\{s\}}\beta\text{ or }M\models_{\{s\}}\gamma\tag{by empty team property}\\
&\iff M\models_s\beta\vee\gamma.\tag{by induction hypothesis}
\end{align*}

If $\alpha=\forall x\beta$, then\allowdisplaybreaks
\begin{align*}
M\models_{\{s\}}\forall x\beta&\iff M\models_{\{s\}(M/x)}\beta\\
&\iff M\models_{\{s(a/x)\}}\beta\text{ for all }a\in M\tag{by flatness}\\
&\iff M\models_{s(a/x)}\beta\text{ for all }a\in M\tag{by induction hypothesis}\\
&\iff M\models_{s}\forall x\beta.
\end{align*}

If $\alpha=\exists x\beta$, then
\begin{align*}
M\models_{\{s\}}\exists x\beta&\iff M\models_{\{s\}(F(s)/x)}\beta\text{ for some }F:\{s\}\to \wp^+(M)\\
&\iff M\models_{\{s(a/x)\}}\beta\text{ for some }a\in  M\tag{by downward closure}\\
&\iff M\models_{s(a/x)}\beta\text{ for some }a\in  M\tag{by induction hypothesis}\\
&\iff M\models_{s}\exists x\beta.
\end{align*}
\end{proof}

We end this section by showing that first-order formulas and dependency  atoms (with variables as arguments) are all definable in \FOT. 

\begin{lem}\label{indepa-def-fot}
Let $\mathsf{u},\mathsf{v},\mathsf{w},\mathsf{x},\mathsf{y}$ be sequences of variables of certain lengths, and variables in $\mathsf{u}, \mathsf{v}, \mathsf{w}$  do not occur in $\mathsf{x},\mathsf{y}$.
\begin{enumerate}[label=(\roman*)]
\item\label{foteam2fot} $\alpha(\mathsf{x})\equiv\forallo \mathsf{v}(\mathsf{v}\subseteq\mathsf{x}\cimp\alpha^\ast(\mathsf{v}/\mathsf{x}))$; 
\item\label{depatm2fot} $\dep(\mathsf{x},\mathsf{y})\equiv \forallo \mathsf{u}\mathsf{v}\mathsf{w}((\mathsf{u}\mathsf{v}\subseteq \mathsf{x}\mathsf{y}\wedge \mathsf{u}\mathsf{w}\subseteq \mathsf{x}\mathsf{y})\cimp \mathsf{v}=\mathsf{w})$;
\item $\dep(\mathsf{t})\equiv\existso \mathsf{v}(\mathsf{v}=\mathsf{t})$, where variables in $\mathsf{v}$ do not occur in the terms in $\mathsf{t}$;
\item\label{exc2fot} $\mathsf{x}\mid\mathsf{y}\equiv \forallo \mathsf{u}\mathsf{v}\big((\mathsf{u}\subseteq \mathsf{x}\wedge \mathsf{v}\subseteq \mathsf{y})\cimp \mathsf{u}\neq\mathsf{v}\big)$;
\item\label{indepatm2fot} $\mathsf{x}\perp\mathsf{y}\equiv \forallo \mathsf{u}\mathsf{v}((\mathsf{u}\subseteq \mathsf{x}\wedge \mathsf{v}\subseteq \mathsf{y})\cimp \mathsf{u}\mathsf{v}\subseteq \mathsf{x}\mathsf{y})$.
\end{enumerate}
\end{lem}
\begin{proof}
We only give the detailed proof for item (i). 
We have that
\begin{align*}
&M\models_X\forallo \mathsf{v}(\mathsf{v}\subseteq \mathsf{x}\cimp\alpha^\ast(\mathsf{v}/\mathsf{x}))\\
\iff& \text{for all sequences }\mathsf{a}\text{ in }M,~M\models_{X(\mathsf{a}/\mathsf{v})}\mathsf{v}\subseteq \mathsf{x}\text{ implies }M\models_{X(\mathsf{a}/\mathsf{v})}\alpha^\ast(\mathsf{v}/\mathsf{x})\\
\iff& \text{for all sequences }\mathsf{a}\text{ in }M,~ \mathsf{a}\in\{s(\mathsf{x})\mid s\in X\}\text{ implies that  for all } s \in X, ~\\
&\qquad\qquad M\models_{\{s(\mathsf{a}/\mathsf{v})\}}\alpha^\ast(\mathsf{v}/\mathsf{x})\tag{by definition of inclusion atom and flatness}\\
\iff& \text{for all sequences }\mathsf{a}\in\{s(\mathsf{x})\mid s\in X\},~ M\models_{\{\emptyset(\mathsf{a}/\mathsf{x})\}}\alpha^\ast(\mathsf{x})\\
\iff&\text{for all } s \in X, ~M\models_{\{s\}}\alpha^\ast(\mathsf{x})\tag{by locality}\\
\iff&\text{for all } s \in X, ~M\models_{\{s\}}\alpha(\mathsf{x})\tag{by Lemma \ref{fo=fot}}\\
\iff&M\models\alpha(\mathsf{x})\tag{by flatness}.
\end{align*}
%
\end{proof}

\section{Characterizing elementary team properties}\label{sec:fot-fotd-expr}

In this section, we prove that (modulo the empty team) \FOT-formulas characterize elementary (or first-order) team properties, and \FOTd-formulas  characterize downward closed  elementary  team properties. 
As a corollary of these expressiveness results, we obtain the compactness theorem of the logics \FOT and \FOTd with respect to sentences (with no free variables). We provide in this section also a proof of compactness theorem for the two logics with respect to arbitrary (possibly open) formulas, which, in the team semantics setting, cannot be derived as a trivial corollary of the compactness with respect to sentences. The compactness theorem will be used in the next sections. 

Let us first define formally the relevant notions. A team $X$ of an $\LL$-model $M$ over a domain $\{v_1,\dots,v_k\}$ can also be viewed as a $k$-ary relation $rel(X)\subseteq M^k$  defined as 
\[rel(X)=\{(s(v_1),\dots,s(v_k))\mid s\in X\}.\]
 We call a collection $\mathcal{P}_M\subseteq \wp(M^k)$ of $k$-ary relations (or teams) of an $\LL$-model $M$  a {\em local team property}; and a {\em (global) team property} is a class $\mathcal{P}$ of local team properties $\mathcal{P}_M$ for all $\LL$-models $M$. A formula $\phi(\mathsf{v})$ of a logic based on team semantics clearly defines a team property $\mathcal{P}^{\phi(\mathsf{v})}$ (or  simply  $\mathcal{P}^{\phi}$) such that for all $M$,
 \[\mathcal{P}^\phi_M=\{rel(X)\mid M\models_X\phi(\mathsf{v})\}.\]
The team properties $\mathcal{P}^\phi$ defined by \FOTd-formulas $\phi(\mathsf{v})$ are clearly {\em downward closed}, that is,  $A\subseteq B\in \mathcal{P}_M^\phi$ implies $A\in \mathcal{P}_M^\phi$ for any $M$.
A team property $\mathcal{P}$ is said to be {\em closed under unions} if $A,B\in \mathcal{P}_M$ implies $A\cup B\in \mathcal{P}_M$ for any $M$.
 The 
team properties $\mathcal{P}^{\alpha}$ defined by first-order formulas $\alpha$ are clearly both downward closed and union closed, and contain the empty relation $\emptyset$.
 
 We call a team property $\mathcal{P}$  {\em elementary} or {\em first-order} if there is a first-order $\LL(R)$-sentence $\alpha(R)$ that defines $\mathcal{P}$, in the sense that 
\begin{equation}\label{fo-def-definition}
(M,A)\models\alpha(R) \iff A\in \mathcal{P}_M
\end{equation}
for all $M$ and all nonempty relations $A$. 
It is worth noting that we are using the terminology ``definability" in two different semantic settings: Even though every elementary 
team property $\mathcal{P}$ is (trivially) definable by some first-order $\LL(R)$-sentence $\alpha(R)$ with an extra relation symbol $R$ (in the sense of (\ref{fo-def-definition}) with respect to the usual semantics of first-order logic), it does not necessarily follow that each such first-order team property $\mathcal{P}$ is definable by some   first-order $\LL$-formula $\beta(\mathsf{v})$ without additional symbols in the team semantics sense (i.e., $\mathcal{P}_M=\mathcal{P}^\beta_M$ for all $M$). As an illustration, consider the following two simple team properties of the empty vocabulary $\mathcal{L}_0$:
\[
\mathcal{P}_{\star1}= \{ (M,rel(X))\ | \ M\textrm{ an $\mathcal{L}_0$-model and } |X|\star 1 \}\text{ for }\star\in\{\leq,\geq\}.
 \]
Neither of these two properties can be defined by any first-order formula $\beta(\mathsf{v})$ of the empty vocabulary $\mathcal{L}_0$, because $\mathcal{P}^\beta$ is closed both under unions and downward, whereas $\mathcal{P}_{\leq 1}$ is not closed under unions and $\mathcal{P}_{\geq 1}$ is not closed downward.

In the rest of the section, we show that  every formula $\phi$ of \FOT defines an elementary team property $\mathcal{P}^{\phi}$; and conversely,  every elementary team property $\mathcal{P}$ with the empty relation $\emptyset$ contained in each local property $\mathcal{P}_M$ can be defined by a formula $\phi$ of \FOT (i.e., $\mathcal{P}_M=\mathcal{P}^\phi_M$ for all $M$). In this sense, we say that \FOT-formulas {\em characterize elementary team properties (modulo the empty team)}. By generalizing this proof, we  also show that \FOTd-formulas characterize downward closed elementary team properties (modulo the empty team), in a similar sense. 

Note that in these characterization results, we  confine ourselves only to  those team properties $\mathcal{P}$ with the empty relation $\emptyset$ contained in each local property $\mathcal{P}_M$, because all  team-based logics considered in this paper have the empty team property (thus  $\emptyset$ is in $\mathcal{P}^\phi_M$ for all $\phi$). In fact, it is easy to verify by induction (or it is observed essentially already in \cite{KontinenNurmi2011}) that team-based first-order logics with the known logical constants (such as $\vee,\vvee,\exists,\exists^1$, etc.) enjoy a certain dichotomy with regard to the empty team, that is, for any formula $\phi$, either $\emptyset\in\mathcal{P}^\phi_M$ holds for all $M$, or $\emptyset\notin\mathcal{P}^\phi_M$ holds for all $M$. For this reason, it does not seem to be possible to find a team-based logic that would characterize {\em all} (downward closed) elementary team properties, at least not with the known logical constants in team semantics. The logics \FOT and \FOTd we introduce in this paper provide reasonably close approximations to such full characterizations.


Now, we show that the team properties defined by formulas  of \FOT and \FOTd are elementary. 

\begin{theorem}\label{fot2fo}
For any $\LL$-formula $\phi(v_1,\dots,v_k)$  of  \FOT or \FOTd, there exists a first-order $\LL(R)$-sentence $\gamma_\phi(R)$ with a fresh $k$-ary relation symbol $R$ such that for any $\LL$-model $M$ and any team $X$ over $\{v_1,\dots,v_k\}$,
\begin{equation}\label{fot2fo_eq}
M \models_X \phi\iff (M,rel(X)) \models \gamma_\phi(R).
\end{equation}
In addition, if $\phi$ is a formula of \FOTd, $R$ occurs only negatively in $\gamma_\phi(R)$  (i.e., every occurrence of $R$ is in the scope of an odd number of nested negation symbols).
\end{theorem}
\begin{proof}
We prove the theorem by proving a slightly more general claim: For any subformula $\theta(\mathsf{v},\mathsf{x})$ of  $\phi(\mathsf{v})$, there exists a first-order $\LL(R)$-formula $\gamma_\theta(R,\mathsf{x})$ such that for all $\LL$-models $M$, teams $X$ and  sequences $\mathsf{a}$ of elements in $M$,
 \begin{equation}\label{trans_thm_eq0}
M \models_{X(\mathsf{a}/\mathsf{x})} \theta(\mathsf{v},\mathsf{x}) \iff (M, rel(X)) \models \gamma_\theta(R,\mathsf{x})(\mathsf{a}/\mathsf{x}).
\end{equation}

Define the formula $\gamma_\theta$ (found essentially in, e.g., \cite{Pietro_I/E,Yang_neg}) inductively as follows:

\begin{itemize}
\item If $\theta(\mathsf{v},\mathsf{x})=\lambda(\mathsf{v},\mathsf{x})$ is a first-order atom, let $\gamma_\theta(R,\mathsf{x})=\forall \mathsf{v}(R\mathsf{v}\to \lambda(\mathsf{v},\mathsf{x}))$.
\item If $\theta(\mathsf{v},\mathsf{x})=\neg\delta(\mathsf{v},\mathsf{x})$ for some quantifier-free and disjunction-free first-order formula $\delta$, let $\gamma_\theta(R,\mathsf{x})=\forall \mathsf{v}(R\mathsf{v}\to \neg\delta(\mathsf{v},\mathsf{x}))$.
\item If $\theta(\mathsf{v},\mathsf{x})=\rho(\mathsf{vx})\subseteq\sigma(\mathsf{vx})$, where $\rho(\mathsf{vx})$ and $\sigma(\mathsf{vx})$ are two sequences of variables from $\mathsf{vx}$, let $\gamma_\theta(R,\mathsf{x})=\forall \mathsf{u}\exists\mathsf{w}\big(R\mathsf{u}\to( R\mathsf{w}\wedge \rho(\mathsf{ux})=\sigma(\mathsf{wx}))\big)$, where $\mathsf{u}$ and $\mathsf{w}$ are two sequences of fresh variables.
\item If $\theta(\mathsf{v},\mathsf{x})=\theta_0(\mathsf{v},\mathsf{x})\wedge\theta_1(\mathsf{v},\mathsf{x})$, 
let $\gamma_{\theta}(R,\mathsf{x})=\gamma_{\theta_0}(R,\mathsf{x})\wedge\gamma_{\theta_1}(R,\mathsf{x})$.
\item If $\theta(\mathsf{v},\mathsf{x})=\theta_0(\mathsf{v},\mathsf{x})\vvee\theta_1(\mathsf{v},\mathsf{x})$, let $\gamma_\theta(R,\mathsf{x})=\gamma_{\theta_0}(R,\mathsf{x})\vee\gamma_{\theta_1}(R,\mathsf{x})$.
\item If $\theta(\mathsf{v},\mathsf{x})=\wcn\theta_0(\mathsf{v},\mathsf{x})$, let $\gamma_\theta(R,\mathsf{x})= \forall \mathsf{v}\neg R\mathsf{v}\vee\neg\gamma_{\theta_0}(R,\mathsf{x})$.
\item If $\theta(\mathsf{v},\mathsf{x})=\existso y\theta_0(\mathsf{v},\mathsf{x}y)$, let $\gamma_\theta(R,\mathsf{x})=\exists y\gamma_{\theta_0}(R,\mathsf{x}y)$.
\item If $\theta(\mathsf{v},\mathsf{x})=\forallo y\theta_0(\mathsf{v},\mathsf{x}y)$, let $\gamma_{\theta}(R,\mathsf{x})=\forall y\gamma_{\theta_0}(R,\mathsf{x}y)$.
\end{itemize}

We only give  the detailed proof of (\ref{trans_thm_eq0})  for some nontrivial cases. If $\theta(\mathsf{v},\mathsf{x})=\lambda(\mathsf{v},\mathsf{x})$ is a first-order atom, then we have that
\begin{align*}
M\models_{X(\mathsf{a}/\mathsf{x})}\lambda(\mathsf{v},\mathsf{x})&\iff M\models_{s}\lambda(\mathsf{v},\mathsf{x})\text{ for all }s\in X(\mathsf{a}/\mathsf{x})\\
&\iff (M,rel(X))\models_{s(\mathsf{a}/\mathsf{x})} R\mathsf{v}\to \lambda(\mathsf{v},\mathsf{x})\text{ for all }s:\{v_1,\dots,v_k\}\to M\\
&\iff (M,rel(X))\models \forall \mathsf{v}(R\mathsf{v}\to \lambda(\mathsf{v},\mathsf{x}))(\mathsf{a}/\mathsf{x}).
\end{align*}

If $\theta=\wcn\theta_0$, then we have that
\begin{align*}
M\models_{X(\mathsf{a}/\mathsf{x})}\wcn\theta_0&\iff X(\mathsf{a}/\mathsf{x})=\emptyset \text{ or }M\not\models_{X(\mathsf{a}/\mathsf{x})}\theta_0\\
&\iff (M,rel(X))\models\forall \mathsf{v}\neg R\mathsf{v}\text{ or }(M, rel(X)) \not\models \gamma_{\theta_0}(R,\mathsf{x})(\mathsf{a}/\mathsf{x}).\tag{by induction hypothesis}\\
&\iff (M, rel(X)) \models \forall \mathsf{v}\neg R\mathsf{v}\ \vee \neg \gamma_{\theta_0}(R,\mathsf{x})(\mathsf{a}/\mathsf{x}).
\end{align*}

If $\phi=\existso y\theta_0(\mathsf{v},\mathsf{x}y)$, then 
we have that
\begin{align*}
M\models_{X(\mathsf{a}/\mathsf{x})}\existso y\theta_0(\mathsf{v},\mathsf{x}y)&\iff M\models_{X(\mathsf{a}b/\mathsf{x}y)}\theta_0(\mathsf{v},\mathsf{x}y)\text{ for some }b\in M\\
&\iff (M, rel(X)) \models \gamma_{\theta_0}(R,\mathsf{x}y)(\mathsf{a}b/\mathsf{x}y)\text{ for some }b\in M\tag{by induction hypothesis}\\
&\iff (M,rel(X))\models\exists y\gamma_{\theta_0}(R,\mathsf{x})(\mathsf{a}/\mathsf{x}).
\end{align*}
\end{proof}

Next we prove the converse direction of \Cref{fot2fo}, from which we can conclude that \FOT formulas characterize exactly elementary team properties (modulo the empty team).

\begin{theorem}\label{fo2fot}
For any first-order $\LL(R)$-sentence $\gamma(R)$ with a $k$-ary relation symbol $R$, there exists an $\LL$-formula $\phi_\gamma(v_1,\dots,v_k)$ of \FOT such that for any $\LL$-model $M$ and any nonempty team $X$ over $\{v_1,\dots,v_k\}$,
\begin{equation*}\label{fo2fot_eq}
M \models_X \phi_\gamma(\mathsf{v})\iff (M,rel(X)) \models \gamma(R).
\end{equation*}
Moreover, if $R$ occurs only negatively in $\gamma(R)$,  $\phi_\gamma$ can also be chosen  in \FOTd.
\end{theorem}
\begin{proof}
We may assume w.l.o.g. that the first-order sentence $\gamma(R)$ is in prenex normal form
$Q_1x_1\dots Q_nx_n\alpha(\mathsf{x})$, where $Q_i\in\{\forall,\exists\}$, $\alpha$ is quantifier-free and in negation normal form (i.e., negation occurs only in front of atomic formulas), 
and every occurrence of $R$ is of the form $R\mathsf{x}_i$ for some sequence $\mathsf{x}_i$ of bound variables  (for $R\mathsf{t}\equiv \exists \mathsf{y}(\mathsf{y}=\mathsf{t}\wedge R\mathsf{y})$).
Define the translation $\phi_\gamma(\mathsf{v}):=Q_1^1x_1\dots Q_n^1x_n\phi_\alpha(\mathsf{x},\mathsf{v})$, where $Q_i^1=\forallo$ if $Q_i=\forall$, $Q_i^1=\existso$ if $Q_i=\exists$, and $\phi_\alpha(\mathsf{x},\mathsf{v})$ is obtained from $\alpha(\mathsf{x})$ by first taking the corresponding formula $\alpha^\ast(\mathsf{x})$ in \FOT and then replacing every first-order atom of the form $R\mathsf{x}_i$ by  $\mathsf{x}_i\subseteq \mathsf{v}$. 
%
%
We show by induction that for any quantifier-free formula $\alpha(\mathsf{x})$,  any nonempty team $X$ over $\{v_1,\dots,v_k\}$ and $a_1,\dots,a_n\in M$,
\begin{equation}\label{trans_thm_altpf_eq1}
M\models_{X(\mathsf{a}/\mathsf{x})}\phi_{\alpha}(\mathsf{x},\mathsf{v})\iff (M,rel(X))\models\alpha(\mathsf{a}/\mathsf{x}).
\end{equation}
%

We only give the detailed proof for the nontrivial cases. If $\alpha=\lambda(\mathsf{x})$ is an atomic formula in which $R$ does not occur, then $\phi_\lambda=\lambda(\mathsf{x})$ and \vspace{-4pt}
\begin{align*}
M\models_{X(\mathsf{a}/\mathsf{x})}\lambda(\mathsf{x})&\iff M\models_{X}\lambda(\mathsf{a}/\mathsf{x})\\
&\iff M\models_{s}\lambda(\mathsf{a}/\mathsf{x})\text{ for all }s\in X\\
&\iff M\models\lambda(\mathsf{a}/\mathsf{x})\tag{since $X\neq\emptyset$ and $R$ does not occur in $\lambda$}\\
&\iff (M,rel(X))\models\lambda(\mathsf{a}/\mathsf{x}).
\end{align*}

If $\alpha=R\mathsf{x}_{i}$, then $\phi_\alpha=\mathsf{x}_i\subseteq \mathsf{v}$ and \vspace{-4pt}
\begin{align*}
M\models_{X(\mathsf{a}/\mathsf{x})}\mathsf{x}_{i}\subseteq \mathsf{v}
\iff& \text{for all }s\in X(\mathsf{a}/\mathsf{x}),\text{ there exists }s'\in X(\mathsf{a}/\mathsf{x})\text{ s.t. }s(\mathsf{x}_{i})=s'(\mathsf{v})\\
\iff& \mathsf{a}_{i}\in rel(X)=\{s'(\mathsf{v})\mid s'\in X\}\tag{since $X\neq\emptyset$}\\
\iff& (M,rel(X))\models R\mathsf{x}_{i}(\mathsf{a}/\mathsf{x}).
\end{align*}

If $\alpha=\neg\lambda(\mathsf{x})$, then $\phi_\alpha=\wcn\phi_{\lambda}(\mathsf{x},\mathsf{v})$ and \vspace{-4pt}
\begin{align*}
M\models_{X(\mathsf{a}/\mathsf{x})}\wcn\phi_{\lambda}(\mathsf{x},\mathsf{v})&\iff M\not\models_{X(\mathsf{a}/\mathsf{x})}\phi_{\lambda}(\mathsf{x},\mathsf{v})\tag{since $X(\mathsf{a}/\mathsf{x})\neq\emptyset$}\\
&\iff (M,rel(X))\not\models{\lambda}(\mathsf{a}/\mathsf{x})\tag{by induction hypothesis}\\
&\iff (M,rel(X))\models\neg{\lambda}(\mathsf{a}/\mathsf{x}).
\end{align*}

%

Finally, we have that
\begin{align*}
&M\models_X Q_1^1x_1\dots Q_n^1x_n\phi_\alpha(\mathsf{x},\mathsf{v})\\
\iff& Q_1a_1\in M\dots Q_n a_n\in M:~M\models_{X(\mathsf{a}/\mathsf{x})}\phi_\alpha(\mathsf{x},\mathsf{v})\\
\iff& Q_1a_1\in M\dots Q_n a_n\in M:~(M,rel(X))\models\alpha(\mathsf{a}/\mathsf{x})\tag{by \ref{trans_thm_altpf_eq1}}\\
\iff& (M,rel(X))\models Q_1x_1\dots Q_nx_n\alpha(\mathsf{x}).
\end{align*}

This completes the proof for the translation into \FOT. Now, if $R$ occurs only negatively in $\gamma$ (thus also in $\alpha$), we can define alternatively the translation into \FOTd by redefining the following two cases:

%
%
If $\alpha=\neg R\mathsf{x}_{i}$, define alternatively $\phi_{\alpha}:=\mathsf{x}_{i}\neq\mathsf{v}$ (which is defined as $\neg\bigwedge_{j\in J}\,x_{ij}=v_j)$, and we have that
\begin{align*}
M\models_{X(\mathsf{a}/\mathsf{x})}\mathsf{x}_{i}\neq\mathsf{v}&\iff \text{for all }s\in X(\mathsf{a}/\mathsf{x}),~s(\mathsf{x}_{i})\neq s(\mathsf{v})\\
&\iff \text{for all }s\in X,~\mathsf{a}_{i}\neq s(\mathsf{v})\\
&\iff \mathsf{a}_{i}\notin rel(X)=\{s(\mathsf{v})\mid s\in X\}\\
&\iff (M,rel(X))\models\neg R\mathsf{x}_{i}(\mathsf{a}/\mathsf{x}).
\end{align*}

If $\alpha:=\delta(\mathsf{x})$ is a literal (atomic or negated atomic formula) that does not contain $R$, define alternatively $\phi_\alpha=\delta(\mathsf{x})$. It is easy to verify that (\ref{trans_thm_altpf_eq1}) still holds.
\end{proof}

\begin{cor}
\FOT-formulas  characterize  elementary  team properties (modulo the empty team).
\end{cor}

\begin{rmk}
Similarly,  the logic defined by the following syntax with exclusion atoms also characterizes elementary team properties (modulo the empty team):
\[\phi::=\lambda\mid\, \mathsf{x}\,|\, \mathsf{y}\,\mid\wcn\phi\mid\phi\wedge\phi\mid\phi\vvee\phi\mid \existso x \phi\mid \forallo x\phi\]
where $\lambda$ is an arbitrary first-order atomic formula. 
To see why, note that for any 
%
exclusion atom $\rho(\mathsf{vx})\mid\sigma(\mathsf{vx})$, its translation for the equivalence (\ref{trans_thm_eq0}) can be defined as 
\[\gamma_{\rho(\mathsf{vx})\mid\sigma(\mathsf{vx})}(R)=\forall \mathsf{vu}((R\mathsf{v}\wedge R\mathsf{u})\to \rho(\mathsf{vx})\neq \sigma(\mathsf{ux})).\]
For the converse direction, in equivalence (\ref{trans_thm_altpf_eq1}), it is easy to see that the translation for the case $\alpha=\neg R\mathsf{x}_i$ can also be defined  as $\phi_{\alpha}=\mathsf{x}_{i}\mid\mathsf{v}$,
and the translation for the case $\alpha=R\mathsf{x}_i$ can thus  be defined alternatively as $\phi_{R\mathsf{x}_i}=\phi_{\neg(\neg R\mathsf{x}_i)}=\wcn \mathsf{x}_{i}\mid\mathsf{v}$. \end{rmk}

To conclude from the above theorems that \FOTd formulas characterize downward closed elementary  team properties (modulo the empty team),  we now prove a characterization theorem for first-order sentences $\alpha(R)$ that define downward closed team properties, by applying Lyndon's Interpolation Theorem of first-order logic, which we recall below.

%
%


\begin{theorem}[Lyndon's Interpolation \cite{Lyndon59}]\label{lyndon_inter}
Let $\alpha$ be a first-order $\LL_0$-formula and $\beta$ a first-order $\LL_1$-formula. If $\alpha\models\beta$, then there is a first-order $\LL_0\cap\LL_1$-formula $\delta$ such that $\alpha\models\delta$ and $\delta\models\beta$, and moreover a predicate symbol has a positive (resp. negative) occurrence in $\delta$ only if it has a positive (resp. negative) occurrence in both $\alpha$ and $\beta$.
\end{theorem}

\begin{prop}\label{dw_neg}
A first-order $\LL(R)$-sentence $\alpha(R)$  defines a downward closed team property with respect to $R$ if and only if there is a first-order $\LL(R)$-sentence $\beta(R)$ such that $\alpha\equiv\beta$ and $R$ occurs only negatively in $\beta$. 
\end{prop}
\begin{proof}
``$\Longleftarrow$": Suppose $\alpha$ is a first-order $\LL(R)$-sentence in which the $k$-ary predicate $R$ occurs only negatively, and we  assume w.l.o.g. that $\alpha$ is in negation normal form. We can show by induction that $\alpha$ is downward closed with respect to $R$. 
The only nontrivial case is when $\alpha=\neg R\mathsf{t}$. In this case,   for any model $M$, any $A\subseteq B\subseteq M^k$, and any assignment $s$, 
\[
(M,B)\models_s \neg R\mathsf{t}\Longrightarrow s(\mathsf{t}^M)\notin B\Longrightarrow  s(\mathsf{t}^M)\notin A\Longrightarrow (M,A)\models_s \neg R\mathsf{t}.
\]

``$\Longrightarrow$": 
Suppose that $\alpha$ is a first-order $\LL(R)$-sentence that is downward closed with respect to $R$. It is easy to see that
\(\alpha\equiv \exists S(\alpha(S/R)\wedge \forall \mathsf{x}(R\mathsf{x}\to S\mathsf{x})),\)
where $\alpha(S/R)$ is obtained from $\alpha$ by replacing every occurrence of $R$ by $S$. Put $\gamma=\alpha(S/R)\wedge \forall \mathsf{x}(R\mathsf{x}\to S\mathsf{x})$, and note that $R$ occurs only  negatively in $\gamma$. Then $\gamma\models\alpha$. 
%
%
Now, by Lyndon's Interpolation Theorem, there is a first-order $\LL(R)$-sentence $\beta(R)$ such that $\gamma(R,S)\models \beta(R)$ and $\beta(R)\models\alpha(R)$, and moreover, $R$ occurs only negatively in $\beta$. 
It remains to  show $\alpha\models\beta$. For any $\LL(R)$-model $(M,A)$ such that $(M,A)\models\alpha(R)$, the $\LL(R,S)$-model $(M,A,A)$ clearly satisfies $(M,A,A)\models \alpha(S/R)\wedge \forall \mathsf{x}(R\mathsf{x}\to S\mathsf{x})$, i.e., $(M,A,A)\models\gamma(R,S)$. Since $\gamma\models \beta$, we conclude $(M,A,A)\models\beta(R)$, thereby $(M,A)\models\beta(R)$.
\end{proof}

\begin{cor}\label{fotd_fodw}
For any $\LL$-formula $\phi(\mathsf{v})$ of \FOTd, there exists a first-order $\LL(R)$-sentence $\gamma_\phi(R)$ with $R$ occurring only negatively such that (\ref{fot2fo_eq}) holds, and vice versa.
In particular, \FOTd-formulas  characterize  downward closed elementary team properties (modulo the empty team).
\end{cor}


Another corollary of the expressive power result for \FOT and \FOTd is that the two logics are compact with respect to sentences (with no free variables), i.e., for any set $\Gamma$ of sentences in the logics, if every finite subset of $\Gamma$ has a model, then $\Gamma$ itself has a model. In the sequel we will use the compactness of \FOT and dependence logic (\D) with respect to arbitrary (possibly open) formulas in the completeness proofs.  For the usual (single-assignment based) first-order logic or existential second-order logic (\ESO), compactness over arbitrary (possibly open) formulas is an immediate corollary of  compactness over sentences, as one can encode every free variable with a fresh constant symbol in the expanded vocabulary. For team-based logics, however, free variables are interpreted over teams, which correspond to relations instead of single values. Relation symbols always have fixed finite arities, whereas a set of formulas may well contain infinitely many free variables. For this reason, compactness for formulas of team-based logics does not follow immediately from a similar argument to the usual case.


We now prove the compactness theorem for arbitrary formulas of team-based logics whose translations in the sense of (\ref{fot2fo_eq}) are sentences of \ESO (including \FOT and \FOTd), given the assumption that the formulas in question contain only countably many free variables. Our argument builds on the one in \cite{Van07dl} where \D is shown to be compact with respect to sentences, by using essentially the compactness of first-order logic.
 It is not clear how to  adapt our argument to cover also the uncountable case.

 We say that a set $\Gamma$ of formulas of a team-based logic is {\em satisfiable} if $M\models_X\Gamma$ for some model $M$ and some nonempty team $X$ of $M$ with $\mathsf{dom}(X)\supseteq \mathsf{Fv}(\Gamma)=\bigcup_{\phi\in\Gamma}\mathsf{Fv}(\phi)$. That is, $\Gamma$ is satisfiable iff $\Gamma\not\models\bot$.  

\begin{theorem}[Compactness]\label{compact}
Let $\Gamma$ be a set of formulas in a team-based logic whose translations are \ESO-sentences, and suppose that the set of free variables occurring in $\Gamma$ is countable.
If every finite subset $\Gamma_0$ of $\Gamma$   is satisfiable, then $\Gamma$ is satisfiable. 
In particular, \FOT and \FOTd are compact (with respect to formulas).
\end{theorem}
\begin{proof}
Let $\LL$ be the vocabulary of  $\Gamma$ and $\mathsf{Fv}(\Gamma)=\{v_n\mid n\in\mathbb{N}\}$. By assumption, 
for every  $\phi\in \Gamma$, there is an $\LL(R_\phi)$-sentence $\gamma_{\phi}(R_\phi)$ in \ESO with a fresh $|\mathsf{Fv}(\phi)|$-ary relation symbol $R_\phi$ such that for every $\LL$-model $M$ and team $X$ over the domain $\mathsf{Fv}(\phi)$, 
\begin{equation}\label{cmpct_eq1}
M \models_X \phi\iff (M,rel(X)) \models \gamma_{\phi}(R_\phi).
\end{equation}
We may assume that the \ESO-sentence $\gamma_{\phi}(R_\phi)$ is of the form $\exists R_{\phi}^1\dots \exists R_{\phi}^{m_\phi}\alpha_{\phi}$, where $\alpha_{\phi}$ is a first-order sentence with the relation symbol $R_\phi$, 
and each relation symbol $R_\phi^i$ is fresh. For every $n\in \mathbb{N}$, let
\[\Gamma_n=\{\phi\in \Gamma\mid \textsf{Fv}(\phi)\subseteq \{v_1,\dots,v_n\}\}\]
be the set of formulas in $\Gamma$ whose free variables are among $v_1,\dots,v_n$, and introduce a $|\textsf{Fv}(\Gamma_n)|$-ary fresh relation symbol $S_n$.  Clearly, $\Gamma_n\subseteq\Gamma_{n+1}$ and $\Gamma=\bigcup_{n\in\mathbb{N}}\Gamma_n$. 
Consider the set $\Gamma'=\{\alpha_\phi\mid \phi\in \Gamma\}\cup\Delta$ of  first-order sentences, where $\alpha_\phi(R_\phi,R_{\phi}^1,\dots, R_{\phi}^{m_\phi})$ is the first-order part of the \ESO-translation of the formula $\phi$,
\begin{align*}
\Delta=\big\{\exists \mathsf{v}R_\phi\mathsf{v},~\forall \mathsf{v}(R_\phi\mathsf{v}\to \exists \mathsf{u}S_n\mathsf{vu}),\forall &\mathsf{v}\big(S_n\mathsf{v}\to R_\phi\sigma_\phi(\mathsf{v})\big),\\
&\forall \mathsf{v}(S_n\mathsf{v}\to \exists \mathsf{u}S_{n+1}\mathsf{vu})\mid n\in\mathbb{N},~\phi\in \Gamma_n\big\},
\end{align*}
$\mathsf{\sigma}_\phi(\mathsf{v})$ lists the free variables in $\phi(\mathsf{v})$, and we assume w.l.o.g.  the variables are ordered  as shown. 
Intuitively, the theory $\Delta$ ensures that the relation $S_n$ encodes the team over the domain $\textsf{Fv}(\Gamma_n)$ that glues 
all relations $R_\phi$ with $\phi\in\Gamma_n$.

We now show that every finite subset $\Gamma_0'\subseteq \Gamma'$ has a model. Consider the corresponding finite subset $\Gamma_0=\{\psi\in \Gamma\mid \alpha_\psi\in \Gamma_0'\}$ of $\Gamma$. By assumption, there is an \LL-model $N$ and nonempty team $Y$ such that $N\models_{Y} \Gamma_0$, where, by locality, we assume that $\textsf{dom}(Y)=\textsf{Fv}(\Gamma)$.  
For every $\psi\in \Gamma_0$, we have that $N\models_{Y_\psi}\psi$ by the locality property, where $Y_\psi=Y\upharpoonright \mathsf{Fv}(\psi)$ is the restriction of $Y$ over the free variables in $\psi$. 
Thus, by (\ref{cmpct_eq1}) we know that $(N,rel(Y_\psi))$ is an $\LL(R_\psi)$-model of $\gamma_{\psi}(R_\psi)=\exists R_{\psi}^1\dots \exists R_{\psi}^{m_\psi}\alpha_{\psi}$. It follows that some expansion $(N,rel(Y_\psi),\mathcal{R}_\psi^N)$ of  the model $(N,rel(Y_\psi))$ in the vocabulary $\LL_\psi=\LL(R_\psi,R_{\psi}^1\dots R_{\psi}^{m_\psi})$ is a model of the first-order sentence $\alpha_\psi$, where 
\(\mathcal{R}_\psi^N=\langle (R_\psi^1)^N,\dots,(R_\psi^{m_\psi})^N\rangle.\)
%
Consider the 
model 
\(N'=(N,\langle R_\phi^{N}\rangle_{\phi\in \Gamma},\langle \mathcal{R}_\psi^N\rangle_{\psi\in\Gamma_0},\langle S_n^N\rangle_{n\in\mathbb{N}})\) in the vocabulary $\{R_\phi\mid \phi\in \Gamma\}\cup(\bigcup_{\psi\in \Gamma_0}\LL_\psi)\cup \{S_n\mid n\in \mathbb{N}\}$, where 
$R_\phi^{N}=rel(Y\upharpoonright \textsf{Fv}(\phi))$ for every $\phi\in \Gamma$,
and $S_n^N=rel(Y\upharpoonright \textsf{Fv}(\Gamma_n))$.
It is  not hard to verify that $N'\models\Delta$ (as $S_n$'s and $R_\phi$'s are interpreted according to the nonempty team $Y$).
On the other hand, since $N'$ is an expansion of $(N,rel(Y_\psi),\mathcal{R}_\psi^N)$ for every $\psi\in\Gamma_0$, we have that $N'\models\{\alpha_\psi\mid\psi\in\Gamma_0\}$. Hence $N'\models\Gamma_0'$.

 Now, by compactness of first-order logic, $\Gamma'$ has a model $M'$ in the vocabulary $\LL'=\bigcup_{\psi\in \Gamma}\LL_\psi\cup \{S_n\mid n\in \mathbb{N}\}$. For every $\phi\in \Gamma$, let $(M,R_\phi^M)$ be the reduct of $M'$ to the vocabulary $\LL(R_\phi)$. Since $M'\models\alpha_\phi$, we have that $(M,R_\phi^M)\models\gamma_{\phi}(R_\phi)$, which by (\ref{cmpct_eq1}) implies that $M\models_{X_\phi}\phi$ for the team $X_\phi=\{s:\mathsf{Fv}(\phi)\to M\mid s(\mathsf{v})\in R_\phi^M\}$ of $M$ over the finite domain $\mathsf{Fv}(\phi)$, where $\mathsf{v}$ lists $\mathsf{Fv}(\phi)$. Since $M'\models\exists \mathsf{v}R_\phi\mathsf{v}$, we  know that $X_\phi\neq\emptyset$. Consider the team $X$ over $\mathsf{Fv}(\Gamma)$, defined as
\[X=\{s:\mathsf{Fv}(\Gamma)\to M\mid s\upharpoonright \mathsf{Fv}(\phi)\in X_\phi\text{ for all }\phi\in \Gamma\}.\]
We claim that for every $\phi\in \Gamma$, $X\upharpoonright \mathsf{Fv}(\phi)=X_\phi$, which would then imply that $X\neq\emptyset$ and  that $M\models_X\Gamma$ by locality, i.e., $\Gamma$ is satisfiable.

By definition, it suffices to show that $X_\phi\subseteq X\upharpoonright \mathsf{Fv}(\phi)$. Let $s\in X_\phi$. For every $n\in\mathbb{N}$, put $X_{n}=\{t:\textsf{Fv}(\Gamma_n)\to M\mid t(\mathsf{v})\in S_n^{M'}\}$, where $\mathsf{v}$ lists $\textsf{Fv}(\Gamma_n)$; that is, $X_n$ is the team associated with the relation $S_n^{M'}$.
 Let $k\in\mathbb{N}$ be (the least number) such that  $\phi\in \Gamma_k$. For every $n\geq k$,  we find inductively an extension $s_n\in X_{n}$ of $s$ over the domain $\textsf{Fv}(\Gamma_n)$ such that $s_n$ extends $s_{m}$ for all $k\leq m\leq n$.
This would then imply that $s_n\upharpoonright \textsf{Fv}(\psi)\in X_\psi$ for all  $\psi\in\Gamma_n$, since 
\(M'\models\forall \mathsf{v}\big(S_{n}\mathsf{v}\to R_{\psi}\sigma_{\psi}(\mathsf{v})\big).\)
Therefore the assignment $\hat{s}=\bigcup_{n\geq k}s_n$ is in the team $X$, thereby $s\in X\upharpoonright \mathsf{Fv}(\phi)$.

Now, since $M'\models \forall \mathsf{v}(R_\phi\mathsf{v}\to\exists \mathsf{u}S_k\mathsf{vu})$, there exists an $s_k\in X_k$ over the domain $\textsf{Fv}(\Gamma_k)$ such that $s_k\upharpoonright \textsf{Fv}(\phi)=s\upharpoonright \textsf{Fv}(\phi)$. 
Suppose $s_n\in X_n$ over $\textsf{Fv}(\Gamma_n)$ for $n\geq k$ has already been defined.  Then, since 
\(M'\models\forall \mathsf{v}(S_n\mathsf{v}\to \exists \mathsf{u}S_{n+1}\mathsf{vu}),\)
there exists an $s_{n+1}\in X_{n+1}$ over  $\textsf{Fv}(\Gamma_{n+1})$ such that $s_{n+1}\upharpoonright \textsf{Fv}(\Gamma_n)=s_n\upharpoonright \textsf{Fv}(\Gamma_n)$.
\end{proof}

Let us remark again that the assumption that $\mathsf{Fv}(\Gamma)$ is countable seems to be crucial for the above proof. How to generalize the result to the uncountable case is left as future work.

\section{An axiomatization of \FOT}\label{sec:FOT}


In this section, we introduce a system of natural deduction for \FOT, and prove the soundness and completeness theorem.  

For the convenience of our proofs, we present our system of natural deduction in sequent style. 

\begin{defn}
The system of natural deduction for \FOT consists of the rules  in \Cref{rules_weak_cnt_qtf} and \Cref{tb_rules_fot}, where  letters  in sans-serif face (such as $\mathsf{x},\mathsf{y}$) stand for sequences of variables,  
$c$ is a constant symbol, $\dep(\mathsf{t})$ is short for $\existso \mathsf{x}(\mathsf{x}=t)$,  
and an inclusion atom  such as $c\mathsf{x}\subseteq v\mathsf{y}$ is short for $\existso u(u=c\wedge u\mathsf{x}\subseteq v\mathsf{y})$, where $\mathsf{x}$ and $u$ are fresh (sequences of) variables. 

We write $\Gamma\vdash_{\FOT}\phi$ or simply $\Gamma\vdash\phi$, if the sequent $\Gamma\vdash\phi$ is derivable in the system. Write $\phi\dashv\vdash\psi$ if $\phi\vdash\psi$ and $\psi\vdash\phi$.
\end{defn}

\begin{table}[t]
\centering
\caption{Rules for the identity and some logical constants}\setlength{\tabcolsep}{6pt}
\begin{tabular}{|C{0.445\linewidth}C{0.495\linewidth}|}
\hline
\AxiomC{}\noLine\UnaryInfC{}\RightLabel{$=$\textsf{I}}\UnaryInfC{$\Gamma\vdash t=t$}\noLine\UnaryInfC{}\DisplayProof&\AxiomC{}\noLine\UnaryInfC{$\Gamma\vdash t=t'$} \AxiomC{}\noLine\UnaryInfC{$\Gamma\vdash\phi(t/x)$} \RightLabel{$=$\textsf{Sub}}\BinaryInfC{$\Gamma\vdash\phi(t'/x)$}\noLine\UnaryInfC{}\DisplayProof\\
\AxiomC{}\noLine\UnaryInfC{$\Gamma\vdash\phi$} \AxiomC{}\noLine\UnaryInfC{$\Gamma\vdash\psi$} \RightLabel{$\wedge$\textsf{I}}\BinaryInfC{$\Gamma\vdash\phi\wedge\psi$}\noLine\UnaryInfC{}\DisplayProof& \def\defaultHypSeparation{\hskip .1in}\AxiomC{}\noLine\UnaryInfC{$\Gamma\vdash\phi\wedge\psi$}\RightLabel{$\wedge$\textsf{E}}\UnaryInfC{$\Gamma\vdash\phi$}\noLine\UnaryInfC{}\DisplayProof\quad
\AxiomC{}\noLine\UnaryInfC{$\Gamma\vdash\phi\wedge\psi$}\RightLabel{$\wedge$\textsf{E}}\UnaryInfC{$\Gamma\vdash\psi$}\noLine\UnaryInfC{}\DisplayProof
 \\
\def\ScoreOverhang{0.5pt}
 \def\defaultHypSeparation{\hskip .1in}\AxiomC{}\noLine\UnaryInfC{$\Gamma\vdash\phi$}\RightLabel{\vveei}\UnaryInfC{$\Gamma\vdash\phi\vvee\psi$}\noLine\UnaryInfC{}\DisplayProof\quad
\AxiomC{}\noLine\UnaryInfC{$\Gamma\vdash\phi$}\RightLabel{\vveei}\UnaryInfC{$\Gamma\vdash\psi\vvee\phi$}\noLine\UnaryInfC{}\DisplayProof
 &\def\ScoreOverhang{0.5pt}
 \def\defaultHypSeparation{\hskip .1in}\AxiomC{$\Gamma\vdash\phi\vvee\psi$}\AxiomC{}\noLine\UnaryInfC{$\Gamma,\phi\vdash\chi$} \AxiomC{}\noLine\UnaryInfC{$\Gamma,\psi\vdash\chi$} \RightLabel{\vveee}\TrinaryInfC{$\Gamma\vdash\chi$}\noLine\UnaryInfC{}\DisplayProof\\
\multirow{2}{*}{\AxiomC{}\noLine\UnaryInfC{$\Gamma\vdash\phi(t/x)$}\AxiomC{}\noLine\UnaryInfC{$\Gamma\vdash\dep(t)$}\RightLabel{\existsoi}\BinaryInfC{$\Gamma\vdash\existso x\phi$}\noLine\UnaryInfC{}\DisplayProof}&  
\AxiomC{$\Gamma\vdash\existso x\phi$\quad\quad\!\!$\Gamma,\phi(v/x),\dep(v)\vdash\psi$}
\RightLabel{\existsoe~(a)}\UnaryInfC{$\Gamma\vdash\psi$}\noLine\UnaryInfC{}\DisplayProof\\
&\AxiomC{$\Gamma\vdash\existso x\phi$} 
\AxiomC{}\noLine\UnaryInfC{$\Gamma,\phi(c/x)\vdash\psi$}
\RightLabel{\existsoe~(b)}\BinaryInfC{$\Gamma\vdash\psi$}\DisplayProof
 \\
\AxiomC{$\Gamma,\dep(v)\vdash\phi(v/x)$}\RightLabel{\foralloi~(a)}\UnaryInfC{$\Gamma\vdash\forallo x\phi$}\DisplayProof
&
\AxiomC{}\noLine\UnaryInfC{$\Gamma\vdash\forallo x\phi$}\AxiomC{}\noLine\UnaryInfC{$\Gamma\vdash\dep(t)$} \RightLabel{\foralloe}\BinaryInfC{$\Gamma\vdash\phi(t/x)$}\DisplayProof\\
\AxiomC{}\noLine\UnaryInfC{$\Gamma\vdash\phi(c/x)$}\RightLabel{\foralloi~(b)}\UnaryInfC{$\Gamma\vdash\forallo x\phi$}\noLine\UnaryInfC{}\DisplayProof
&{\footnotesize (a). $v\notin\textsf{Fv}(\Gamma\cup\{\phi,\psi\})$\quad\quad\quad\quad\quad\quad\quad\quad\quad\quad\quad\quad (b). $c$ does not occur in $\Gamma,\phi,\psi$}\\\hline
\end{tabular}
\label{rules_weak_cnt_qtf}
\end{table}%

\begin{table}[H]
\centering
\caption{Other rules for \FOT}\setlength{\tabcolsep}{6pt}
\vspace{8pt}
\begin{tabular}{|C{0.23\linewidth}C{0.45\linewidth}C{0.22\linewidth}|}
\hline
\multirow{2}{*}{\def\ScoreOverhang{0.5pt}
\AxiomC{}\noLine\UnaryInfC{}\noLine\UnaryInfC{$\phi\in\Gamma$}\RightLabel{\textsf{AssmI}}\UnaryInfC{$\Gamma\vdash\phi$}\DisplayProof}&
\def\ScoreOverhang{0.5pt}
\AxiomC{}\noLine\UnaryInfC{$\Gamma,\phi\vdash\bot$} \RightLabel{\wcn\textsf{I}}\UnaryInfC{$\Gamma\vdash\wcn\phi$} \DisplayProof
\quad~\AxiomC{}\noLine\UnaryInfC{$\Gamma\vdash\phi$} \AxiomC{}\noLine\UnaryInfC{$\Gamma\vdash\wcn\phi$} \RightLabel{$\wcn\mathsf{E}$}\BinaryInfC{$\Gamma\vdash\psi$} \DisplayProof
&\def\ScoreOverhang{0.5pt}
\AxiomC{}\noLine\UnaryInfC{}\noLine\UnaryInfC{}\RightLabel{\consi}\UnaryInfC{$\Gamma\vdash\dep(\mathsf{c})$}\DisplayProof
\\
&\AxiomC{}\noLine\UnaryInfC{$\Gamma,\wcn\phi\vdash\bot$} \RightLabel{\raa}\UnaryInfC{$\Gamma\vdash\phi$} \DisplayProof&
\AxiomC{}\noLine\UnaryInfC{$\Gamma\vdash\dep(\mathsf{t})$}\RightLabel{\consi}\UnaryInfC{$\Gamma\vdash\dep(f\mathsf{t})$}\DisplayProof\\
\multicolumn{3}{|C{0.96\linewidth}|}{\quad\AxiomC{}\noLine\UnaryInfC{}\RightLabel{\incid}\UnaryInfC{$\Gamma\vdash\mathsf{x}\subseteq\mathsf{x}$}\DisplayProof\quad\quad\quad\quad\quad\quad\quad\quad
\AxiomC{}\noLine\UnaryInfC{$\Gamma\vdash\mathsf{x}\subseteq\mathsf{y}$}\AxiomC{$\Gamma\vdash\mathsf{y}\subseteq\mathsf{z}$}\RightLabel{\inctr}\BinaryInfC{$\Gamma\vdash\mathsf{x}\subseteq\mathsf{z}$}\DisplayProof}\\
\multicolumn{3}{|C{0.96\linewidth}|}{\AxiomC{}\noLine\UnaryInfC{$\Gamma\vdash x_1\dots x_n\subseteq y_1\dots y_n$}\RightLabel{\incpro~(a)}\UnaryInfC{$\Gamma\vdash x_{i_1}\dots x_{i_k}\subseteq y_{i_1}\dots y_{i_k}$}\noLine\UnaryInfC{}\DisplayProof
\quad
\AxiomC{}\noLine\UnaryInfC{$\Gamma\vdash\mathsf{x}\subseteq \mathsf{y}$}\AxiomC{}\noLine\UnaryInfC{$\Gamma\vdash\alpha(\mathsf{y}/\mathsf{z})$}\RightLabel{\inccmp~(b)}\BinaryInfC{$\Gamma\vdash\alpha(\mathsf{x}/\mathsf{z})$}\noLine\UnaryInfC{}\DisplayProof}\\
\multicolumn{3}{|C{0.96\linewidth}|}{\def\ScoreOverhang{0.5pt}
\AxiomC{}\noLine\UnaryInfC{$\Gamma\vdash\dep(\mathsf{x})$}\AxiomC{}\noLine\UnaryInfC{$\Gamma\vdash\mathsf{y}\subseteq\mathsf{z}$}\RightLabel{\incwcon}\BinaryInfC{$\Gamma\vdash \mathsf{x}\mathsf{y}\subseteq\mathsf{x}\mathsf{z}$}\noLine\UnaryInfC{}\DisplayProof
\quad\quad
\AxiomC{}\noLine\UnaryInfC{$\Gamma\vdash\wcn\mathsf{x}\subseteq \mathsf{y}$}\AxiomC{}\noLine\UnaryInfC{$\Gamma,\mathsf{c}\subseteq \mathsf{x},\wcn\mathsf{c}\subseteq \mathsf{y}\vdash\phi$}\RightLabel{\incwe~(c)}\BinaryInfC{$\Gamma\vdash \phi$}\noLine\UnaryInfC{}\DisplayProof}\\
\multicolumn{3}{|C{0.96\linewidth}|}{\AxiomC{}\noLine\UnaryInfC{$\Gamma\vdash\con(\mathsf{x})$}\AxiomC{}\noLine\UnaryInfC{$\Gamma\vdash\mathsf{x}\subseteq \mathsf{y}$}\RightLabel{\incweo}\BinaryInfC{$\Gamma\vdash\existso \mathsf{z}(\mathsf{zx}\subseteq\mathsf{wy})$}\noLine\UnaryInfC{}\DisplayProof
\quad\AxiomC{}\noLine\UnaryInfC{$\Gamma\vdash\wcn\lambda(\mathsf{x})$}\AxiomC{}\noLine\UnaryInfC{$\Gamma,\mathsf{c}\subseteq \mathsf{x},\wcn\lambda(\mathsf{c})\vdash\phi$}\RightLabel{\foawe~(c)}\BinaryInfC{$\Gamma\vdash \phi$}\noLine\UnaryInfC{}\DisplayProof}\\
\multicolumn{3}{|C{0.96\linewidth}|}{
\def\ScoreOverhang{0.5pt}
\AxiomC{$\Gamma,\existso \mathsf{z}R\mathsf{z},\phi(R)\vdash\bot$}\RightLabel{\incwi~~(d)}\UnaryInfC{$\Gamma,\phi(\mathsf{v})\vdash\bot$}\noLine\UnaryInfC{}\DisplayProof}\\
\multicolumn{3}{|L{.96\linewidth}|}{\footnotesize (a). $\{i_1,\dots,i_k\}\subseteq\{1,\dots,n\}$\quad (b). $\alpha$ is $\wcn$ and inclusion atom-free, and the free variables of $\alpha(\mathsf{z})$ are among $\mathsf{z}$. \quad (c). $\mathsf{c}$ is a sequence of constant symbols  not occurring in $\Gamma$ or $\phi$, and $\lambda$ is a first-order atom.}\\
\multicolumn{3}{|L{.96\linewidth}|}{{\footnotesize (d). $\Gamma$ is a set of sentences in which $R$ does not occur, $\phi(R)$ is an inclusion atom-free sentence in which  $R$ occurs only in the form $R\mathsf{x}$, and $\phi(\mathsf{v})$ is the formula with free variables $\mathsf{v}$ obtained from $\phi(R)$ by replacing every $R\mathsf{x}$ by $\mathsf{x}\subseteq\mathsf{v}$.
%
}}\\\hline
\end{tabular}
\label{tb_rules_fot}
\end{table}

The weak disjunction $\vvee$ admits the usual introduction and elimination rule, whereas the usual elimination rule is not sound for $\vee$.
The soundness of the introduction and elimination rule for $\existso$ follows from the equivalence $\existso x\phi\equiv \exists x(\dep(x)\wedge\phi)$. 
The introduction and elimination rule for $\forallo$ have a similar flavor.  When applying these introduction and elimination rules for the two weak quantifiers $\existso$ and $\forallo$ (e.g., in Proposition \ref{atm_var_prf}\ref{atm_var_prf_vivj}\ref{atm_var_prf_foatm} below), we will often introduce fresh constant symbols $c$, which we  assume to be always available. {Our $\forallo$-introduction and $\existso$-elimination rules both have a variable version and a constant version. It is possible that one version is actually derivable from the other; we leave it as future work to determine whether this is the case. }
The weak classical negation $\wcn$ admits the classical rules. The constancy atom introduction rule \consi characterizes the fact that constants or terms formed by constants have constant values.
 The rules \incid, \incpro, \inctr and \inccmp for inclusion atoms were first introduced  in \cite{Hannula_fo_ind_13} for first-order independence logic, where the first three rules (i.e., identity, projection and transitivity) are known to  axiomatize completely the implication problem of inclusion dependencies in database theory \cite{inclusion_dep_CFP_82}. The last rule \inccmp, inclusion atom compression rule, which applies to formulas $\alpha$ without any occurrence of $\wcn$ and inclusion atom only, describes essentially the flatness of such formulas through their interaction with inclusion atoms.
The two weakening rules \incwcon and \incweo for inclusion atoms extend the length of an inclusion atom. These rules for inclusion atoms  are also sound  if constants are allowed to occur as arguments in inclusion atoms (i.e., to allow inclusion atoms, e.g., of the form $c\mathsf{x}\subseteq v\mathsf{y}$). 
Such more general rules for inclusion atoms are easily derivable in our system, using the fact $\vdash\existso x(x=c)$ (which follows from \consi and \existsoi) and the rules for identity. 
The elimination rules \incwe and \foawe  describe the meanings of a negated inclusion atom $\wcn\mathsf{x}\subseteq \mathsf{y}$ and a negated first-order atom $\lambda(\mathsf{x})$ by  providing a witness $\mathsf{c}$ for the failure of the atoms. These two rules are designed for deriving Proposition \ref{atm_var_prf}\ref{atm_var_prf_vivj}\ref{atm_var_prf_foatm} (which are key for the normal form lemma, Lemma \ref{FOT_NF_prf}). 
The rule \incwi  simulates the transformation in \Cref{fo2fot}, which will be crucial for the proof of the completeness theorem (\Cref{completeness_fot}). 


Let us also mention that closely related to \FOT is the {\em inquisitive first-order logic} (\textsf{InqBQ}) \cite{Ciardelli_PhD}, which adopts a similar type of team semantics in a different setting. The disjunction and quantifiers in \textsf{InqBQ} have also essentially the same semantics as our  weak disjunction $\vvee$ and weak  quantifiers $\forallo$ and $\existso$.  A sound system of natural deduction for \textsf{InqBQ} was introduced in \cite{Ciardelli_PhD}, and it was shown in \cite{Grilletti2021} to be complete for the classical antecedent fragment of \textsf{InqBQ}. The introduction and elimination rules for the quantifiers in this system correspond exactly to the introduction and elimination rules for $\forallo$ and $\existso$ in our system for \FOT, as  the so-called {\em rigid terms} in \textsf{InqBQ} can be interpreted in our setting as terms of constant values.

Over constant values, our system can actually be viewed, in the sense of \Cref{FO2FOT_eqiv} below, as an extension of the standard system of first-order logic. These two systems are bridged via the rule \incwi in a way to be presented in the proof of the completeness theorem (\Cref{completeness_fot}). 
The additional relation symbol $R$ in the rule \incwi is introduced for the purpose of proving the completeness theorem; a similar trick of introducing additional relation symbol in a deduction system for a team-based logic was used also in \cite{Kontinen15foc}.    Most derivations in our system (see illustrations in \Cref{sec:app}), however, do not involve the use of the additional symbol $R$ or the rule \incwi. 
How  to simplify the rule \incwi is left as future work.
%


\begin{theorem}[Soundness]\label{atm2rel}
$\Gamma\vdash_\FOT\phi\Longrightarrow\Gamma\models\phi$.
\end{theorem}
\begin{proof}
We only give detailed proof for the soundness of the nontrivial rules. 

\inccmp: It suffices to show that $\mathsf{x}\subseteq\mathsf{y},\alpha(\mathsf{y}/\mathsf{z})\models \alpha(\mathsf{x}/\mathsf{z})$. Suppose that $M\models_X\mathsf{x}\subseteq\mathsf{y}$ and  $M\models_X\alpha(\mathsf{y}/\mathsf{z})$. Put $X[\mathsf{x}]=\{s(\mathsf{x})\mid s\in X\}$. Consider 
\(X_0=\{s\in X\mid s(\mathsf{y})\in X[\mathsf{x}]\}\subseteq X.\)
Since $M\models_X\mathsf{x}\subseteq\mathsf{y}$, it is easy to verify that $X[\mathsf{x}]=X_0[\mathsf{y}]$. Now, since $\alpha$ is downward closed, we have that $M\models_{X_0}\alpha(\mathsf{y}/\mathsf{z})$, which then implies that $M\models_X\alpha(\mathsf{x}/\mathsf{z})$ by locality (as free variables in the formula $\alpha(\mathsf{z})$ are among $\mathsf{z}$).

\incwcon: It suffices to show that $\dep(\mathsf{x}),\mathsf{y}\subseteq\mathsf{z}\models\mathsf{xy}\subseteq\mathsf{xz}$. Suppose that $M\models_X\dep(\mathsf{x})$ and $M\models_X\mathsf{y}\subseteq\mathsf{z}$. The latter implies that for any $s\in X$,  there exists $s'\in X$ such that

\noindent   $s(\mathsf{y})=s'(\mathsf{z})$. But since $M\models_X\dep(\mathsf{x})$, we also have  $s(\mathsf{x})=s'(\mathsf{x})$. Thus $s(\mathsf{xy})=$$s'(\mathsf{xz})$, thereby  $M\models_X\mathsf{xy}\subseteq\mathsf{xz}$.

\incweo:  It suffices to show that $\dep(\mathsf{x}),\mathsf{x}\subseteq\mathsf{y}\models\existso \mathsf{z}(\mathsf{zx}\subseteq\mathsf{wy})$. Suppose that $M\models_X\dep(\mathsf{x})$ and $M\models_X\mathsf{x}\subseteq\mathsf{y}$. Pick $s_0\in X$. Then there exists $s'_0\in X$ such that $s_0(\mathsf{x})=s'_0(\mathsf{y})$. Let $\mathsf{a}=s_0'(\mathsf{w})$. We show that $M\models_{X(\mathsf{a}/\mathsf{z})}\mathsf{zx}\subseteq\mathsf{wy}$. For any $s\in X(\mathsf{a}/\mathsf{z})$, clearly $s'_0(\mathsf{a}/\mathsf{z})\in X(\mathsf{a}/\mathsf{z})$. Since  $M\models_X\dep(\mathsf{x})$, we have that $s(\mathsf{x})=s_0(\mathsf{x})$. Thus 
\(s(\mathsf{zx})=\mathsf{a}s_0(\mathsf{x})=s'_0(\mathsf{w})s_0'(\mathsf{y})=s'_0(\mathsf{a}/\mathsf{z})(\mathsf{wy}).\)

\incwe: Suppose $\Gamma\models\wcn\mathsf{x}\subseteq \mathsf{y}$ and $\Gamma,\mathsf{c}\subseteq \mathsf{x},\wcn\mathsf{c}\subseteq \mathsf{y}\models\phi$, and suppose that for some $\LL$-model $M$ and nonempty team $X$, $M\models_X\Gamma$. Then we have that $M\models_X\wcn\mathsf{x}\subseteq \mathsf{y}$, which implies that there exists $s\in X$ such that for the $\LL(\mathsf{c})$-model $(M,s(\mathsf{x}))$, we have $(M,s(\mathsf{x}))\models_X\mathsf{c}\subseteq \mathsf{x}\wedge\wcn\mathsf{c}\subseteq \mathsf{y}$. Thus, by the assumption, $(M,s(\mathsf{x}))\models_X\phi$, which gives $M\models_X\phi$ since constant symbols in $\mathsf{c}$ do not occur in $\phi$.

\foawe: Suppose $\Gamma\models\wcn\lambda(\mathsf{x})$ and $\Gamma,\mathsf{c}\subseteq \mathsf{x},\wcn\lambda(\mathsf{c})\models\phi$, and suppose that  for some $\LL$-model $M$ and nonempty team $X$, $M\models_X\Gamma$.  Then we have that $M\models_X\wcn\lambda(\mathsf{x})$, which implies that $M\not\models_X\lambda(\mathsf{x})$. Since $\lambda(\mathsf{x})$ is flat, this means that  there exists $s\in X$ such that for the $\LL(\mathsf{c})$-model $(M,s(\mathsf{x}))$, $(M,s(\mathsf{x}))\not\models_X\lambda(\mathsf{c})$, or $(M,s(\mathsf{x}))\models_X\wcn\lambda(\mathsf{c})$. Clearly, we also have that $(M,s(\mathsf{x}))\models_X\mathsf{c}\subseteq \mathsf{x}$. Hence, by assumption we conclude that $(M,s(\mathsf{x}))\models_X\phi$, which gives $M\models_X\phi$ since $\mathsf{c}$ do not occur in $\phi$.

 \incwi: Suppose $\Gamma,\phi(\mathsf{v})\not\models\bot$, where the formulas in $\Gamma\cup\{\phi\}$ are in the vocabulary $\mathcal{L}$ with $R\notin \mathcal{L}$. 
 We may  w.l.o.g. assume that  $\phi(\mathsf{v})$ is in prenex and negation normal form (cf. Corollary \ref{prenex_nnf}). Then there exist an $\mathcal{L}$-model $M$ and a nonempty team $X$ such that $M\models_X\Gamma$ and $M\models_X\phi(\mathsf{v})$. 
 Consider the \FO-sentence  $\phi_\ast(R)$ obtained from the inclusion atom-free \FOT-sentence $\phi(R)$ by replacing every logical constant in \FOT by its counterpart in \FO, i.e., by replacing 
$\wcn$ by $\neg$, $\vvee$ by $\vee$, $\forallo$ by $\forall$, and $\existso$ by $\exists$.
Now, by (the proof of) \Cref{fo2fot}, we have  $(M,rel(X))\models\phi_\ast(R)$ in \FO.
%
By Lemma \ref{fo=fot}, $(M,rel(X))\models_{\{\emptyset\}}\phi(R)$ in \FOT follows
. Since $\Gamma$ is a set of sentences in which $R$ does not occur, we also have $(M,rel(X))\models_{\{\emptyset\}}\Gamma$. Also, since $X\neq\emptyset$,  $(M,rel(X))\models\existso \mathsf{z}R\mathsf{z}$. Hence, we conclude $\Gamma,\existso \mathsf{z}R\mathsf{z},\phi(R)\not\models\bot$.
\end{proof}

%

We list some  basic facts concerning  the logical constants in \FOT  in the following proposition. 

\begin{prop}\label{FOT_der_rules}
\begin{enumerate}[label=(\roman*)]
\item\label{FOTd_der_rules_forallo_existso_con} $\Gamma,\forallo x\phi\vdash\phi(c/x)$ and 
$\Gamma,\phi(c/x)\vdash\existso x\phi$.
\item\label{FOTd_der_rules_qf_conj_disj} Let $Q^1\in\{\forallo, \existso\}$. Then $Q^1 x\phi\wedge\psi\dashv\vdash Q^1 x(\phi\wedge\psi)$ and $Q^1 x\phi\vvee\psi\dashv\vdash Q^1 x(\phi\vvee\psi)$, whenever $x\notin\textsf{Fv}(\psi)$.
\item\label{FOT_der_rules_cn_swap} $\vdash\phi\vvee\wcn\phi$, $\wcn\wcn\phi\vdash\phi$, and $\Gamma,\phi\vdash\psi\iff\Gamma,\wcn\psi\vdash\wcn\phi$.
\item\label{FOT_der_rules_qf_cn} $\wcn \forallo x\phi\dashv\vdash \existso x\wcn\phi$ and $\wcn \existso x\phi\dashv\vdash \forallo x\wcn\phi$. 
\item\label{FOT_der_rules_conjdis_cn} $\wcn(\phi\vvee\psi)\dashv\vdash\wcn\phi\wedge\wcn\psi$ and $\wcn(\phi\wedge\psi)\dashv\vdash\wcn\phi\vvee\wcn\psi$.
\end{enumerate}
\end{prop}
\begin{proof}
Item \ref{FOTd_der_rules_forallo_existso_con} follows easily from \consi (the fact $\vdash\dep(c)$ in particular), \foralloe and \existsoi. The proofs for the other items are routine and thus left to the reader.
\end{proof}

In the sequel, we often abbreviate a sequence $Q^1_1x_1\dots Q^1_nx_n$ of quantifications (with each $Q^1_i\in\{\forallo,\existso\}$) as $Q^1\mathsf{x}$. We say that a formula  is in prenex and negation normal form if it is of the form $Q^1\mathsf{x}\theta$, where $\theta$ is a quantifier-free formula with negation $\wcn$ occurring only in front of atomic formulas.

\begin{cor}\label{prenex_nnf}
Every \FOT-formula $\phi$ is provably equivalent to a formula $Q^1\mathsf{x}\theta$ in prenex and negation normal form. 
\end{cor}
\begin{proof}
Apply  exhaustively  Proposition \ref{FOT_der_rules}\ref{FOTd_der_rules_qf_conj_disj}-\ref{FOT_der_rules_conjdis_cn}. 
\end{proof}

\begin{cor}\label{replacement_deduction_thm}
\begin{enumerate}[label=(\roman*)]
\item \textbf{Replacement Lemma:}  $\theta\dashv\vdash\chi\Longrightarrow\phi\dashv\vdash\phi(\chi/\theta)$, where the formula $\phi(\chi/\theta)$ is obtained from $\phi$ by replacing an occurrence of $\theta$ in $\phi$ by $\chi$.
\item \textbf{Deduction Theorem:} $\Gamma,\phi\vdash\psi \iff \Gamma\vdash\phi\cimp\psi.$
\end{enumerate}
\end{cor}
\begin{proof}
Item (i) is proved by a routine inductive argument that uses Proposition \ref{FOT_der_rules}\ref{FOTd_der_rules_forallo_existso_con} in the cases for the quantifiers $\forallo$ and $\existso$. Item (ii) is proved easily by using  Proposition \ref{FOT_der_rules}\ref{FOT_der_rules_cn_swap}.
\end{proof}

 
 In the following proposition, we list some derivable technical clauses that will be used in the proof of the completeness theorem. 
\begin{prop}\label{atm_var_prf}
Let $\xi$ and $\eta$ be two sequences of variables of the same length. 
\begin{enumerate}[label=(\roman*)]
 \item\label{atm_var_prf_repeat} 
 $\mathsf{x}\mathsf{y}\mathsf{\xi}\subseteq \mathsf{v}\mathsf{v}\mathsf{\eta}\dashv\vdash \mathsf{x}=\mathsf{y}\wedge \mathsf{x}\mathsf{\xi}\subseteq \mathsf{v}\mathsf{\eta} $.
 \item\label{atm_var_prf_vivj} $\xi\subseteq \eta\dashv\vdash \forallo\mathsf{x}(\mathsf{x}\subseteq \xi\cimp\mathsf{x}\subseteq\eta)$, where variables in $\mathsf{x}$ are fresh.

\item\label{atm_var_prf_vix} $\con(\mathsf{z})\vdash\mathsf{w}\xi\subseteq \mathsf{z}\eta\cbimp(\mathsf{w}=\mathsf{z}\wedge \xi\subseteq\eta)$.



\item\label{atm_var_prf_vext} $\con(\mathsf{x})\vdash\mathsf{x}\subseteq\mathsf{v}\cbimp\existso\mathsf{y}(\mathsf{x}\mathsf{y}\subseteq \mathsf{v}\mathsf{u})$, where variables in $\mathsf{y}$ are fresh.

\item\label{atm_var_prf_foatm} If $\lambda(\mathsf{z})$ is a first-order atom, then $\lambda(\mathsf{z})\dashv\vdash \forallo \mathsf{w}(\mathsf{w}\subseteq \mathsf{z}\cimp \lambda(\mathsf{w}))$, where variables in $\mathsf{w}$ are fresh.
\end{enumerate}
\end{prop}
\begin{proof}
Item \ref{atm_var_prf_repeat}: 
For the right to left direction, we first derive $\mathsf{x}\mathsf{\xi}\subseteq \mathsf{v}\mathsf{\eta}\vdash \mathsf{xx}\mathsf{\xi}\subseteq \mathsf{vv}\mathsf{\eta}$ by applying the rule \incpro (which can be applied to add repeated arguments to an inclusion atom). Then, we apply the rules of identity to obtain $ \mathsf{xx}\mathsf{\xi}\subseteq \mathsf{vv}\mathsf{\eta},\mathsf{x}=\mathsf{y}\vdash \mathsf{xy}\mathsf{\xi}\subseteq \mathsf{vv}\mathsf{\eta}$.

 For the left to right direction, $\mathsf{x}\mathsf{y}\mathsf{\xi}\subseteq \mathsf{v}\mathsf{v}\mathsf{\eta}\vdash\mathsf{x}\mathsf{\xi}\subseteq \mathsf{v}\mathsf{\eta}$ follows from \incpro. Next, by \incpro and rules of identity, we have  $\mathsf{x}\mathsf{y}\mathsf{\xi}\subseteq \mathsf{v}\mathsf{v}\mathsf{\eta}\vdash\mathsf{x}\mathsf{y}\subseteq \mathsf{v}\mathsf{v}\vdash\mathsf{x}\mathsf{y}\subseteq \mathsf{v}\mathsf{v}\wedge\mathsf{v}=\mathsf{v}$. Finally, viewing formula $\mathsf{v}=\mathsf{v}$ as $(\mathsf{x}=\mathsf{y})(\mathsf{vv}/\mathsf{xy})$, we applying \inccmp to conclude $\mathsf{x}\mathsf{y}\subseteq \mathsf{v}\mathsf{v},\mathsf{v}=\mathsf{v}\vdash \mathsf{x}=\mathsf{y}$.
 


Item \ref{atm_var_prf_vivj}:  For the right to left direction, by Proposition \ref{FOT_der_rules}\ref{FOT_der_rules_cn_swap}\ref{FOT_der_rules_qf_cn}, it suffices to  show the contrapositive $\wcn\xi\subseteq\eta\vdash \existso\mathsf{x}(\mathsf{x}\subseteq \xi\wedge\wcn\mathsf{x}\subseteq\eta)$. For any sequence $\mathsf{c}$ of fresh constant symbols, we have $ \mathsf{c}\subseteq \xi,\wcn\mathsf{c}\subseteq\eta\vdash \existso\mathsf{x}(\mathsf{x}\subseteq \xi\wedge\wcn\mathsf{x}\subseteq\eta)$ by Proposition \ref{FOT_der_rules}\ref{FOTd_der_rules_forallo_existso_con}. Then the desired clause follows from \incwe.
For the other direction, by \foralloi, it suffices to show that $\xi\subseteq \eta\vdash \mathsf{c}\subseteq \xi\cimp\mathsf{c}\subseteq\eta$ for $\mathsf{c}$ a sequence of fresh constant symbols, which is further reduced to showing that $\xi\subseteq \eta, \mathsf{c}\subseteq \xi\vdash\mathsf{c}\subseteq\eta$. But this follows from \inctr.


Item \ref{atm_var_prf_vix}:  We first show $\con(\mathsf{z})\vdash\mathsf{w}\xi\subseteq \mathsf{z}\eta\cimp(\mathsf{w}=\mathsf{z}\wedge \xi\subseteq\eta)$, which is equivalent to $\con(\mathsf{z}),\mathsf{w}\xi\subseteq \mathsf{z}\eta\vdash\mathsf{w}=\mathsf{z}\wedge \xi\subseteq\eta$. By \incwcon and \incpro we have $\con(\mathsf{z}),\mathsf{w}\subseteq\mathsf{z}\vdash\mathsf{w}\mathsf{z}\subseteq\mathsf{z}\mathsf{z}$. By item \ref{atm_var_prf_repeat}, $\mathsf{w}\mathsf{z}\subseteq\mathsf{z}\mathsf{z}\vdash\mathsf{w}=\mathsf{z}$. Hence, by \incpro the desired clause follows. 
Next, we show $\con(\mathsf{z}),\mathsf{w}=\mathsf{z}, \xi\subseteq\eta\vdash\mathsf{w}\xi\subseteq \mathsf{z}\eta$. Again by \incwcon we have that $\xi\subseteq\eta,\con(\mathsf{z})\vdash\mathsf{z}\xi\subseteq\mathsf{z}\eta$, and thus the desired clause follows from rules of identity.


Item \ref{atm_var_prf_vext}:  The direction $\con(\mathsf{x}),\mathsf{x}\subseteq\mathsf{v}\vdash\existso\mathsf{y}(\mathsf{x}\mathsf{y}\subseteq \mathsf{v}\mathsf{u})$ is given by $\subseteq\!\textsf{W}_{\existso}$, and the other direction $\con(\mathsf{x}),\existso\mathsf{y}(\mathsf{x}\mathsf{y}\subseteq \mathsf{v}\mathsf{u})\vdash\mathsf{x}\subseteq\mathsf{v}$ follows easily from \incpro.


Item \ref{atm_var_prf_foatm}:  We first show the left to right direction, which, by \foralloi,  is reduced to  $\lambda(\mathsf{z})\vdash \mathsf{c}\subseteq \mathsf{z}\cimp \lambda(\mathsf{c})$ for $\mathsf{c}$ a sequence of fresh constant symbols. But this follows from \inccmp.
Conversely, we show the contrapositive $\wcn\lambda(\mathsf{z})\vdash\existso \mathsf{w}(\mathsf{w}\subseteq \mathsf{z}\wedge \wcn\lambda(\mathsf{w}))$. For any sequence $\mathsf{c}$ of fresh constant symbols, we have $ \mathsf{c}\subseteq \mathsf{z},\wcn\lambda(\mathsf{c})\vdash \existso \mathsf{w}(\mathsf{w}\subseteq \mathsf{z}\wedge \wcn\lambda(\mathsf{w}))$ by Proposition \ref{FOT_der_rules}\ref{FOTd_der_rules_forallo_existso_con}. Then the desired clause follows from \foawe.
\end{proof}

Note that in the above proof, in items \ref{atm_var_prf_vivj} and \ref{atm_var_prf_foatm}, when Proposition \ref{FOT_der_rules}\ref{FOTd_der_rules_forallo_existso_con} (or the rule \existsoi) and the rule \foralloi are applied, we have introduced fresh constant symbols $c$, which, as commented already, we assume to be always available. The reader can compare these derivations with similar derivations  in the system of the usual first-order logic that involve the use of the standard introduction rules for the standard first-order quantifiers $\exists$ and $\forall$ instead. In those cases, one often applies similar tricks and introduces fresh first-order variables $x$, which are taken from some fixed infinite set $\mathsf{Var}$ of first-order variables. 
In the present setting, 
we essentially assumes (in a similar manner) an infinite set $\mathbf{Var}$ of second-order variables that contains infinitely many function variables of every arity (with constant variables identified as  $0$-ary function variables) and infinitely many relation variables of every arity.

To prove the completeness theorem we also need the following three  lemmas. The first lemma emphasizes the fact that all variables quantified by the weak quantifiers have constant values, 
the second one proves a normal form for \FOT-formulas, and the last one shows that  derivations in the system of \FO can be simulated in the system of \FOT. 
\begin{lem}\label{montone_constancy_lm}
Let $\phi(\mathsf{v})=Q^1\mathsf{x}\theta(\mathsf{x},\mathsf{v})$ be a formula in prenex and negation normal form.
Then  $\phi\dashv\vdash\phi_{\mathsf{con}}$, where $\phi_{\mathsf{con}}$
is the formula obtained from $\phi$ by replacing every (first-order or inclusion) literal $\mu(\mathsf{x},\mathsf{v})$ (i.e., an atom or negated atom)  by $\mu\wedge\con(\mathsf{x})$.

\end{lem}
\begin{proof}
By  Proposition \ref{FOT_der_rules}\ref{FOTd_der_rules_qf_conj_disj}, $Q^1\textsf{I}$ and $Q^1\textsf{E}$, it is easy to prove  that 
\(Q^1 \mathsf{x}\theta(\mathsf{x},\mathsf{v})\dashv\vdash Q^1 \mathsf{x}(\theta(\mathsf{x},\mathsf{v})\wedge \con(\mathsf{x})).\)
We then push the formula $\con(\mathsf{x})$ inside the quantifier-free formula $\theta$ in negation normal form all the way to the front of literals by using Replacement Lemma (Corollary \ref{replacement_deduction_thm}(i)) and the (standard) equivalences
\((\theta_0\wedge \theta_1)\wedge \con(\mathsf{x})\dashv\vdash(\theta_0\wedge \con(\mathsf{x}))\wedge (\theta_1\wedge \con(\mathsf{x}))\)
\(\text{ and }(\theta_0\vvee \theta_1)\wedge \con(\mathsf{x})\dashv\vdash(\theta_0\wedge \con(\mathsf{x}))\vvee (\theta_1\wedge \con(\mathsf{x})).\)
\end{proof}

\begin{lem}\label{FOT_NF_prf}
For every \FOT-formula $\phi$, we have $\phi(\mathsf{v})\dashv\vdash Q^1\mathsf{x}\theta(\mathsf{x},\mathsf{v})$, where $\theta(\mathsf{x},\mathsf{v})$ is a quantifier-free formula in negation normal form in which first-order atoms are of the form $\lambda(\mathsf{x})$,  and inclusion atoms  are of the form $\mathsf{x}_{i}\subseteq \mathsf{v}$ for some variables $\mathsf{x_i} $ from $\mathsf{x}$. 
\end{lem}
\begin{proof}
By Corollary \ref{prenex_nnf}, we may assume that  $\phi(\mathsf{v})$ is in prenex and negation normal form.
Assume that the bound variables of $\phi(\mathsf{v})$ are among $\mathsf{x}$. By Lemma \ref{montone_constancy_lm} we may also assume that every literal $\mu(\mathsf{x},\mathsf{v})$ in $\phi$ is replaced by $\mu(\mathsf{x},\mathsf{v})\wedge \con(\mathsf{x})$ (call such a formula a formula in  {\em constant normal form}). 
Observe that now in $\phi(\mathsf{v})$ a generic first-order atom is of the form $\lambda(\mathsf{x},\mathsf{v})$, and a generic inclusion atom is of the form $\eta\xi\rho\sigma\subseteq\eta'\xi'\rho'\sigma'$ (modulo permutation  by \incpro), where $|\eta|=|\eta'|\geq 0$, $|\xi|=|\xi'|\geq 0$, $|\rho|=|\rho'|\geq 0$ and $|\sigma|=|\sigma'|\geq 0$,
\begin{itemize}
\item $(\eta,\eta')=(\mathsf{x}_i,\mathsf{x}_j)$ for some bound variables $\mathsf{x}_i,\mathsf{x}_j$  from $\mathsf{x}$; 
\item $(\xi,\xi')=(\mathsf{x}_i,\mathsf{v}_i)$ for some bound variables $\mathsf{x}_i$ from $\mathsf{x}$, and free variables $\mathsf{v}_i$ from $\mathsf{v}$;
\item $(\rho,\rho')=(\mathsf{v}_i,\mathsf{x}_i)$ for some free variables $\mathsf{v}_i$ from $\mathsf{v}$, and bound variables $\mathsf{x}_i$ from $\mathsf{x}$;

\item $(\sigma,\sigma')=(\mathsf{v}_i,\mathsf{v}_j)$ for some free variables $\mathsf{v}_i,\mathsf{v}_j$ from $\mathsf{v}$.
\end{itemize}
To obtain the required normal form we have to transform every (first-order or inclusion) atom in $\phi$ to the required form. We achieve this in several steps.

In \underline{Step 1} of our transformation, we replace in $\phi(\mathsf{v})$ every inclusion atom of the form $\eta\xi\mathsf{v}_i\sigma\subseteq\eta'\xi'\mathsf{x}_i\sigma'$  by  $\mathsf{v}_i=\mathsf{x}_i\wedge \eta\xi\sigma\subseteq\eta'\xi'\sigma'$. 
Note that by Proposition \ref{atm_var_prf}\ref{atm_var_prf_vix}, we have
\[\con(\mathsf{x}_i)\wedge\eta\xi\mathsf{v}_i\sigma\subseteq\eta'\xi'\mathsf{x}_i\sigma'\dashv\vdash\con(\mathsf{x}_i)\wedge(\mathsf{v}_i=\mathsf{x}_i\wedge \eta\xi\sigma\subseteq\eta'\xi'\sigma')\]
\[\text{and }\con(\mathsf{x}_i)\wedge\wcn\eta\xi\mathsf{v}_i\sigma\subseteq\eta'\xi'\mathsf{x}_i\sigma'\dashv\vdash\con(\mathsf{x}_i)\wedge\wcn(\mathsf{v}_i=\mathsf{x}_i\wedge \eta\xi\sigma\subseteq\eta'\xi'\sigma').\]
Hence, by Replacement Lemma (Corollary \ref{replacement_deduction_thm}(i)), the resulting formula $\phi_1(\mathsf{v})$ is provably equivalent to $\phi$. We  further assume  (here and also in the other steps) that $\phi_1(\mathsf{v})$ is turned into prenex, negation and constant normal form by applying  Corollary \ref{prenex_nnf} 
and Lemma \ref{montone_constancy_lm}.

In \underline{Step 2}, we  replace in $\phi_1(\mathsf{v})$ every first-order atom $\lambda(\mathsf{x},\mathsf{v})$ by  
\(\forallo \mathsf{y}\mathsf{z}(\mathsf{y}\mathsf{z}\subseteq \mathsf{x}\mathsf{v}\cimp\lambda(\mathsf{y},\mathsf{z})).\)
By Proposition \ref{atm_var_prf}\ref{atm_var_prf_foatm}, the resulting formula $\phi_2(\mathsf{v})$ is provably equivalent to $\phi_1(\mathsf{v})$. Up to now, every first-order atom in the formula is transformed to the required  form. The steps afterwards will not generate first-order atoms not in normal form.

In  \underline{Step 3}, we apply Proposition \ref{atm_var_prf}\ref{atm_var_prf_vivj} to replace in $\phi_2(\mathsf{v})$ every inclusion atom  $\eta\xi\mathsf{v}_i\subseteq\eta'\xi'\mathsf{v}_j$ by  
\(\forallo \mathsf{wyz}(\mathsf{wy}\mathsf{z}\subseteq \eta\xi\mathsf{v}_i\cimp\mathsf{wyz}\subseteq \eta'\xi'\mathsf{v}_j),\) and denote the resulting formula by $\phi_3(\mathsf{v})$.
In  \underline{Step 4}, we apply Proposition \ref{atm_var_prf}\ref{atm_var_prf_vix} to replace in $\phi_3(\mathsf{v})$ every inclusion atom  $\mathsf{x}_i\mathsf{x}_k\subseteq\mathsf{x}_j\mathsf{v}_k$ by  $\mathsf{x}_i=\mathsf{x}_j\wedge\mathsf{x}_k\subseteq \mathsf{v}_k$, and denote the resulting formula by $\phi_4(\mathsf{v})$. 
Up to now every inclusion atom in the formula is transformed to the form $\mathsf{x}_i\subseteq\mathsf{v}_i$, where $\mathsf{x}_i$ are  bound variables  and $\mathsf{v}_i$ are  free variables in $\phi_4(\mathsf{v})$. Yet, $\mathsf{v}_i$ may contain repetitions, and it may  be only a subsequence of $\mathsf{v}$. Handling these requires two additional steps. 

In  \underline{Step 5}, we remove repetitions on the right side of the inclusion atoms, by applying Proposition \ref{atm_var_prf}\ref{atm_var_prf_repeat}  to replace in $\phi_4(\mathsf{v})$ every inclusion atom of the form $\mathsf{x}_i\mathsf{x}_j\mathsf{x}_k\subseteq\mathsf{v}_i\mathsf{v}_i\mathsf{v}_j$ by $\mathsf{x}_i=\mathsf{x}_j\wedge \mathsf{x}_i\mathsf{x}_k\subseteq\mathsf{v}_i\mathsf{v}_j$. Denote the resulting  formula by $\phi_5(\mathsf{v})$.
In  \underline{Step 6}, we extend the length of those shorter inclusion atoms. Assuming $\mathsf{v}=\mathsf{v}_i\mathsf{v}_j$, we apply  Proposition  \ref{atm_var_prf}\ref{atm_var_prf_vext}  to replace in $\phi_5(\mathsf{v})$ every inclusion atom of the form $\mathsf{x}_i\subseteq\mathsf{v}_i$ by $\existso \mathsf{y}(\mathsf{x}_i\mathsf{y}\subseteq\mathsf{v}_i\mathsf{v}_j)$. Denote the resulting  formula by $\phi_6(\mathsf{v})$. 

As before we assume that $\phi_6(\mathsf{v})$ is turned into prenex and negation normal form, but now we shall apply Lemma \ref{montone_constancy_lm} in a reverse manner to remove the added constancy atoms  for each literal in $\phi_6$. Finally, the resulting provably equivalent formula is in the required normal form. Note that our transformation clearly terminates, because we have performed the steps of the transformation in such an order that each step  will not generate new formulas for which the transformations in the previous steps apply.
\end{proof}


\begin{lem}\label{FO2FOT_eqiv}
Let $\Delta\cup\{\delta\}$ be a set of \FO-formulas whose free variables are among $\mathsf{x}$. If $\Delta\vdash_\FO\delta$, then $\Delta^\ast(\mathsf{c}/\mathsf{x})\vdash_\FOT\delta^\ast(\mathsf{c}/\mathsf{x})$, where $\Delta^\ast=\{\alpha^\ast\mid \alpha\in\Delta\}$, $(\cdot)^\ast$ is the \linebreak operation that replaces every  logical constant in \FO by its counterpart in \FOT, and $\mathsf{c}$ is a sequence of fresh constant symbols. 
In particular, if $\Delta\cup\{\delta\}$ is a set of \FO-sentences, then $\Delta\vdash_\FO\delta$ implies $\Delta^\ast\vdash_\FOT\delta^\ast$.
\end{lem}
\begin{proof}
We prove that $\Delta\vdash_\FO\delta$ implies $\Delta^\ast(\mathsf{c}/\mathsf{x})\vdash_\FOT\delta^\ast(\mathsf{c}/\mathsf{x})$ by induction on the depth of the proof tree of $\Delta\vdash_\FO\delta$. If the proof tree has depth 1, then either  $\delta\in \Delta$ or $\delta$ is the identity axiom $t=t$. In both cases  $\Delta^\ast(\mathsf{c}/\mathsf{x})\vdash_\FOT \delta^\ast(\mathsf{c}/\mathsf{x})$ trivially holds. 

If the proof tree has depth $>1$, and the last step of the derivation of $\Delta\vdash_\FO\delta$ is an application of a rule for $\neg$ or $\wedge$ or $\vee$ in \FO, then we derive $\Delta^\ast(\mathsf{c}/\mathsf{x})\vdash_\FOT\delta^\ast(\mathsf{c}/\mathsf{x})$ by directly applying the induction hypothesis and the corresponding (classical)  rules for $\wcn$ or $\wedge$ or $\vvee$ in our system for \FOT.

If the last step of the derivation of $\Delta\vdash_\FO\delta$ is an application of the $\exists\textsf{I}$ rule:
\begin{center}
\AxiomC{$~\vdots~\pi$}\noLine\UnaryInfC{$\Delta\vdash\alpha(t/x)$}\RightLabel{$\exists$\textsf{I}}\UnaryInfC{$\Delta\vdash\exists x\alpha$}\DisplayProof
\end{center}
\noindent where the free variables occurring in $\Delta(\mathsf{y})$ and $\exists x\alpha(x,\mathsf{y})$ are among $\mathsf{y}$, and the variables and constant symbols occurring in the term $t$ are, respectively, among $\mathsf{v}\mathsf{y}$ and $\mathsf{d}$ (denoted as $t(\mathsf{v}\mathsf{y},\mathsf{d})$), then the corresponding derivation in \FOT is:
\begin{center}
\AxiomC{$~\vdots~\pi^\ast$}\noLine\UnaryInfC{$\Delta^\ast(\mathsf{c}/\mathsf{y})\vdash\alpha^\ast(t(\mathsf{c}'/\mathsf{v},\mathsf{c}/\mathsf{y},\mathsf{d})/x,\mathsf{c}/\mathsf{y})$}\AxiomC{}\noLine\UnaryInfC{}\RightLabel{\consi}\UnaryInfC{$\Delta^\ast(\mathsf{c}/\mathsf{y})\vdash\con(\mathsf{c'cd})$}\RightLabel{\consi}\UnaryInfC{$\Delta^\ast(\mathsf{c}/\mathsf{y})\vdash\dep(t(\mathsf{c}'/\mathsf{v},\mathsf{c}/\mathsf{y},\mathsf{d}))$}\RightLabel{\existsoi}\BinaryInfC{$\Delta^\ast(\mathsf{c}/\mathsf{y})\vdash\existso x\alpha^\ast(x,\mathsf{c}/\mathsf{y})$}\noLine\UnaryInfC{}\DisplayProof
\end{center}
\noindent where $\pi^\ast$ is, by  induction hypothesis, the derivation in \FOT that corresponds to $\pi$.

If the last step of the derivation of $\Delta\vdash_\FO\delta$ is an application of the $\exists\textsf{E}$ rule:
\begin{center}
\AxiomC{$~\vdots~\pi_1$}\noLine\UnaryInfC{$\Delta\vdash\exists x\alpha$}\AxiomC{$~\vdots~\pi_2$}\noLine\UnaryInfC{$\Delta,\alpha(v/x)\vdash\delta$}\RightLabel{$\exists$\textsf{E}}\BinaryInfC{$\Delta\vdash\delta$}\DisplayProof
\end{center}
\noindent where $ v\notin\textsf{Fv}(\Delta\cup\{\alpha,\delta\})$ and the free variables occurring in $\Delta,\exists x\alpha,\delta$ are among $\mathsf{y}$, then the corresponding derivation in \FOT is:
\def\ScoreOverhang{0.5pt}
\begin{center}
\AxiomC{$~\vdots~\pi_1^\ast$}\noLine\UnaryInfC{$\Delta^\ast(\mathsf{c}/\mathsf{y})\vdash\existso x\alpha^\ast(x,\mathsf{c}/\mathsf{y})$}\AxiomC{$~\vdots~\pi_2^\ast$}\noLine\UnaryInfC{$\Delta^\ast(\mathsf{c}/\mathsf{y}),(\alpha(v/x))^\ast(d/v,\mathsf{c}/\mathsf{y})\vdash\delta^\ast(\mathsf{c}/\mathsf{y})$}\UnaryInfC{$\Delta^\ast(\mathsf{c}/\mathsf{y}),\alpha^\ast(d/x,\mathsf{c}/\mathsf{y})\vdash\delta^\ast(\mathsf{c}/\mathsf{y})$}\RightLabel{\existsoe}\BinaryInfC{$\Delta^\ast(\mathsf{c}/\mathsf{y})\vdash\delta^\ast(\mathsf{c}/\mathsf{y})$}\DisplayProof
\end{center}
\noindent where $\pi_1^\ast,\pi_2^\ast$ are, by induction hypothesis, the derivations that correspond to $\pi_1,\pi_2$ in \FOT respectively, and $d$ is a fresh constant symbol.

If the last step of the derivation of $\Delta\vdash_\FO\delta$ is an application of the $\forall\textsf{I}$ rule:
\begin{center}
\AxiomC{$~\vdots~\pi$}\noLine\UnaryInfC{$\Delta\vdash\alpha(v/x)$}\RightLabel{$\forall$\textsf{I}}\UnaryInfC{$\Delta\vdash\forall x\alpha$}\DisplayProof
\end{center}
\noindent where $ v\notin\textsf{Fv}(\Delta\cup\{\alpha\})$ and the free variables occurring in $\Delta,\forall x\alpha$ are among $\mathsf{y}$, then the corresponding derivation in \FOT is:
\begin{center}
\AxiomC{$~\vdots~\pi^\ast$}\noLine\UnaryInfC{$\Delta^\ast(\mathsf{c}/\mathsf{y})\vdash(\alpha(v/x))^\ast(d/v,\mathsf{c}/\mathsf{y})$}\RightLabel{\foralloi, \consi}\UnaryInfC{$\Delta^\ast(\mathsf{c}/\mathsf{y})\vdash\forallo x\alpha^\ast(x,\mathsf{c}/\mathsf{y})$}\DisplayProof
\end{center}
where $d$ is a fresh constant symbol.

If the last step of the derivation of $\Delta\vdash_\FO\delta$ is an application of the $\forall\textsf{E}$ rule:
\begin{center}
\AxiomC{$~\vdots~\pi$}\noLine\UnaryInfC{$\Delta\vdash\forall x\alpha$}\RightLabel{$\forall$\textsf{E}}\UnaryInfC{$\Delta\vdash\alpha(t/x)$}\DisplayProof
\end{center}
\noindent where the free variables occurring in $\Delta,\forall x\alpha$ are among $\mathsf{y}$ and $t=t(\mathsf{vy},\mathsf{d})$ for some sequence $\mathsf{d}$ of constant symbols, then the corresponding derivation in \FOT is:
\begin{center}
\AxiomC{$~\vdots~\pi^\ast$}\noLine\UnaryInfC{$\Delta^\ast(\mathsf{c}/\mathsf{y})\vdash\forallo x\alpha^\ast(x,\mathsf{c}/\mathsf{y})$}
\AxiomC{}\RightLabel{\consi}\UnaryInfC{$\Delta^\ast(\mathsf{c}/\mathsf{y})\vdash\con(\mathsf{c'cd})$}\RightLabel{\consi}\UnaryInfC{$\Delta^\ast(\mathsf{c}/\mathsf{y})\vdash\dep(t(\mathsf{\mathsf{c}'/\mathsf{v},\mathsf{c}/\mathsf{y},d}))$}\RightLabel{\foralloe}\BinaryInfC{$\Delta^\ast(\mathsf{c}/\mathsf{y})\vdash\alpha^\ast(t(\mathsf{c}'/\mathsf{v},\mathsf{c}/\mathsf{y},\mathsf{d})/x,\mathsf{c}/\mathsf{y})$}\DisplayProof
\end{center}
\end{proof}

Finally, we are in a position to prove the completeness theorem of our system. {We first prove the weak completeness, and then derive the strong completeness as a consequence of the weak one and the compactness theorem (Theorem \ref{compact}) we proved in Section \ref{sec:fot-fotd-expr}. Finding a direct proof of the strong completeness is left for future work.}

\begin{theorem}[Weak completeness]\label{completeness_fot}
$\models\phi\iff\vdash_\FOT\phi$. 
\end{theorem}
\begin{proof}
Suppose that $\nvdash_\FOT\phi$. 
By  \raa we derive $\wcn\phi\nvdash_\FOT\bot$, which is equivalent to $\psi(\mathsf{v})\nvdash_\FOT\bot$, where
$\psi(\mathsf{v})=Q^1\mathsf{x}\theta(\mathsf{x},\mathsf{v})$ is the formula $\wcn\phi$ in the normal form given by Lemma \ref{FOT_NF_prf}.
Let $R$ be a fresh relation symbol (which is assumed to be always available), and 
let $\psi(R)$ be the inclusion atom-free sentence obtained from $\psi(\mathsf{v})$ by replacing every inclusion atom $\mathsf{x}_i\subseteq \mathsf{v}$ by $R\mathsf{x}_i$. 
By  \incwi we have $\existso \mathsf{z}R\mathsf{z}, \psi(R)\nvdash_\FOT\bot$.  
%
Consider the \FO-formula  $\psi_\ast(R)$ obtained from the \FOT-formula $\psi(R)$ by replacing every logical constant in \FOT by its counterpart in \FO. It follows from  Lemma \ref{FO2FOT_eqiv} that $\exists \mathsf{z}R\mathsf{z},\psi_\ast(R)\nvdash_\FO\bot$.
Now, by the completeness theorem of \FO, there exists a model $(M,R^M)$ such that $R^M\neq\emptyset$ and $(M,R^M)\models\psi_\ast(R)$. It then follows from (the proof of) \Cref{fo2fot} that $M\models_{X_R}\psi(\mathsf{v})$, where 
\(X_R=\{s:\{v_1,\dots,v_k\}\to M\mid s(\mathsf{v})\in R^M\}\)
is the (nonempty) team associated with $R$. Hence $\wcn\phi\not\models\bot$, thereby $\not\models\phi$.
\end{proof}

\begin{cor}[Strong completeness]\label{completeness_fot_str}
For any set $\Gamma\cup\{\phi\}$ of \FOT-formulas that contains countably many free variables, we have that
 $\Gamma\models\phi\iff\Gamma\vdash_\FOT\phi$. 
\end{cor}
\begin{proof}
Suppose $\Gamma\models\phi$. Since $\Gamma,\wcn\phi\models\bot$ and $\Gamma\cup\{\phi\}$ contains countably many free variables, by compactness of \FOT (\Cref{compact}) we may assume that $\Gamma$ is a finite set. Now, we have that $\models\wcn\bigwedge\Gamma\vvee\phi$. Thus, by the weak completeness, we obtain $\vdash_{\FOT}\wcn\bigwedge\Gamma\vvee\phi$, which, by \vveee and $\wcn\mathsf{E}$, implies $\Gamma\vdash_\FOT\phi$.
\end{proof}


\section{Partial axiomatization of $\D\oplus\FOTd$}\label{sec:fotd}






Having axiomatized completely the logic \FOT, in this section we turn to the weaker logic \FOTd that captures downward closed elementary team properties (Corollary \ref{fotd_fodw}). 
Recall that the stronger downward closed logic {\em dependence logic} (\D) is the extension of first-order logic with dependence atoms. More precisely, \D-formulas are defined as:
\[\phi::=\lambda\mid\dep(\mathsf{t},\mathsf{t}')\mid\neg\alpha \mid  \phi\wedge\phi\mid\phi\vee\phi\mid \exists x \phi\mid \forall x\phi\]
where $\lambda$ is a first-order atom, $\mathsf{t},\mathsf{t}'$ are sequences of terms, and $\alpha$ is a first-order formula\footnote{In this syntax we only allow negation $\neg$  to occur  in front of  first-order formulas $\alpha$. Such treatment of negation is slightly more general than that in most literature on dependence logic, where formulas are often assumed to be in negation normal form, and negation $\neg$ is thus allowed only in front of first-order atomic formulas.
Since a first-order negated formula $\neg\alpha$ can always be transformed into negation normal form, this difference is not essential. Our choice of the syntax actually gives rise to more general and standard deduction rules for negation.
 }.  We now consider an extension  $\D\oplus\FOTd$ of \D obtained by enriching the syntax of \D with  the weak  connectives  and quantifiers $\vvee,\existso,\forallo$ from \FOTd:
 \[\phi::=\lambda\mid\dep(\mathsf{t},\mathsf{t}')\mid\neg\alpha \mid  \phi\wedge\phi\mid\phi\vee\phi\mid\phi\vvee\phi\mid \exists x \phi\mid \forall x\phi\mid\existso x\phi\mid\forallo x\phi,\]
where $\lambda$, $\mathsf{t},\mathsf{t}'$, and  $\alpha$ are as before (in particular, a first-order formula $\alpha$ cannot contain $\vvee,\existso$, or $\forallo$). Since \D is known to already  capture all downward closed existential second-order team properties \cite{Van07dl,KontVan09}, and the logical constants $\vvee,\existso,\forallo$ preserve downward closure and have essentially only first-order expressive power (as shown in Theorem \ref{fot2fo}), the logic $\D\oplus\FOTd$ with the enriched syntax is actually expressively equivalent to \D. We give a sketch of this proof below.
 \begin{theorem}\label{D=D+FOTd}
For every formula $\phi$ in $\D\oplus\FOTd$, there is a formula $\psi$ in $\D$ such that $\phi\equiv\psi$; and vice versa. In other words, $\D\oplus\FOTd$ is equi-expressive with \D.
\end{theorem}
\begin{proof}
The ``vice versa" direction is trivial. For the non-trivial direction, we show that for every $\D\oplus\FOTd$-formula $\phi(v_1,\dots,v_k)$ in some vocabulary $\mathcal{L}$, there exists an $\mathcal{L}(R)$-sentence $\gamma_\phi(R)$ with a fresh $k$-ary relation symbol $R$ occurring only negatively such that equivalence (\ref{fot2fo_eq}) from Theorem \ref{fot2fo} holds for any $\mathcal{L}$-model $M$ and nonempty team $X$. This would then suffice for the following reasons: We know by \cite[Theorem 4.9]{KontVan09} that for every such existential second-order sentence $\gamma_\phi(R)$, there exists an $\mathcal{L}$-formula $\psi(v_1,\dots,v_k)$ in $\D$ such that equivalence (\ref{fot2fo_eq}) from Theorem \ref{fot2fo} with respect to $\psi$ holds for any $\mathcal{L}$-model $M$ and nonempty team $X$. This then implies that $\phi$ and $\psi$ are equivalent over arbitrary models and nonempty teams. Over the empty team, on the other hand, both $\phi$ and $\psi$ are satisfied (by the empty team property of the logics). Putting all these together, we would be able to conclude that $\phi\equiv\psi$.

Now, to give the translation from $\D\oplus\FOTd$ into existential second-order logic, we combine our translation from $\FOTd$ into first-order logic as given in Theorem \ref{fot2fo}, and the translation from $\D$ into existential second-order logic as given in \cite[Theorem 6.2]{Van07dl}. 
 If  the $\D\oplus\FOTd$-formula $\phi(\mathsf{v})$ is a dependence atom, or if its main connective or quantifier is one from $\D$,  we define $\gamma_\phi(R)$ in the same way as in \cite[Theorem 6.2]{Van07dl}. E.g., $\gamma_{\theta_0\vee\theta_1}(R,\mathsf{x})=\exists S_1\exists S_2(\forall \mathsf{u}(R\mathsf{u}\to S_1\mathsf{u}\vee S_2\mathsf{u})\wedge \gamma_{\theta_0}(S_1,\mathsf{x})\wedge \gamma_{\theta_1}(S_2,\mathsf{x}))$, etc. If $\phi(\mathsf{v})$ is a first-order (negated) formula, or if its main connective or quantifier is $\vvee,\forallo,$ or $\existso$, 
 we define $\gamma_\phi(R)$ in the same way as in Theorem \ref{fot2fo}. The formula $\gamma_\phi(R)$ is clearly existential second-order  and it also satisfies equivalence (\ref{fot2fo_eq}) from Theorem \ref{fot2fo}.
\end{proof}

%

Denote by \FOp the extension of \FO by the weak universal quantifier $\forallo$. We say that a $\D\oplus\FOTd$-formula $\theta$ is {\em pseudo-flat}  if the universal quantifiers $\forall$ in $\theta$ quantify over \FOp-formulas only. For instance, the formula $\exists x(\dep(x)\vvee \forall y\forallo z\alpha)$ is pseudo-flat, and all \FOp-formulas are clearly pseudo-flat. A typical \D-formula $\forall x\exists y(\dep(x,y)\wedge\phi)$ is not pseudo-flat, whereas $\forallo x\exists y(\dep(x,y)\wedge (x\neq y\vee y=0))$ is pseudo-flat. 

In this section, we introduce a system of natural deduction for the logic $\D\oplus\FOTd$ by extending the systems of \cite{Axiom_fo_d_KV} and \cite{Yang_neg} for \D so that this new system is sound and  complete for certain pseudo-flat consequences, in the sense that
\begin{equation}\label{eq_fotd_consequence}
\Gamma\models\theta\iff\Gamma\vdash\theta
\end{equation}
holds whenever $\Gamma$ is a set of $\D\oplus\FOTd$-formulas, and $\theta$ is an \FOp-formula or a pseudo-flat $\D\oplus\FOTd$-sentence.
Recall that the system of \D given in \cite{Axiom_fo_d_KV} is complete for first-order consequences in the sense that (\ref{eq_fotd_consequence}) holds for $\Gamma$ being a set of \D-sentences and $\theta$ being an \FO-sentence. In our completeness proof, we will generalize the argument in \cite{Axiom_fo_d_KV}, and also apply and generalize a trick developed in \cite{Yang_neg} which extends the completeness result in \cite{Axiom_fo_d_KV} to allow the formula $\theta$ to be in any language that 
is closed under the weak classical negation \wcn, for which the {\em reductio ad absurdum} (\textsf{RAA}) rule is sound.
In our argument, we make use of the fact that  \FOp formulas and pseudo-flat $\D\oplus\FOTd$-sentences also admit certain \textsf{RAA} rule (see Lemma \ref{cn_defin_in_FOTd}).
%
%
%
As a corollary of our result, we will also be able to conclude that the original system for \D as introduced in \cite{Axiom_fo_d_KV} is already complete with respect to first-order consequences over arbitrary formulas (under the assumption that the set of free variables occurring in the formulas in question is countable). This answers an open problem in \cite{Axiom_fo_d_KV} in the affirmative. 


Let us start by observing that \FOp-formulas are all flat.
\begin{fact}\label{fop_flat}
\FOp-formulas  are flat.
\end{fact}
\begin{proof}
Formulas $\phi$ of \FOp  are clearly downward closed and satisfy the empty team property. It is then sufficient to show, by induction, that $\phi$ is closed under unions. For the only non-trivial case $\phi=\forallo x\psi$, suppose that $M\models_{X_i}\forallo x\psi$ for any $i\in I\neq\emptyset$. For any $a\in M$, we have that $M\models_{X_i(a/x)}\psi$. Thus, by induction hypothesis, $M\models_{X(a/x)}\psi$ for $X=\bigcup_{i\in I}X_i$. Hence we conclude $M\models_{X}\forallo x\psi$.
\end{proof}


For any $\D\oplus\FOTd$-formula $\theta$, denote by $\theta_\ast$  the first-order formula obtained from $\theta$ by replacing every dependence atom $\dep(\mathsf{t},\mathsf{t}')$ by $\top$, every weak logical constant  by its counterpart in \FO, i.e., by replacing $\vvee$ by $\vee$, $\existso$ by $\exists$, and $\forallo$ by $\forall$. We now show that $\theta$ and $\neg\theta_\ast$ are contradictory to each other in case $\theta$ is a pseudo-flat sentence or a (possibly open) formula in \FOp (which is also pseudo-flat).

\begin{lem}\label{cn_defin_in_FOTd}
Let $\theta$ be a $\D\oplus\FOTd$-formula. 
\begin{enumerate}[label=(\roman*)]
\item If $\theta$ is pseudo-flat, then $M\models_{\{s\}}\theta\iff M\models_{\{s\}}\theta_\ast$.
\item\label{raa_star_sound}\label{neg2cn} If $\theta$ is a pseudo-flat sentence or a \FOp-formula, then 
\(\Gamma,\neg\theta_\ast\models\bot\iff\Gamma\models\theta.\)
\end{enumerate}
\end{lem}
\begin{proof}
(i) This is proved by an inductive argument, for which most cases follow essentially from Lemma \ref{fo=fot}. We only illustrate two non-trivial cases. If $\theta=\dep(\mathsf{t},\mathsf{t}')$, then both $M\models_{\{s\}}\dep(\mathsf{t},\mathsf{t}')$ and $M\models_{\{s\}}\top$ always hold. If $\theta=\forall x\eta$ where $\eta$ is in \FOp, then  $\eta$ is flat (by Fact \ref{fop_flat}), and the rest of the proof is essentially the same as the corresponding case in Lemma \ref{fo=fot}. 

%

(ii) Suppose $\Gamma\models\theta$, and suppose $M\models_X\Gamma$ and $M\models_X\neg\theta_\ast$. If $X\neq\emptyset$, then for an arbitrary $s\in X$, we have that $M\not\models_{\{s\}}\theta_\ast$, which implies $M\not\models_{\{s\}}\theta$ by item (i). Hence, by the downward closure of $\theta$, we conclude that $M\not\models_X\theta$, which contradicts the assumption that $\Gamma\models\theta$. Hence, $X=\emptyset$ and $M\models_X\bot$.

Conversely, suppose $\Gamma,\neg\theta_\ast\models\bot$, and suppose $M\models_X\Gamma$ and $M\not\models_X\theta$. We  derive a contradiction.  Let us first treat the case when $\theta$ is a pseudo-flat sentence with no free variable. By locality, $M\not\models_X\theta$ implies that $M\not\models_{\{\emptyset\}}\theta$ for the empty assignment $\emptyset$. Thus, by item (i), we obtain $M\models_{\{\emptyset\}}\neg\theta_\ast$. Observe that $\neg\theta_\ast$ is also a sentence with no free variable, thus $M\models_X\neg\theta_\ast$, by locality again. Now, since $M\models_X\Gamma$, by assumption we are forced to conclude that $M\models_X\bot$, meaning that $X=\emptyset$. But in view of the empty team property, this contradicts the assumption that $M\not\models_X\theta$. Hence we are done.

Next, let us consider the case when $\theta$ is a (possibly open) \FOp-formula. Since now $\theta$ is flat (by Fact \ref{fop_flat}), we have that $M\not\models_{\{s\}}\theta$ for some $s\in X$.
This means, by item (i), that $M\models_{\{s\}}\neg\theta_\ast$. Since formulas in $\Gamma$ are all downward closed, we then have that $M\models_{\{s\}}\Gamma$. But then, the assumption $\Gamma,\neg\theta_\ast\models\bot$ implies that $M\models_{\{s\}}\bot$, which is not the case. Hence we are also done.
\end{proof}

Before presenting our axiomatization, let us analyze the behavior of the weak quantifiers and connectives and their interaction with the usual logical constants and dependence atoms in the enriched language $\D\oplus\FOTd$. By Theorem \ref{D=D+FOTd}, the weak quantifiers and connectives can be eliminated in the logic, as, e.g., for every \D-formula $\phi$, $\forallo x\phi$ is expressible in $\D$. Furthermore, the weak existential quantifier $\existso$ and  disjunction $\vvee$ are actually both definable in  $\D\oplus\FOTd$ in terms of the constancy atoms and the standard existential quantifier and disjunction, under the assumption that every model has at least two elements (which we postulate throughout the paper). 
\begin{prop}\label{existso_vvee_definable}
\begin{enumerate}[label=(\roman*)]
\item $\existso x\phi\equiv\exists x(\dep(x)\wedge\phi)$.
\item $\phi\vvee\psi\equiv\exists x\exists y\big(\dep(x)\wedge \dep(y)\wedge \big((x=y\wedge \phi)\vee (x\neq y\wedge \psi)\big)\big)$, where $x,y$ are fresh. 
\end{enumerate}
\end{prop}
\begin{proof}
Item (i) is easy to prove. We only give detailed proof for item (ii). Suppose $M\models_X\phi\vvee\psi$. Then either $M\models_X\phi$ or $M\models_X\psi$. Pick distinct $a,b\in M$. Define  functions $F:X\to\wp^+(M)$ and $G:X(F/x)\to \wp^+(M)$ as $F(s)=\{a\}$, and $G(s)=\{a\}$ if $M\models_X\phi$ and $G(s)=\{b\}$ if $M\models_X\psi$.
Put $Y=X(F/x)(G/y)$. It is easy to see that $M\models_Y\dep(x)\wedge\dep(y)$, and $M\models_Yx=y\wedge\phi$ in case $M\models_X\phi$, and $M\models_Yx\neq y\wedge\psi$ in case $M\models_X\psi$.

Conversely, suppose there are suitable functions $F,G$ such that for  the team $Y=X(F/x)(G/y)$, $M\models_Y\dep(x)\wedge\dep(y)$ and $M\models_Y(x=y\wedge \phi)\vee (x\neq y\wedge \psi)$. Since $x,y$ have constant values in $Y$, it must be that either $M\models_Yx=y\wedge \phi$ or $M\models_Yx\neq y\wedge \psi$. Hence we conclude that $M\models_Y\phi\vvee\psi$. Since $x,y\notin \textsf{Fv}(\phi)\cup\textsf{Fv}(\psi)$, we have by the locality property that $M\models_X\phi\vvee\psi$.
\end{proof}




The weak universal quantifier $\forallo$, however, cannot be defined uniformly in terms of the other logical constants in $\D\oplus\FOTd$, for otherwise (by Proposition \ref{existso_vvee_definable}) it would also be  uniformly definable in \D, which contradicts the results in  \cite{Pietro_uniform}.
We characterize the interaction between  $\forallo$ and the standard quantifiers in the next lemma. 
\begin{lem}\label{wsqf_int}{For any two distinct variables $x$ and $y$,} we have that
$\exists x\forallo y\phi(x,y,\mathsf{z})\equiv\forallo y\exists x(\dep(\mathsf{z},x)\wedge \phi)$ and $\forall x\forallo y\phi\equiv\forallo y\forall x\phi$.
\end{lem}
\begin{proof}
The second equivalence is clear. We only verify the first equivalence. Suppose $M\models_X\exists x\forallo y\phi(x,y,\mathsf{z})$, and we may w.l.o.g. assume also that $\textsf{dom}(X)=\mathsf{z}$. Then there exists a function $F:X\to \wp^+(M)$ such that for all $a\in M$, $M\models_{X(F/x)(a/y)}\phi(x,y,\mathsf{z})$. Since $\phi$ is downward closed, we may w.l.o.g. assume that  $F(s)$ is a singleton for each $s\in X$. Let $a\in M$ be arbitrary. We define a function $F':X(a/y)\to \wp^+(M)$ as
\(F'(s)=F(s\upharpoonright \textsf{dom}(X)).\)
Clearly, $M\models_{X(a/y)(F'/x)}\dep(\mathsf{z},x)$. Observe that $X(a/y)(F'/x)=X(F/x)(a/y)$. Thus $M\models_{X(a/y)(F'/x)}\phi$, as required.

Conversely, suppose $M\models_X\forallo y\exists x(\dep(\mathsf{z},x)\wedge \phi)$. Pick $a\in M$. By assumption, there exists a function $F_a:X(a/y)\to \wp^+(M)$ such that $M\models_{X(a/y)(F_a/x)}\dep(\mathsf{z},x)\wedge \phi$. Define a function $F':X\to \wp^+(M)$ as
\(F'(s)=F_a(s(a/y))\). For any $b\in M$, it suffices to show that $M\models_{X(F'/x)(b/y)}\phi$. Now, since $M\models_{X(b/y)(F_b/x)}\dep(\mathsf{z},x)$, we have that for all $s_0\in X$, $F_b(s_0(b/y))=F_a(s_0(a/y))$ is a singleton. We claim that $X(b/y)(F_b/x)=X(F'/x)(b/y)$, which would then imply $M\models_{X(F'/x)(b/y)}\phi$, since we have $M\models_{X(b/y)(F_b/x)}\phi$ by assumption. Indeed, if $s\in X(F'/x)(b/y)$, then 
\[s(x)=F'(s\upharpoonright \textsf{dom}(X))=F_a((s\upharpoonright \textsf{dom}(X)(a/y))=F_b((s\upharpoonright \textsf{dom}(X)(b/y)),\]
which implies that $s\in X(b/y)(F_b/x)$. The converse direction is symmetric.
%
\end{proof}

Let us now define the system for $\D\oplus\FOTd$.

\begin{defn}
The system of natural deduction for $\D\oplus\FOTd$ consists of all rules of the system of \D defined in \cite{Axiom_fo_d_KV} (including particularly those rules in \Cref{rules_d}), 
all rules in  \Cref{rules_weak_cnt_qtf} from \Cref{sec:FOT} where dependence atoms $\dep(t)$ are now read as atomic formulas (instead of shorthands),  and the rules in \Cref{rules_d_fotd},
where   letters  in sans-serif face stand for sequences of variables,  and $\alpha$ ranges over first-order formulas only. 
\end{defn}

\begin{table}[t]
\centering
\caption{Some rules (adapted) from the system  \cite{Axiom_fo_d_KV} of \D\vspace{-4pt}}\setlength{\tabcolsep}{6pt}
\begin{tabular}{|C{0.46\linewidth}C{0.46\linewidth}|}
\hline
\multicolumn{2}{|c|}{\AxiomC{}\noLine\UnaryInfC{$\Gamma,\alpha\vdash\bot$} \RightLabel{$\neg$\textsf{I} }\UnaryInfC{$\Gamma\vdash\neg\alpha$} \DisplayProof
\quad\quad\AxiomC{}\noLine\UnaryInfC{$\Gamma\vdash\alpha$} \AxiomC{}\noLine\UnaryInfC{$\Gamma\vdash\neg\alpha$} \RightLabel{$\neg\mathsf{E}$ }\BinaryInfC{$\Gamma\vdash\phi$} \DisplayProof\quad\quad\AxiomC{}\noLine\UnaryInfC{$\Gamma,\neg\alpha\vdash\bot$} \RightLabel{\textsf{RAA} }\UnaryInfC{$\Gamma\vdash\alpha$} \DisplayProof}\\
\AxiomC{}\noLine\UnaryInfC{$\Gamma\vdash\phi$}\RightLabel{$\vee$\textsf{I}}\UnaryInfC{$\Gamma\vdash\phi\vee\psi$}\noLine\UnaryInfC{}\DisplayProof\quad\quad\AxiomC{}\noLine\UnaryInfC{$\Gamma\vdash\phi$}\RightLabel{$\vee$\textsf{I}}\UnaryInfC{$\Gamma\vdash\psi\vee\phi$}\noLine\UnaryInfC{}\DisplayProof &\AxiomC{}\noLine\UnaryInfC{$\Gamma\vdash\phi\vee\psi$}\AxiomC{}\noLine\UnaryInfC{$\Gamma,\phi\vdash\chi$}\RightLabel{$\vee$\textsf{Mon}}\BinaryInfC{$\Gamma\vdash\chi\vee\psi$}\noLine\UnaryInfC{}\DisplayProof
\\
\AxiomC{}\noLine\UnaryInfC{$\Gamma\vdash\phi$}\RightLabel{$\forall$\textsf{I} {\footnotesize(a)}}\UnaryInfC{$\Gamma\vdash\forall x\phi$}\noLine\UnaryInfC{}\DisplayProof&\AxiomC{}\noLine\UnaryInfC{$\Gamma\vdash\phi(t/x)$}\RightLabel{$\exists$\textsf{I}}\UnaryInfC{$\Gamma\vdash\exists x\phi$}\noLine\UnaryInfC{}\DisplayProof
\\
\AxiomC{}\noLine\UnaryInfC{$\Gamma\vdash\forall x\phi$}\RightLabel{$\forall$\textsf{E}}\UnaryInfC{$\Gamma\vdash\phi(t/x)$}\noLine\UnaryInfC{}\DisplayProof&\AxiomC{$\Gamma\vdash\exists x\phi$} 
\AxiomC{}\noLine\UnaryInfC{$\Gamma,\phi(v/x)\vdash\psi$}
\RightLabel{$\exists$\textsf{E} {\footnotesize(b)}}\BinaryInfC{$\Gamma\vdash\psi$}\noLine\UnaryInfC{}\DisplayProof \\%
\multicolumn{2}{|c|}{\def\ScoreOverhang{0.5pt}\AxiomC{$\Gamma\vdash\exists \mathsf{x}\forall y\phi(\mathsf{x},y,\mathsf{z})$}\RightLabel{$\mathsf{DepI}$}\UnaryInfC{$\Gamma\vdash\forall y\exists \mathsf{x}(\dep(\mathsf{z},\mathsf{x})\wedge\phi)$}\DisplayProof}\\
\multicolumn{2}{|c|}{\AxiomC{}\noLine\UnaryInfC{\quad\quad\quad$\displaystyle\Gamma\vdash\forall \mathsf{x}\exists\mathsf{y}\big(\bigwedge_{i\in I}\dep(\sigma^i_{\mathsf{xyv}},y_i)\,\wedge\phi(\mathsf{x},\mathsf{y},\mathsf{v})\big)$\quad\quad\quad}\RightLabel{$\mathsf{DepE}$ (c)}\UnaryInfC{$\displaystyle\Gamma\vdash\forall \mathsf{x}\exists\mathsf{y}\Big(\phi(\mathsf{x},\mathsf{y},\mathsf{v})\wedge\forall \mathsf{x}'\exists\mathsf{y}'\big(\phi(\mathsf{x}',\mathsf{y}',\mathsf{v})\wedge\bigwedge_{i\in I}(\sigma^i_{\mathsf{xyv}}=\sigma^i_{\mathsf{x}'\mathsf{y}'\mathsf{v}}\to y_i=y_i'\big)\Big)$}\DisplayProof}\\ 
\multicolumn{2}{|l|}{\footnotesize (a) $x\notin\textsf{Fv}(\Gamma)$\quad\quad(b) $v\notin\textsf{Fv}(\Gamma\cup\{\phi,\psi\})$  \quad\quad(c) {each $\sigma^i_{\mathsf{xyv}}$ is a subsequence of $\mathsf{xyv}$.}}\\\hline
\end{tabular}
\label{rules_d}
\end{table}%
\begin{table}[t]
\centering
\caption{Additional rules for $\D\oplus\FOTd$}
\setlength{\tabcolsep}{6pt}
\begin{tabular}{|C{0.46\linewidth}C{0.46\linewidth}|}
\hline
\AxiomC{}\noLine\UnaryInfC{}\noLine\UnaryInfC{}\noLine\UnaryInfC{} \RightLabel{\textsf{Dom}}\UnaryInfC{$\Gamma\vdash\exists x\exists y(x\neq y)$}\noLine\UnaryInfC{}\DisplayProof
&
\AxiomC{$\Gamma,\neg\theta_\ast\vdash\bot$}\RightLabel{$\textsf{RAA}_\ast$ {\footnotesize (a)}}\UnaryInfC{$\Gamma\vdash\theta$}\DisplayProof\\
\AxiomC{$\Gamma\vdash\phi\vee\bot$} \RightLabel{$\bot\vee$\textsf{E}}\UnaryInfC{$\Gamma\vdash\phi$}\noLine\UnaryInfC{}\DisplayProof
&\AxiomC{$\Gamma\vdash\forallo x\alpha$} \RightLabel{$\forallo\forall$\textsf{Trs}}\UnaryInfC{$\Gamma\vdash\forall x\alpha$}\noLine\UnaryInfC{}\DisplayProof\\
\multicolumn{2}{|c|}{\AxiomC{$\Gamma\vdash\forallo x\phi\vee\psi$}\RightLabel{$\forallo$\textsf{Ext} {\footnotesize (b)}}\doubleLine\UnaryInfC{$\Gamma\vdash \exists yz\forallo x((y=z\wedge \phi)\vee(y\neq z\wedge\psi))$}\noLine\UnaryInfC{}\noLine\UnaryInfC{}\DisplayProof }
\\
\AxiomC{$\Gamma,\con(\mathsf{x})\vdash\dep(\mathsf{y})$}
\RightLabel{$\mathsf{DepWI}$}\UnaryInfC{$\Gamma\vdash\dep( \mathsf{x},\mathsf{y})$}\noLine\UnaryInfC{}\noLine\UnaryInfC{}\DisplayProof
&\AxiomC{$\Gamma\vdash\dep(\mathsf{x},\mathsf{y})$} \AxiomC{$\Gamma\vdash\con(\mathsf{x})$} \RightLabel{$\mathsf{DepWE}$}\BinaryInfC{$\Gamma\vdash\dep(\mathsf{y})$}\noLine\UnaryInfC{}\noLine\UnaryInfC{}\DisplayProof\\
\AxiomC{$\Gamma\vdash\dep(\mathsf{x})$}\AxiomC{$\Gamma\vdash\dep(\mathsf{y})$}
\RightLabel{$\mathsf{conExt}$}\BinaryInfC{$\Gamma\vdash\dep(\mathsf{\mathsf{xy}})$}\noLine\UnaryInfC{}\DisplayProof
&\AxiomC{$\Gamma\vdash\dep(\mathsf{xyz})$}\RightLabel{$\mathsf{conW}$}\UnaryInfC{$\Gamma\vdash\dep(\mathsf{y})$}\noLine\UnaryInfC{}\DisplayProof\\
\multicolumn{2}{|c|}{\AxiomC{$\Gamma,\alpha(\mathsf{x})\vdash\phi$}\AxiomC{$\Gamma,\neg\alpha(\mathsf{x})\vdash\phi$}
\RightLabel{$\mathsf{conTrs}$}\BinaryInfC{$\Gamma,\con(\mathsf{x})\vdash\phi$}\noLine\UnaryInfC{}\noLine\UnaryInfC{}\DisplayProof}\\
\multicolumn{2}{|l|}{\footnotesize  
 (a) $\theta$ is an \FOp-formula or a pseudo-flat sentence. \quad(b) $x\notin\textsf{Fv}(\psi)$, $y,z$ are fresh.}\\\hline
\end{tabular}
\label{rules_d_fotd}
\end{table}

The axiom \textsf{Dom}  stipulates that the domain of a model has at least two elements, which we assume throughout the paper (and especially in this section). This domain assumption is often postulated in the literature for first-order models, as models with singleton domains is considered trivial  in many respects. For dependence logic in particular, over models with single elements in the domain, all dependence atoms become trivially true, since for a fixed set of variables, there is only one  assignment over such domains. In the present section, the axiom \textsf{Dom} or the corresponding domain assumption is also required for the  fact that the weak disjunction $\vvee$ is definable in terms of  $\vee$ in \D (see Proposition \ref{existso_vvee_definable}(ii) and Proposition \ref{FOTd_der_rules}\ref{vvee_def_prop}), and the soundness of the rule $\forallo$\textsf{Ext}. {Finding a deduction system for the same logic $\D\oplus\FOTd$ without this domain assumption is left as future work. }
The standard introduction rule $\neg$\textsf{I} and elimination rule $\neg$\textsf{E}  for negation apply only to negation in front of first-order formulas.  The rule $\textsf{RAA}_\ast$ is crucial for our completeness proof when applying a trick  developed in \cite{Yang_neg}, where a similar rule  was formulated originally in terms of the weak classical negation $\wcn$ instead. The soundness of $\textsf{RAA}_\ast$ follows from Lemma \ref{cn_defin_in_FOTd}\ref{raa_star_sound}.
The falsum-disjunction elimination rule $\bot\vee$\textsf{E} 
and the universal quantifiers transition rule $\forallo\forall$\textsf{Trs}  are self-explanatory. 
The invertible 
rule $\forallo$\textsf{Ext} is an adaption of a similar rule in the system of \D in \cite{Axiom_fo_d_KV}. It is  also inspired by a similar equivalence given in \cite{Hierarchies_Ind_GHK}. 
The weak introduction $\mathsf{DepWI}$ and elimination rule $\mathsf{DepWE}$ for dependence atoms were first introduced in \cite{VY_PD} in the context of propositional dependence logic. These two rules characterize the equivalence $\dep(\mathsf{x},\mathsf{y})\equiv \dep(\mathsf{x})\to\dep(\mathsf{y})$, which was first identified in \cite{AbVan09}, and later discussed also in the context of inquisitive logic in \cite{Ciardelli2015}. 
The constancy atom extension and weakening rules $\mathsf{conExt}$ and $\mathsf{conW}$ characterize the fact that $\dep(\mathsf{xy})\equiv\dep(\mathsf{x})\wedge\dep(\mathsf{y})$. These two rules for dependence atoms were not included in the original system of \D in \cite{Axiom_fo_d_KV}, where dependence atoms $\dep(\mathsf{t},t')$ have in the second component only single terms $t'$ instead of sequences of terms. 
The constancy atom transition rule $\mathsf{conTrs}$ characterizes the entailment $\con(\mathsf{x})\models\alpha(\mathsf{x})\vvee\neg\alpha(\mathsf{x})$. We will show in Proposition \ref{FOTd_der_rules}\ref{vvee_wi} below that this semantic entailment is also syntactically derivable in the system.



\begin{theorem}[Soundness]\label{soundness}
$\Gamma\vdash\phi\Longrightarrow\Gamma\models\phi$.
\end{theorem}
\begin{proof}
We only give detailed proof for the soundness of the nontrivial rules.


$\forallo\forall$\textsf{Trs}: It suffices to show that $\forallo x\alpha\models\forall x\alpha$. Suppose $M\models_X\forallo x\alpha$. Then for all $a\in M$, $M\models_{X(a/x)}\alpha$. Observe that $X(M/x)=\bigcup_{a\in M}X(a/x)$. Since $\alpha$ is closed under unions, we then conclude that $M\models_{X(M/x)}\alpha$, thus $M\models_X\forall x\alpha$.

$\forallo$\textsf{Ext}: It suffices to show that $\forallo x\phi\vee\psi\equiv\exists yz\forallo x((y=z\wedge \phi)\vee(y\neq z\wedge\psi))$.
For the left to right direction, suppose $M\models_X\forallo x\phi\vee\psi$, and we may w.l.o.g. also assume that $x,y,z\notin \textsf{dom}(X)$. Then there exist $Y,Z\subseteq X$ such that $X=Y\cup Z$, $M\models_Y\forallo x\phi$ and $M\models_Z\psi$. Since $\forallo x\phi$ and $\psi$ are downward closed, we may w.l.o.g. assume that $Y$ and $Z$ are disjoint. 
Let $a,b$ be two distinct elements in $M$ (which is assumed to have at least two elements). Define $F:X\to \wp^+(M)$ as $F(s)=\{a\}$, and define $G:X(F/y)\to \wp^+(M)$ by taking 
$G(s)=\{a\}$ if $s\upharpoonright \textsf{dom}(X)\in Y$, and $G(s)=\{b\}$ if $s\upharpoonright \textsf{dom}(X)\in Z$.
Putting $X'=X(F/y)(G/z)$ we show that $M\models_{X'(c/x)}(y=z\wedge \phi)\vee(y\neq z\wedge\psi)$ for arbitrary $c\in M$.
 Define $Y'=\{s\in X'(c/x)\mid s(z)=a\}$ and $Z'=X'(c/x)\setminus Y'$. Clearly, $Y'\cup Z'=X'(c/x)$, $M\models_{Y'}y=z$ and $M\models_{Z'}y\neq z$. Since $M\models_Z\psi$ and $x,y,z\notin \textsf{dom}(Z)$, we have $M\models_{Z'}\psi$. Also, since $M\models_Y\forallo x\phi$, we have $M\models_{Y(c/x)}\phi$. Since $\textsf{dom}(X)\cup\{x\}\subseteq \textsf{Fv}(\phi)$ and $Y(c/x)\upharpoonright (\textsf{dom}(X)\cup\{x\})=Y'\upharpoonright (\textsf{dom}(X)\cup\{x\})$, we conclude that  $M\models_{Y'}\phi$.

For the right to left direction,  suppose $M\models_X\exists yz\forallo x((y=z\wedge \phi)\vee(y\neq z\wedge\psi))$ and $x,y,z\not\in\textsf{dom}(X)$. Then there exist appropriate functions $F,G$ such that for any $a\in M$, there exists $Y_a\subseteq X(F/y)(G/z)(a/x)=X'(a/x)$ such that $M\models_{Y_a}y=z\wedge \phi$ and $M\models_{X'(a/x)\setminus Y_a}y\neq z\wedge \psi$. 
We claim that for any $a,b\in M$, $Y_a\upharpoonright \textsf{dom}(X)=Y_b\upharpoonright \textsf{dom}(X)$. Indeed, for any $s\in Y_a\upharpoonright \textsf{dom}(X)$, there exists an extension $\hat{s}$ of $s$ in $Y_a$. Since $M\models_{Y_a}y=z$, we have $\hat{s}(y)=\hat{s}(z)$. Let $t=\hat{s}(b/x)\in X'(b/x)\supseteq Y_b$. Since $t(y)=\hat{s}(y)=\hat{s}(z)=t(z)$ and $M\models_{X'(b/x)\setminus Y_b}y\neq z$, we must have that $t\in Y_b$, and thus $s=\hat{s}\upharpoonright \textsf{dom}(X)=t\upharpoonright \textsf{dom}(X)\in Y_b\upharpoonright \textsf{dom}(X)$. This shows that $Y_a\upharpoonright \textsf{dom}(X)\subseteq Y_b\upharpoonright \textsf{dom}(X)$. The other inclusion is proved similarly.

 %
 
 Now, to show $M\models_X\forallo x\phi\vee\psi$, let $Y=Y_a\upharpoonright \textsf{dom}(X)$ and $Z=X\setminus Y$ for any $a\in M$. Since $M\models_{X'(a/x)\setminus Y_a}\psi$ and 
 \(Z=X\setminus (Y_a\upharpoonright \textsf{dom}(X))=(X'(a/x)\setminus Y_a)\upharpoonright \textsf{dom}(X),\)
 we obtain $M\models_Z\psi$ by locality. Meanwhile, for any $b\in M$,  we have $Y=Y_a\upharpoonright \textsf{dom}(X)=Y_b\upharpoonright \textsf{dom}(X)$. Since $M\models_{Y_b}\phi$ and $Y_b\upharpoonright (\textsf{dom}(X)\cup\{x\})=Y(b/x)$, we obtain $M\models_{Y(b/x)}\phi$.

$\mathsf{DepWI}$: Suppose $\Gamma,\dep(\mathsf{x})\models\dep(\mathsf{y})$ and suppose $M\models_X\Gamma$. For any $s,s'\in X$, if $s(\mathsf{x})=s'(\mathsf{x})$, then $M\models_{\{s,s'\}}\dep(\mathsf{x})$. Since $\Gamma$ is downward closed, we have that $M\models_{\{s,s'\}}\Gamma$. Thus, by assumption we obtain $M\models_{\{s,s'\}}\dep(\mathsf{y})$, hence $s(\mathsf{y})=s'(\mathsf{y})$.
\end{proof}

In the following proposition we list some derivable clauses that  will be used in our  proof of the completeness theorem. 
In addition, the clauses in Proposition \ref{FOT_der_rules}\ref{FOTd_der_rules_forallo_existso_con}\ref{FOTd_der_rules_qf_conj_disj} are still derivable in the system of $\D\oplus\FOTd$ by the same derivations (where $\vdash\dep(c)$ is given by item \ref{constancy_dep_lm_1} below).

\begin{prop}\label{FOTd_der_rules}
\begin{enumerate}[label=(\roman*)]
\item\label{term_dep_atm_eli} $\dep(\mathsf{t},\mathsf{t}')\dashv\vdash \exists \mathsf{x}\mathsf{y}(\dep(\mathsf{x},\mathsf{y})\wedge \mathsf{x}=\mathsf{t}\wedge \mathsf{y}=\mathsf{t}')$, where the variables in $\mathsf{x}\mathsf{y}$ do not occur in the sequences $\mathsf{t}, \mathsf{t}'$ of terms.
\item\label{constancy_dep_lm_decompose} $\dep(\mathsf{x},\mathsf{yz})\dashv\vdash\dep(\mathsf{x},\mathsf{y})\wedge\dep(\mathsf{x},\mathsf{z})$.
\item\label{constancy_dep_lm_1} $\vdash\dep(\mathsf{x},c)$, and in particular $\vdash\dep(c)$ for any constant symbol $c$.
\item\label{constancy_dep_lm_2} $\dep(c\mathsf{x},\mathsf{y})\dashv\vdash\dep(\mathsf{x},\mathsf{y})$  for any constant symbol $c$. 
\item\label{forall1_triv_dep_ind} $\forallo vQ\mathsf{u}(\phi\wedge\dep(v\mathsf{x},\mathsf{y}))\dashv\vdash\forallo vQ\mathsf{u}(\phi\wedge\dep(\mathsf{x},\mathsf{y}))$ and $\forallo v Q\mathsf{u}(\phi\wedge\dep(\mathsf{x},v))\dashv\vdash\forallo vQ\mathsf{u}\phi $.
\item\label{existso_def_prop} $\existso x\phi\dashv\vdash \exists x(\dep(x)\wedge \phi)$.
\item\label{vvee_wi} $\dep(\mathsf{x})\vdash\alpha(\mathsf{x})\vvee\neg\alpha(\mathsf{x})$.
\item\label{vvee_def_prop} $\phi\vvee\psi\dashv\vdash\exists x\exists y\big(\dep(x)\wedge \dep(y)\wedge ((x=y\wedge \phi)\vee (x\neq y\wedge \psi))\big)$, where $x,y\notin \mathsf{Fv}(\{\phi,\psi\})$.
\item\label{forallo_ex_swap} $\exists x\forallo y\phi(x,y,\mathsf{z})\vdash\forallo y\exists x(\dep(\mathsf{z},x)\wedge \phi)$ and $\forall x\forallo y\phi\dashv\vdash\forallo y\forall x\phi$.
\end{enumerate}
\end{prop}
\begin{proof}


Item \ref{term_dep_atm_eli}: A routine derivation using rules for identity and $\exists$.

Item \ref{constancy_dep_lm_decompose}: For the left to right direction, we only give the detailed derivation of $\dep(\mathsf{x},\mathsf{yz})\vdash\dep(\mathsf{x},\mathsf{y})$. First, we derive that $\dep(\mathsf{x},\mathsf{yz}),\dep(\mathsf{x})\vdash\dep(\mathsf{yz})\vdash\dep(\mathsf{y})$ by applying $\mathsf{DepWE}$ and $\mathsf{conW}$. Thus, $\dep(\mathsf{x},\mathsf{yz})\vdash\dep(\mathsf{x},\mathsf{y})$ follows by applying $\mathsf{DepWI}$.
 
 Conversely, we have by $\mathsf{DepWE}$ that $\dep(\mathsf{x},\mathsf{y}),\dep(\mathsf{x})\vdash \dep(\mathsf{y})$ and $\dep(\mathsf{x},\mathsf{z}),\dep(\mathsf{x})\vdash \dep(\mathsf{z})$, which then imply, by $\mathsf{conW}$, that $\dep(\mathsf{x},\mathsf{y}),\dep(\mathsf{x},\mathsf{z})\dep(\mathsf{x})\vdash\dep(\mathsf{yz})$. Hence, we conclude by applying $\mathsf{DepWI}$ that $\dep(\mathsf{x},\mathsf{y}),\dep(\mathsf{x},\mathsf{z})\vdash\dep(\mathsf{x},\mathsf{yz})$.

Item \ref{constancy_dep_lm_1}: By $=$\textsf{I} we have $\vdash c=c\wedge \mathsf{x}=\mathsf{x}$, which implies  $\vdash\exists u\forall v (u=c\wedge \mathsf{x}=\mathsf{x})$. Now, by \textsf{DepI} we derive $\vdash\forall v\exists u(\dep(\mathsf{x},u)\wedge u=c\wedge \mathsf{x}=\mathsf{x})$, which yields $\vdash\dep(\mathsf{x},c)$.


Item \ref{constancy_dep_lm_2}:  For the direction $\dep(c\mathsf{x},\mathsf{y})\vdash\dep(\mathsf{x},\mathsf{y})$, by  \textsf{DepWE} we have 
\(\dep(c\mathsf{x},\mathsf{y}),\dep(c),\con(\mathsf{x})\linebreak\vdash\dep(\mathsf{y}).\)
Since $\vdash\dep(c)$ by item (\ref{constancy_dep_lm_1}), we conclude  that $\dep(c\mathsf{x},\mathsf{y})\vdash\dep(\mathsf{x},\mathsf{y})$ by \textsf{DepWI}. {The other direction is derived similarly by applying $\mathsf{DepWE}$, $\mathsf{conW}$ and $\mathsf{DepWI}$.}
%


Item \ref{forall1_triv_dep_ind}: By Proposition \ref{FOT_der_rules}\ref{FOTd_der_rules_forallo_existso_con}, \foralloi it suffices to show $Q\mathsf{u}(\phi(c/v)\wedge\dep(c\mathsf{x},\mathsf{y}))\dashv\vdash Q\mathsf{u}\phi(c/v)\wedge\dep(\mathsf{x},\mathsf{y})$ and $Q\mathsf{u}(\phi(c/v)\wedge\dep(\mathsf{x},c))\dashv\vdash Q\mathsf{u}\phi(c/v)$. But these follow from items \ref{constancy_dep_lm_1} and \ref{constancy_dep_lm_2}. 

Item \ref{existso_def_prop}: The direction $\exists x(\dep(x)\wedge \phi)\vdash \existso x\phi$ follows easily from $\exists\textsf{E}$ and \existsoi. For the other direction, we first derive
\(\phi(c/x)\vdash\dep(c)\wedge\phi(c/x)\vdash\exists x(\dep(x)\wedge \phi)\) by \consi and $\exists\textsf{I}$, where $c$ is a fresh constant symbol. Then, since $\existso x\phi\vdash\existso x\phi$, we conclude $\existso x\phi\vdash\exists x(\dep(x)\wedge \phi)$ by applying \existsoe.



Item \ref{vvee_wi}: By $\vvee\mathsf{I}$, we have $\alpha(\mathsf{x})\vdash\alpha(\mathsf{x})\vvee\neg\alpha(\mathsf{x})$ and $\neg\alpha(\mathsf{x})\vdash\alpha(\mathsf{x})\vvee\neg\alpha(\mathsf{x})$. Thus, we obtain $\con(\mathsf{x})\vdash\alpha(\mathsf{x})\vvee\neg\alpha(\mathsf{x})$ by applying $\mathsf{conTrs}$.

Item \ref{vvee_def_prop}: We first prove the right to left direction. By $\exists\textsf{E}$ it suffices to prove that $\dep(x), \dep(y),(x=y\wedge \phi)\vee (x\neq y\wedge \psi)\vdash\phi\vvee\psi$. We first derive $\dep(x),\dep(y)\vdash x=y\vvee x\neq y$ by {$\mathsf{conExt}$} and  item \ref{vvee_wi}. Next, by applying $\vee\textsf{Mon}$, $\bot\vee\textsf{E}$ and \vveei, we derive 
\[x=y,(x=y\wedge \phi)\vee (x\neq y\wedge \psi)\vdash \phi\vee(x=y\wedge x\neq y)\vdash\phi\vee\bot\vdash\phi\vdash\phi\vvee\psi\]
and similarly $x\neq y,(x=y\wedge \phi)\vee (x\neq y\wedge \psi)\vdash\phi\vvee\psi$.
Hence,  we conclude by  \vveee that $x=y\vvee x\neq y,(x=y\wedge \phi)\vee (x\neq y\wedge \psi)\vdash \phi\vvee\psi$, from which the desired clause follows.

For the left to right direction, by {$\vvee\mathsf{E}$}, it suffices to prove that the right formula is derivable from both $\phi$ and $\psi$. 
We now first derive the right hand side from $\phi$. By the rules of identity,  $\vdash\exists x\forall z(x=x)$ for some fresh variables $x,z$. Thus, we conclude by applying $\mathsf{DepI}$ that $\vdash\forall z\exists x(\dep(x)\wedge x=x)$, which reduces to $\vdash\exists x\dep(x)$. Next, we derive by rules of identity that 
$\vdash \exists x(\dep(x)\wedge  \exists y(x=y))$ and thus $\vdash \exists x\exists y(\dep(x)\wedge\dep(y)\wedge x=y)$.
Lastly, we  conclude by  the introduction rule $\vee\mathsf{I}$ that 
\[\phi\vdash  \exists x\exists y(\dep(x)\!\wedge\!\dep(y)\wedge x=y\wedge\phi)\vdash \exists x\exists y(\dep(x)\!\wedge\!\dep(y)\wedge ((x=y\wedge\phi)\vee (x\neq y\wedge \psi))).\]

Similarly, to derive the right hand side from $\psi$, first note that by  \textsf{Dom} we have $\vdash\exists x\exists y(x\neq y)$ for some fresh variables $x,y$
, which then yields $\vdash\exists x\forall z\exists y\forall v(x\neq y)$ by rules for the quantifiers. Then, by a similar argument to the above, we derive by applying $\mathsf{DepI}$  that $\vdash\forall z\exists x(\dep(x)\wedge\exists y\forall v(x\neq y))$, which reduces to $\vdash \exists x\exists y\forall v(\dep(x)\wedge x\neq y)$.  We now apply $\mathsf{DepI}$ again to obtain $\vdash \exists x\forall v\exists y(\dep(y)\wedge \dep(x)\wedge x\neq y)$, which simplifies to $\vdash \exists x\exists y(\dep(x)\wedge \dep(y)\wedge x\neq y)$. Finally, again by $\vee\mathsf{I}$, we derive
\[\psi\vdash  \exists x\exists y(\dep(x)\!\wedge\!\dep(y)\wedge x\neq y\wedge\psi)\vdash \exists x\exists y(\dep(x)\!\wedge\!\dep(y)\wedge ((x=y\wedge\phi)\vee (x\neq y\wedge \psi))).\]


Item \ref{forallo_ex_swap}: $\forall x\forallo y\phi\dashv\vdash\forallo y\forall x\phi$ follows easily from Proposition \ref{FOT_der_rules}\ref{FOTd_der_rules_forallo_existso_con}. For the other clause, by $\exists\textsf{E}$, it suffices to prove  $\forallo y\phi(x,y,\mathsf{z})\vdash\forallo y\exists x(\dep(\mathsf{z},x)\wedge \phi)$. 
We derive by Proposition \ref{FOT_der_rules}\ref{FOTd_der_rules_forallo_existso_con} and the rules for $\forall$,$\exists$ that $\forallo y\phi(x,y,\mathsf{z})\vdash\phi(x,c,\mathsf{z})\vdash\exists x\forall w\phi(x,c,\mathsf{z})$
for some fresh constant symbol $c$ and  variable $w$. Moreover, by $\mathsf{DepI}$ we derive that
\[\exists x\forall w\phi(x,c,\mathsf{z})\vdash\forall w\exists x(\dep(\mathsf{z},x)\wedge \phi(x,c,\mathsf{z}))\vdash\exists x(\dep(\mathsf{z},x)\wedge \phi(x,c,\mathsf{z})).\]
Putting these together  and by applying $\forallo\textsf{I}$ we conclude that  
\(\forallo y\phi(x,y,\mathsf{z})\vdash\exists x(\dep(\mathsf{z},x)\wedge \phi(x,c,\mathsf{z}))\vdash\forallo y\exists x(\dep(\mathsf{z},x)\wedge \phi).\)
\end{proof}


Our proof of the completeness theorem for $\D\oplus\FOTd$ mainly follows the argument of that for the system of \D in \cite{Axiom_fo_d_KV}. In view of Proposition \ref{FOTd_der_rules}\ref{existso_def_prop}\ref{vvee_def_prop}, we will  treat  $\vvee$ and $\existso$ as defined logical constants. Moreover, by Proposition \ref{FOTd_der_rules}\ref{term_dep_atm_eli}\ref{constancy_dep_lm_decompose}, dependence atoms of the form $\dep(\mathsf{t},\mathsf{t}')$ will be interpreted as shorthands for the formula $\exists \mathsf{xy}(\bigwedge_{i}\dep(\mathsf{x},y_i)\wedge \mathsf{x}=\mathsf{t}\wedge \mathsf{y}=\mathsf{t}')$ in the original language of \D in \cite{Axiom_fo_d_KV}.  It thus remains to handle properly the weak universal quantifier $\forallo$ when going through the argument in \cite{Axiom_fo_d_KV}, which actually  gives the equivalence (\ref{eq_fotd_consequence}) for $\theta$ being a first-order sentence. In order to prove the equivalence (\ref{eq_fotd_consequence}) for $\theta$ being a \FOp-formula or pesudo-flat sentence, we will apply a trick developed in \cite{Yang_neg} that involves the rule $\textsf{RAA}_\ast$.

Now, recall from \cite{Van07dl} that every \D-formula $\phi(\mathsf{z})$ is semantically equivalent to, and provably implies a formula of the form
{\begin{equation}\label{nf_d}
\forall \mathsf{x}\exists \mathsf{y}\big(\bigwedge_{i\in I}\dep(\sigma_{\mathsf{xyz}}^i,y_i)\wedge\alpha(\mathsf{x},\mathsf{y},\mathsf{z})\big),
\end{equation}}
where each $\sigma_{\mathsf{xyz}}^i$ is a subsequence of $\mathsf{xyz}$,  each $y_i$ is from $\mathsf{y}$, and $\alpha$ is  first-order.
In the next theorem we derive a similar normal form for formulas in $\D\oplus\FOTd$. 
\begin{theorem}\label{FOTd_nf_thm}
Every $\D\oplus\FOTd$-formula $\phi(\mathsf{z})$ is semantically equivalent to, and provably implies a formula of the form
\begin{equation}\label{nf_dw}
{\forallo \mathsf{v}\forall \mathsf{x}\exists \mathsf{y}\big(\bigwedge_{i\in I}\dep(\sigma^i_{\mathsf{xyz}},y_i)\wedge\alpha(\mathsf{v},\mathsf{x},\mathsf{y},\mathsf{z})\big),}
\end{equation}
where each $\sigma_{\mathsf{xyz}}^i$ is a subsequence of $\mathsf{xyz}$,  each $y_i$ is from $\mathsf{y}$, and $\alpha$ is  first-order.
\end{theorem}
\begin{proof}
We adapt the argument for the corresponding proof in \cite[Proposition 9]{Axiom_fo_d_KV}. We give the semantic and syntactic proof at the same time, and the semantic equivalence clearly follows from the syntactic equivalence (by soundness theorem) whenever the latter is available in the following steps of the proof. First, rewrite every occurrence of $\dep(\mathsf{t},\mathsf{t}')$, $\theta\vvee\chi$ and $\existso x\eta$ in $\phi(\mathsf{z})$ using the equivalent formulas with dependence atoms and logical constants from \D given by Proposition \ref{FOTd_der_rules}\ref{term_dep_atm_eli}\ref{constancy_dep_lm_decompose}\ref{existso_def_prop}\ref{vvee_def_prop}, and denote the resulting provably equivalent formula 
by $\phi'$. Next, turn the formula $\phi'$ into an equivalent formula $Q_1v_1\dots Q_nv_n\theta$ in prenex normal form, where each $Q_i\in\{\forall,\exists ,\forallo \}$ and $\theta(\mathsf{v},\mathsf{z})$ is quantifier-free. This step is done, as in \cite[Proposition 9]{Axiom_fo_d_KV}, by induction on the complexity of the formula $\phi'$, where the base case and induction steps for all the connectives and quantifiers from \D are proved exactly as in  \cite[Proposition 9]{Axiom_fo_d_KV}, and the induction steps for $\forallo$ follow from the provable equivalences given by Proposition \ref{FOT_der_rules}\ref{FOTd_der_rules_qf_conj_disj} and the rule $\forallo$\textsf{Ext}.


Now, observe that $\theta(\mathsf{v},\mathsf{z})$ is a formula of \D. We then proceed in the same way as in  \cite[Proposition 9]{Axiom_fo_d_KV} to turn $\theta$ into a formula of the form $\forall \mathsf{u}\exists\mathsf{w}(\bigwedge_{i\in I} \dep(\sigma^i_{\mathsf{uwvz}},w_i)\wedge\beta)=\forall \mathsf{u}\exists\mathsf{w}\theta'$, where the variables in each sequence $\mathsf{u}_i\mathsf{w}_i\mathsf{v}_i\mathsf{z}_i$ are from $\mathsf{uwvz}$, each $w_i$ is from $\mathsf{w}$, and $\beta$ is first-order. The formula $\theta$ is semantically equivalent to $\forall \mathsf{u}\exists\mathsf{w}\theta'$, and in the deduction system of \D we can prove (as is done in  \cite[Proposition 9]{Axiom_fo_d_KV}) that $\theta\vdash\forall \mathsf{u}\exists\mathsf{w}\theta'$. Thus, altogether we now have $\phi\equiv Q\mathsf{v}\forall \mathsf{u}\exists\mathsf{w}\theta'$ and $\phi\vdash Q\mathsf{v}\forall \mathsf{u}\exists\mathsf{w}\theta'$.

To turn the formula $Q\mathsf{v}\forall \mathsf{u}\exists\mathsf{w}\theta'$ finally into the required normal form (\ref{nf_dw}), we first swap the order of the existential and universal quantifiers and obtain an equivalent formula of the form $\forallo \mathsf{v}\forall \mathsf{x}\exists \mathsf{y}\big(\bigwedge_{i\in I}\dep(\sigma^i_{\mathsf{xyzv}},y_i)\wedge\alpha\big)$, where each $\sigma^i_{\mathsf{xyzv}}$ is a subsequence of $\mathsf{xyzv}$, each $y_i$ is from $\mathsf{y}$, and $\alpha$ is first-order. This is done by exhaustedly applying the rule $\mathsf{DepI}$ and Proposition \ref{FOTd_der_rules}\ref{forallo_ex_swap} on the syntactic side, and on the semantic side  the equivalences $\exists x\forall y\psi(x,y,\mathsf{z})\equiv\forall y\exists x(\dep(\mathsf{z},x)\wedge \psi)$ (see \cite{Axiom_fo_d_KV}), 
\(\exists x\forallo y\psi(x,y,\mathsf{z})\equiv\forallo y\exists x(\dep(\mathsf{z},x)\wedge \psi)\text{ and }\forall x\forallo y\psi\equiv\forallo y\forall x\psi\) (given by Lemma \ref{wsqf_int}). To conclude the proof, we apply  Proposition \ref{FOTd_der_rules}\ref{forall1_triv_dep_ind}  to remove the variables $\mathsf{v}_i$ quantified by $\forallo$ in the dependence atoms $\dep(\sigma^i_{\mathsf{xyzv}},y_i)$. 
\end{proof}

Recall also from \cite{Axiom_fo_d_KV} that for every \D-sentence $\psi$ in normal form (\ref{nf_d}), there is a first-order sentence $\Psi$ of infinite length (called the {\em game expression} of $\psi$) such that for any countable model $M$, 
\begin{equation}\label{game_expr_equivalence}
M\models\psi~\text{ iff }~M\models\Psi.
\end{equation}
Moreover, the infinitary first-order sentence $\Psi$ can be approximated by some first-order sentences $\Psi_n$ $(n \in\mathbb{N})$ of finite length in the sense that for any recursively saturated (or finite) model $M$, 
\begin{equation}\label{game_approx_equivalence}
M\models\Psi~\text{ iff }~M\models\Psi_n\text{ for all }n\in\mathbb{N}.
\end{equation}
In the system of \D one derives $\psi\vdash_\D\Psi_n$ for any $n\in\mathbb{N}$. {We include the technical definitions of $\Psi$ and $\Psi_n$ in the Appendix; the reader is also referred to \cite[Section 5]{Axiom_fo_d_KV} for further details.} As a simple illustration, the game expression $\Psi_0$ of the \D-sentence $\psi_0=\forall\mathsf{x}\exists \mathsf{y}(\dep(x_0,y_0)\wedge\alpha(\mathsf{x},\mathsf{y}))$ with $\mathsf{x}=\langle x_0,\dots,x_k\rangle$ and $\mathsf{y}=\langle y_0,\dots,y_m\rangle$  is defined as the infinitary first-order sentence:\allowdisplaybreaks
\begin{align*}
\Psi_0=\forall &\mathsf{x}_0\exists \mathsf{y}_0\big(\alpha(\mathsf{x}_0,\mathsf{y}_0)\wedge\tag{round 0}\\
&\forall \mathsf{x}_1\exists \mathsf{y}_1(\alpha(\mathsf{x}_1,\mathsf{y}_1)\wedge( x_{00}=x_{10}\to y_{00}=y_{10})\wedge\tag{round 1}\\
&\quad\dots\quad\dots\\
&\quad\forall \mathsf{x}_n\exists \mathsf{y}_n(\alpha(\mathsf{x}_n,\mathsf{y}_n)\wedge \bigwedge_{i\leq n}(x_{i0}=x_{n0}\to y_{i0}=y_{n0})\wedge\tag{round n}\\
&\quad\quad\quad\dots\quad\quad\dots \quad)\dots\quad\dots)\big)
\end{align*}
Intuitively,  the \D-sentence $\psi_0$ is satisfied on the team $\{\emptyset\}$ over a countable model $M$, iff, for some suitable sequence $\mathsf{F}$ of functions, in the team $X_{\mathsf{\forall x\exists y}}=\{\emptyset\}(M/\mathsf{x})(\mathsf{F}/\mathsf{y})=\{s_0,s_1,\dots,s_n,\dots\}$ generated by the quantifier block $\forall\mathsf{x}\exists\mathsf{y}$, the dependence atom $\dep(x_0,y_0)$ and the first-order formula $\alpha(\mathsf{x},\mathsf{y})$ are satisfied. One may view each assignment $s_i$ as the history of one round in the infinite evaluation game of the first-order formula $\Psi_0$ on the model $M$, played between two players $\forall$ and $\exists$. Now the first $n$ ``layers'' of the first-order formula $\Psi_0$ (which correspond to the first $n$ rounds in the infinite game) describes exactly the team semantics of the formula $\psi_0$ on the team $\{s_0,\dots,s_n\}$, namely, each $s_i$ satisfies $\alpha(\mathsf{x},\mathsf{y})$ or each $\alpha(\mathsf{x}_i,\mathsf{y}_i)$ is true, and any two assignments $s_i,s_j$ agreeing on $x_0$ also agree on $y_0$ or $x_{i0}=x_{j0}\to y_{i0}=y_{j0}$ is true.
{The $n$-approximation $(\Psi_0)_n$ of the infinitary formula $\Psi_0$ is its ``finite slice'' characterizing  ``the first $n$ rounds of the infinite evaluation game". To be more precise, }\allowdisplaybreaks
\begin{align*}
(\Psi_0)_n=\forall &\mathsf{x}_0\exists \mathsf{y}_0\big(\alpha(\mathsf{x}_0,\mathsf{y}_0)\wedge\\
&\forall \mathsf{x}_1\exists \mathsf{y}_1(\alpha(\mathsf{x}_1,\mathsf{y}_1)\wedge( x_{00}=x_{10}\to y_{00}=y_{10})\wedge\\
&\quad\dots\quad\dots\\
&\quad\forall \mathsf{x}_n\exists \mathsf{y}_n(\alpha(\mathsf{x}_n,\mathsf{y}_n)\wedge \bigwedge_{i\leq n}(x_{i0}=x_{n0}\to y_{i0}=y_{n0})))\big).
\end{align*}

{Now, consider a $\D\oplus\FOTd$-sentence $\phi=\forallo\mathsf{v}\psi(\mathsf{v})$ of the form (\ref{nf_dw})  in some vocabulary $\mathcal{L}$, where $\psi(\mathsf{v})$ is a \D-formula. Let $\mathsf{c}$ be a sequence of fresh constant symbols of the same length as $\mathsf{v}$. Observe that $\psi(\mathsf{c}/\mathsf{v})$ is a $\D$-sentence  in normal form (\ref{nf_d}) in the vocabulary $\mathcal{L}(\mathsf{c})$. Write $\Psi(\mathsf{c})$ and $\Psi_n(\mathsf{c})$ for its game expression and $n$-approximation, respectively. 
We then define the game expression $\Phi^\ast$ of $\phi$ as $\forall\mathsf{v}\Psi(\mathsf{v}/\mathsf{c})$, and the $n$-approximation $\Phi^\ast_n$  as $\forall \mathsf{v}\Psi_n(\mathsf{v}/\mathsf{c})$ for each $n\in\mathbb{N}$. We now show that the game expression and its finite approximations for $\D\oplus\FOTd$ satisfy the same properties as the original ones in $\D$, in the sense of the following lemma.}


\begin{lem}\label{rec_sat_lm}
Let $\phi$ be a $\D\oplus\FOTd$-sentence in the normal form (\ref{nf_dw}).
\begin{enumerate}[label=(\roman*)]
\item For any countable model $M$, we have that $M\models\phi$ iff $M\models\Phi^\ast$.
\item For any recursively saturated (or finite) model $M$, we have that $M\models\Phi^\ast$ iff $M\models\Phi^\ast_n$ for all $n\in\mathbb{N}$.
\end{enumerate}
\end{lem}
\begin{proof}
{Let $\phi=\forallo\mathsf{v}\psi(\mathsf{v})$ with $\psi(\mathsf{v})$ a \D-formula in the normal form (\ref{nf_d}). Let $\mathcal{L}$ be the vocabulary of $\phi$ and $\mathsf{c}$ a sequence of fresh variables of the same length as $\mathsf{v}$.  To prove item (i), for any countable $\mathcal{L}$-model $M$, we have that
\begin{align*}
M\models \forallo\mathsf{v}\psi(\mathsf{v})&\iff \text{for any }\mathsf{a}\in M,~M\models_{\{\emptyset\}(\mathsf{a}/\mathsf{v})}\psi(\mathsf{v})\\
&\iff\text{for any }\mathsf{a}\in M,~(M,\mathsf{a})\models\psi(\mathsf{c}/\mathsf{v})\text{ with }\mathsf{c}^{(M,\mathsf{a})}=\mathsf{a}\tag{by an easy inductive proof}\\
&\iff\text{for any }\mathsf{a}\in M,~(M,\mathsf{a})\models\Psi(\mathsf{c})\text{ with }\mathsf{c}^{(M,\mathsf{a})}=\mathsf{a}\tag{by (\ref{game_expr_equivalence})}\\
&\iff M\models\forall \mathsf{v}\Psi(\mathsf{v}/\mathsf{c}).
\end{align*}}

{For item (ii), for any recursively saturated (or finite) $\mathcal{L}$-model $M$, we have that
\begin{align*}
M\models\forall \mathsf{v}\Psi(\mathsf{v}/\mathsf{c})&\iff\text{for any }\mathsf{a}\in M,~(M,\mathsf{a})\models\Psi(\mathsf{c})\text{ with }\mathsf{c}^{(M,\mathsf{a})}=\mathsf{a}\\
&\iff \text{for any }\mathsf{a}\in M,~\text{for any }n\in \mathbb{N},~(M,\mathsf{a})\models\Psi_n(\mathsf{c})\text{ with }\mathsf{c}^{(M,\mathsf{a})}=\mathsf{a}\tag{by (\ref{game_approx_equivalence})}\\
&\iff \text{for any }n\in \mathbb{N},~M\models\forall \mathsf{v}\Psi_n(\mathsf{v}/\mathsf{c}).
\end{align*}
}
\end{proof}


 Next, as in \cite{Axiom_fo_d_KV}, we show that every $n$-approximation $\Phi^\ast_n$ can be derived from  $\phi$.

%

\begin{theorem}\label{approx_provable}
For any $\D\oplus\FOTd$-sentence $\phi$  and any  $n\in\mathbb{N}$, $\phi\vdash\Phi_n^\ast$.
\end{theorem}
\begin{proof}
By Theorem \ref{FOTd_nf_thm}, we may, without loss of generality, assume that the $\mathcal{L}$-sentence $\phi=\forallo\mathsf{v}\psi(\mathsf{v})$ is of the form (\ref{nf_dw}) with $\psi(\mathsf{v})$ a \D-sentence. Let $\mathsf{c}$ be a sequence of fresh constant symbols of the same length as the sequence $\mathsf{v}$. 
{
Since the $\mathcal{L}(\mathsf{c})$-sentence $\psi(\mathsf{c}/\mathsf{v})$ is in the normal form (\ref{nf_d}) for \D-formulas, by the result in \cite{Axiom_fo_d_KV} we have $\psi(\mathsf{c}/\mathsf{v})\vdash\Psi_n(\mathsf{c})$. Moreover, by Proposition \ref{FOT_der_rules}\ref{FOTd_der_rules_forallo_existso_con}, we obtain that $\forallo\mathsf{v}\psi(\mathsf{v})\vdash\psi(\mathsf{c}/\mathsf{v})$. It thus follows that $\forallo\mathsf{v}\psi(\mathsf{v})\vdash\Psi_n(\mathsf{c})$. Since the constant symbols in $\mathsf{c}$ are fresh, by \foralloi we further derive $\forallo\mathsf{v}\psi(\mathsf{v})\vdash\forallo\mathsf{v}\Psi_n(\mathsf{v}/\mathsf{c})$. Next, by $\forallo\forall\mathsf{Trs}$, we conclude that $\forallo\mathsf{v}\psi(\mathsf{v})\vdash\forall\mathsf{v}\Psi_n(\mathsf{v}/\mathsf{c})$, namely $\phi\vdash\Phi_n^\ast$.}
%
\end{proof}

%

\begin{theorem}[Completeness]\label{completeness}
Let $\Gamma\cup\{\theta\}$ be a set of $\D\oplus\FOTd$-formulas with countably many free variables.  If $\theta$ is an  \FOp-formula or a pseudo-flat sentence, then $\Gamma\models\theta\Longrightarrow\Gamma\vdash\theta$. 
\end{theorem}
\begin{proof}
We only give a sketch of the proof, which combines the arguments in \cite{Axiom_fo_d_KV} and in \cite{Yang_neg}. Suppose $\Gamma\models\theta$. By Lemma \ref{cn_defin_in_FOTd}\ref{neg2cn}, we have that $\Gamma,\neg\theta_\ast\models\bot$. Since $\D\oplus\FOTd$ has the same expressive power as \D (Theorem \ref{D=D+FOTd}), by \Cref{compact},  $\D\oplus\FOTd$ is compact. We may thus assume that $\Gamma$ is finite. Now, suppose that $\Gamma\nvdash\theta$. Then  by $\textsf{RAA}_{\ast}$ we have $\Gamma,\neg\theta_\ast\nvdash\bot$. Let $\mathsf{x}$ list all free variables in the finite set $\Gamma$ and in $\neg\theta_\ast$. Put $\phi=\exists \mathsf{x}(\bigwedge\Gamma\wedge\neg\theta_\ast)$, which is a sentence. By $\exists\textsf{I}$ we derive $\phi\nvdash\bot$.

Now, by \Cref{approx_provable}, we have that $\{\Phi_n^\ast\mid n\in\mathbb{N}\}\nvdash\bot$, where each $\Phi_n^\ast$ is the $n$-approximation of the game expression $\Phi^\ast$ of the formula $\phi$. Since restricted to first-order formulas our extended system (or the system of \D as defined in \cite{Axiom_fo_d_KV}) has the same rules as the  system of the usual  first-order logic, we obtain $\{\Phi_n^\ast\mid n\in\mathbb{N}\}\nvdash_{\FO}\bot$. From this point on we follow exactly the same argument as in \cite{Axiom_fo_d_KV}: By the completeness theorem of first-order logic, the set $\{\Phi_n^\ast\mid n\in\mathbb{N}\}$ has a model $M$. It is known that for every infinite model $N$, there exists a recursively saturated countable model $N'$ such that $N$ and $N'$ are elementary equivalent (see e.g., \cite{BarwiseSchlipf76}). We may thus assume w.l.o.g. that $M$ itself is countable and recursively saturated or finite. Now, by Lemma \ref{rec_sat_lm}, we obtain that $M$ is a model for the original $\D\oplus\FOTd$-sentence $\phi=\exists \mathsf{x}(\bigwedge\Gamma\wedge\neg\theta_\ast)$.
Thus, $M\models_X\Gamma$ and $M\models_X\neg\theta_\ast$ for some (nonempty) team $X=\{\emptyset\}(\mathsf{F}/\mathsf{x})$ and sequence $\mathsf{F}$ of suitable functions. This contradict the assumption that $\Gamma,\neg\theta_\ast\models\bot$.
%
\end{proof}

Clearly \FO-formulas  are \FOp-formulas. When applied to \FO-formulas $\alpha$, the $\mathsf{RAA}_\ast$ rule is identical to the $\mathsf{RAA}$ rule, which was included in the original system of \D in \cite{Axiom_fo_d_KV}. From this we conclude that the system given in \cite{Axiom_fo_d_KV} is complete for first-order consequences of \D over (arbitrary) formulas (under the countability assumption).

\begin{cor}\label{dl_complete}
For any set $\Gamma$ of \D-formulas with countably many free variables and \FO-formula $\alpha$,  we have that $\Gamma\models\alpha\Longrightarrow\Gamma\vdash_\D\alpha$, where $\vdash_\D$ is the consequence relation induced by the system of \D given in \cite{Axiom_fo_d_KV}.
\end{cor}
\begin{proof}
We use the same argument as in Theorem \ref{completeness}, where for a first-order formula $\alpha$, we have $\neg\alpha_\ast=\neg\alpha$, and the rule $\mathsf{RAA}_\ast$ applied to $\alpha$ and a set $\Gamma$ of \D-formulas is now identical to the $\mathsf{RAA}$ rule of the system of \D, as given in \cite{Axiom_fo_d_KV}  (listed also in Table \ref{rules_d}). Also note that the game expression $\Phi^\ast$ and its finite approximations $\Phi^\ast_n$ of the $\D$-sentence $\phi=\exists \mathsf{x}(\bigwedge\Gamma\wedge\neg\alpha)$ are then identical to the original game expression $\Psi$ and its finite approximations $\Psi_n$.
For these reasons, the proof does not use any other rules than  those already in the original system of \D, as given in \cite{Axiom_fo_d_KV}.
\end{proof}


Let us end this section with a remark that the $\textsf{RAA}$ rule with respect to (the usual) negation of first-order formulas $\neg\alpha$ is not in general  sound in independence logic (which is \FO extended with independence atoms) or in inclusion logic (which is \FO extended with inclusion atoms). In the systems of these two logics (introduced in \cite{Hannula_fo_ind_13} and \cite{YangInc20}) the corresponding $\textsf{RAA}$ rule has an extra side condition  that  all formulas in the context set $\Gamma$ have to be first-order.  For this reason, the same argument as in Corollary \ref{dl_complete} for dependence logic does not go through for the other two logics. Nevertheless, the system of independence logic in \cite{Hannula_fo_ind_13} together with the $\textsf{RAA}$ rule with respect to $\wcn$ was proved in \cite{Yang_neg} to be  complete (over formulas) for  first-order consequences, and the system of inclusion logic  in \cite{YangInc20} (in which the $\textsf{RAA}$ rule with respect to $\wcn$ and classical formula is derivable) is indeed also complete (over formulas) for first-order consequences.


\section{Applications}\label{sec:app}



In this section, we illustrate the power of our proof systems for \FOT and $\D\oplus\FOTd$ by discussing some examples.

We first consider the system of \FOT. Recall that first-order formulas  themselves (with respect to team semantics) are  expressible in \FOT (Lemma \ref{indepa-def-fot}\ref{foteam2fot}).
Therefore valid entailments $\Delta\models\alpha$ of the usual first-order logic are all derivable in our system for \FOT (via  translations). We now give some examples of derivations of entailments  of arbitrary (not necessarily first-order) \FOT-formulas.


Recall from Lemma \ref{indepa-def-fot}\ref{depatm2fot}\ref{indepatm2fot} that dependence and independence atoms are expressible in \FOT. These atoms correspond, respectively, to  {\em functional dependencies} \cite{Fagin77_prop} and {\em independences} \cite{Geiger-Paz-Pearl_1991} in database theory.
The {\em implication problem} of dependence atoms (i.e., the problem of whether $\Gamma\models\phi$ for a set $\Gamma\cup\{\phi\}$ of dependence atoms) is known to be completely axiomatizable by the set of {\em Armstrong's Axioms} \cite{Armstrong_Axioms}, and  {\em Geiger-Paz-Pearl axioms} \cite{Geiger-Paz-Pearl_1991} are known to axiomatize completely the implication problem of independence atoms. By our completeness theorem, all these axioms are derivable in our system for \FOT. We now demonstrate the derivations of Armstrong's Axioms. 

\begin{exmp}\label{ex_Armstrong}
%
The following clauses, known as Armstrong's Axioms (when written as rules), are derivable in the system of \FOT.
\begin{enumerate}[label=(\roman*)]
\item $\vdash\dep(\mathsf{x},\mathsf{x})$
\item $\dep(\mathsf{x}\mathsf{y},\mathsf{z})\vdash\dep(\mathsf{y}\mathsf{x},\mathsf{z})$
\item $\dep(\mathsf{x}\mathsf{x},\mathsf{y})\vdash\dep(\mathsf{x},\mathsf{y})$
\item $\dep(\mathsf{y},\mathsf{z})\vdash\dep(\mathsf{x}\mathsf{y},\mathsf{z})$
\item $\dep(\mathsf{x},\mathsf{y}),\dep(\mathsf{y},\mathsf{z})\vdash\dep(\mathsf{x},\mathsf{z})$
\end{enumerate}
\end{exmp}
\begin{proof}
Items (ii), (iii) and (iv) follow easily from \incpro. Item (i) is also easy to prove. We only give the detailed proof for item (v).
It suffices to derive that
\begin{align}\label{ex_Armstrong_eq2}
\nonumber
&\forallo \mathsf{u}\mathsf{v}\mathsf{v}'\big((\mathsf{u}\mathsf{v}\subseteq \mathsf{x}\mathsf{y}\wedge\mathsf{u}\mathsf{v}'\subseteq \mathsf{x}\mathsf{y})\cimp\mathsf{v}=\mathsf{v}'\big),\forallo \mathsf{v}\mathsf{w}\mathsf{w}'\big((\mathsf{v}\mathsf{w}\subseteq \mathsf{y}\mathsf{z}\wedge\mathsf{v}\mathsf{w}'\subseteq \mathsf{y}\mathsf{z})\cimp\mathsf{w}=\mathsf{w}'\big)\\
\vdash& \forallo \mathsf{u}\mathsf{w}\mathsf{w}'\big((\mathsf{u}\mathsf{w}\subseteq \mathsf{x}\mathsf{z}\wedge\mathsf{u}\mathsf{w}'\subseteq \mathsf{x}\mathsf{z})\cimp\mathsf{w}=\mathsf{w}'\big).
\end{align}
For any sequences $\mathsf{c},\mathsf{e}$ of constant symbols, by \consi and \incweo we have that
\begin{equation}\label{ex_Armstrong_eq5}
\mathsf{c}\mathsf{e}\subseteq \mathsf{x}\mathsf{z}\vdash\existso \mathsf{v}(\mathsf{c}\mathsf{v}\mathsf{e}\subseteq  \mathsf{x}\mathsf{y}\mathsf{z})~\text{ and }~\mathsf{c}\mathsf{e}'\subseteq \mathsf{x}\mathsf{z}\vdash\existso \mathsf{v}'(\mathsf{c}\mathsf{v}'\mathsf{e}\subseteq  \mathsf{x}\mathsf{y}\mathsf{z}).
\end{equation}
For any sequences $\mathsf{d},\mathsf{d}',\mathsf{e}'$ of constant symbols,  by Proposition \ref{FOT_der_rules}\ref{FOTd_der_rules_forallo_existso_con} we have that
\begin{align}\label{ex_Armstrong_eq3}\nonumber
&\forallo \mathsf{u}\mathsf{v}\mathsf{v}'\big((\mathsf{u}\mathsf{v}\subseteq \mathsf{x}\mathsf{y}\wedge\mathsf{u}\mathsf{v}'\subseteq \mathsf{x}\mathsf{y})\cimp\mathsf{v}=\mathsf{v}'\big),\forallo \mathsf{v}\mathsf{w}\mathsf{w}'\big((\mathsf{v}\mathsf{w}\subseteq \mathsf{y}\mathsf{z}\wedge\mathsf{v}\mathsf{w}'\subseteq \mathsf{y}\mathsf{z})\cimp\mathsf{w}=\mathsf{w}'\big)\\
\vdash&\big((\mathsf{c}\mathsf{d}\subseteq \mathsf{x}\mathsf{y}\wedge\mathsf{c}\mathsf{d}'\subseteq \mathsf{x}\mathsf{y})\cimp\mathsf{d}=\mathsf{d}'\big)\wedge\big((\mathsf{d}\mathsf{e}\subseteq \mathsf{y}\mathsf{z}\wedge\mathsf{d}\mathsf{e}'\subseteq \mathsf{y}\mathsf{z})\cimp\mathsf{e}=\mathsf{e}'\big).
\end{align}
Furthermore, we derive that
\begin{align*}
&\mathsf{c}\mathsf{d}\mathsf{e}\subseteq  \mathsf{x}\mathsf{y}\mathsf{z},\mathsf{c}\mathsf{d}'\mathsf{e}'\subseteq  \mathsf{x}\mathsf{y}\mathsf{z},(\mathsf{c}\mathsf{d}\subseteq \mathsf{x}\mathsf{y}\wedge\mathsf{c}\mathsf{d}'\subseteq \mathsf{x}\mathsf{y})\cimp\mathsf{d}=\mathsf{d}',(\mathsf{d}\mathsf{e}\subseteq \mathsf{y}\mathsf{z}\wedge\mathsf{d}\mathsf{e}'\subseteq \mathsf{y}\mathsf{z})\cimp\mathsf{e}=\mathsf{e}'\\
\vdash&\mathsf{c}\mathsf{d}\mathsf{e}\subseteq  \mathsf{x}\mathsf{y}\mathsf{z}\wedge\mathsf{c}\mathsf{d}'\mathsf{e}'\subseteq  \mathsf{x}\mathsf{y}\mathsf{z}\wedge\mathsf{d}=\mathsf{d}'\wedge \big((\mathsf{d}\mathsf{e}\subseteq \mathsf{y}\mathsf{z}\wedge\mathsf{d}\mathsf{e}'\subseteq \mathsf{y}\mathsf{z})\cimp\mathsf{e}=\mathsf{e}'\big)\tag{$=$\textsf{Sub}, \incpro}\\
\vdash&\mathsf{e}=\mathsf{e}'.\tag{$=$\textsf{Sub}, \incpro}
\end{align*}
Putting this and (\ref{ex_Armstrong_eq5}) together, by \existsoe and the deduction theorem we obtain that
\begin{align*}
&(\mathsf{c}\mathsf{d}\subseteq \mathsf{x}\mathsf{y}\wedge\mathsf{c}\mathsf{d}'\subseteq \mathsf{x}\mathsf{y})\cimp\mathsf{d}=\mathsf{d}',(\mathsf{d}\mathsf{e}\subseteq \mathsf{y}\mathsf{z}\wedge\mathsf{d}\mathsf{e}'\subseteq \mathsf{y}\mathsf{z})\cimp\mathsf{e}=\mathsf{e}'\\
\vdash&\big(\mathsf{c}\mathsf{e}\subseteq \mathsf{x}\mathsf{z}\wedge\mathsf{c}\mathsf{e}'\subseteq \mathsf{x}\mathsf{z}\big)\cimp\mathsf{e}=\mathsf{e}'
\end{align*}
Hence, by (\ref{ex_Armstrong_eq3}), we finally obtain the desired clause (\ref{ex_Armstrong_eq2}).
\end{proof}

Next, without going into much detail we point out that the well-known {\em Arrow's Impossibility Theorem}  \cite{Arrow50} in social choice can be formulated in \FOT, and thus also has a formal proof in our system of \FOT. A recent work \cite{PacuitYang2016} provides  in the language of independence logic a formalization of Arrow's Theorem. The theorem is formulated as the entailment 
\(\Gamma_{\textsf{Arrow}}\models\theta_{\textsf{dictator}},\)
where $\Gamma_{\textsf{Arrow}}$ is a set of assumptions for the aggregation function  imposed in Arrow's Theorem, and $\theta_{\textsf{dictator}}$ is a condition of the form $\delta(v_1)\vvee\dots\vvee\delta(v_n)$ stating that one of the voters among $v_1,\dots,v_n$ is a dictator. Formulas in $\Gamma_{\textsf{Arrow}}\cup\{\delta(v_1),\dots,\delta(v_n)\}$ are either first-order or simple formulas formed using dependence, independence or inclusion atoms.
All of these  formulas  are  expressible in \FOT (see Lemma \ref{indepa-def-fot})). One non-trivial condition that is part of the so-called {\em universal domain} condition is characterized by the atom $\mathsf{All}(\mathsf{v})$ with $\mathsf{v}=\langle v_1,\dots,v_n\rangle$, defined as
\begin{itemize}
\item $M\models_X\mathsf{All}(\mathsf{v})\iff X=\emptyset$ or $\{s(\mathsf{v})\mid s\in X\}=M^n$.
\end{itemize}
Intuitively, the atom $\mathsf{All}(\mathsf{v})$ expresses that the sequence $\mathsf{v}$ takes all possible values from the Cartesian product of the domain. Clearly, $\mathsf{All}(\mathsf{v})$ can be defined in \FOT as
\(\mathsf{All}(\mathsf{v})\equiv\forallo \mathsf{x}(\mathsf{x}\subseteq \mathsf{v}).\)
We illustrate in the next example one type of argument that constitutes an important step in the (formal) proof of Arrow's Theorem. Intuitively, the example states that if each of the variables $v_1,\dots,v_n$ satisfies the universal domain condition, and each variable $v_i$ behaves independently of all the other variables, then the sequence $\langle v_1,\dots,v_n\rangle$ as a whole satisfies the universal domain condition.

\begin{exmp}\label{ex_arrow}
\(\displaystyle\mathsf{All}(v_1),\dots,\mathsf{All}(v_n),\bigwedge_{i=1}^nv_i\perp\langle v_j\rangle_{j\neq i}\vdash\mathsf{All}(v_1,\dots,v_n).\)
\end{exmp}
\begin{proof}
We only give the proof for the case  $n=3$. We show that 
\begin{equation}\label{ex_arrow_eq1}
\begin{split}
\forallo x_1(x_1\subseteq v_1),&\forallo x_2(x_2\subseteq v_2),\forallo x_3(x_3\subseteq v_3),\\
&v_1\perp v_2v_3,v_2\perp v_1v_3,v_3\perp v_1v_2\vdash \forallo x_1x_2x_3(x_1x_2x_3\subseteq v_1v_2v_3).
\end{split}
\end{equation}

Let $c,d,e$ be arbitrary constant symbols. We first derive that
\begin{align*}
&\forallo x_1(x_1\subseteq v_1),\forallo x_2(x_2\subseteq v_2),\forallo x_3(x_3\subseteq v_3)\\
\vdash& c\subseteq v_1\wedge d\subseteq v_2\wedge e\subseteq v_3 \tag{\foralloe}\\
\vdash& c\subseteq v_1\wedge \existso y(dy\subseteq v_2v_3)\wedge e\subseteq v_3\tag{\consi, \incweo}.
\end{align*}
Next, for an arbitrary constant symbol $e'$, we derive that
\begin{align*}
&c\subseteq v_1,de'\subseteq v_2v_3,v_1\perp v_2v_3\\
\vdash &c\subseteq v_1\wedge de'\subseteq v_2v_3\wedge \big((c\subseteq v_1\wedge de'\subseteq v_2v_3)\cimp cde'\subseteq v_1v_2v_3\big)\tag{\foralloe}\\
\vdash&cde'\subseteq v_1v_2v_3.
\end{align*}
Furthermore, we derive that
\begin{align*}
&e\subseteq v_3,cde'\subseteq v_1v_2v_3,v_3\perp v_1v_2\\
\vdash &e\subseteq v_3\wedge cd\subseteq v_1v_2\wedge \big((e\subseteq v_3\wedge cd\subseteq v_1v_2)\cimp cde\subseteq v_1v_2v_3\big)\tag{\incpro, \foralloe}\\
\vdash&cde\subseteq v_1v_2v_3.
\end{align*}
Putting all these three derivations together, by \existsoe we conclude that (\ref{ex_arrow_eq1}) holds.
\end{proof}

Let us now turn to the system of $\D\oplus\FOTd$. 
First note that Armstrong's Axioms (Example \ref{ex_Armstrong}) are also derivable in $\D\oplus\FOTd$, from the rules $\mathsf{DepWI}$ and $\mathsf{DepWE}$.
%
%
%
%
%
%
For a second example, consider the $\D\oplus\FOTd$-sentence
\[\phi_{\infty}:=\existso z\forall x\exists y(\dep(y,x)\wedge y\neq z)\footnotemark\]
that satisfies $M\models\phi_{\infty}$ iff $|M|$ is infinite.
Intuitively, the sentence expresses that there is a function $f(x)=y$ (guaranteed by the team semantics of $\forall x\exists y$) on $M$ that is injective (since $\dep(y,x)$) but not surjective (since $y\neq z$). Consider also the \FOTd-sentence \footnotetext{This sentence is a slight modification of an equivalent one in \D defined in \cite[Section 4.2]{Van07dl}.}
\[\phi_{\geq n}:=\existso v_1\dots \existso v_n\bigwedge_{i<j\leq n}v_i\neq v_j.\]
Clearly, $M\models\phi_{\geq n}$ iff $|M|\geq n$.
\begin{exmp}
$\phi_{\infty}\vdash\phi_{\geq n}$ for all $n\in\mathbb{N}$.
\end{exmp}
\begin{proof}
We only illustrate the proof for  $n=3$. By Proposition \ref{FOT_der_rules}(\ref{FOTd_der_rules_forallo_existso_con}) it suffices to derive that $\phi_{\infty}\vdash c\neq d\wedge d\neq e\wedge c\neq e$ for constant symbols $c,d,e$, which by \existsoe is reduced to deriving that
\begin{equation*}
\forall x\exists y(\dep(y,x)\wedge y\neq c)\vdash c\neq d\wedge d\neq e\wedge c\neq e,
\end{equation*}
which is further equivalent to
\begin{equation}\label{infty_der_eq1}
\forall x\exists y\exists z(\dep(y,z)\wedge y\neq c\wedge x=z)\vdash c\neq d\wedge d\neq e\wedge c\neq e,
\end{equation}
Now, by $\mathsf{DepE}$, we have that
\[\begin{split}
&\forall x\exists y\exists z(\dep(y,z)\wedge y\neq c\wedge x=z)\\
\vdash\,&\forall x_0\exists y_0\exists z_0\big(y_0\neq c\wedge x_0=z_0\wedge \forall x_1\exists y_1\exists z_1(y_1\neq c\wedge x_1=z_1
\wedge (y_0=y_1\to z_0=z_1))\big)\\
\vdash\,&\forall x_0\exists y_0\big(y_0\neq c\wedge \forall x_1\exists y_1(y_1\neq c\wedge (y_0=y_1\to x_0=x_1))\big).
\end{split}\]
Then, to prove (\ref{infty_der_eq1}) by $\forall \textsf{E}$, $\exists \textsf{E}$ we need to derive (by instantiating $x_0=c$, $y_0=d=x_1$, $y_1=e$) that
\(d\neq c\wedge e\neq c\wedge (d=e\to c=d)\vdash c\neq d\wedge d\neq e\wedge c\neq e.\)
But this is clear.
\end{proof}

Finally, we derive that the constancy atom $\dep(x)$ is equivalent to $\existso y(y=x)$, even though our completeness theorem (\Cref{completeness}) actually does not apply to the equivalence, as $\dep(x)$ is not in $\FO(\forallo)$ or a pseudo-flat sentence. Note also that the consequences of many derivable clauses in \Cref{sec:fotd} (especially in Proposition \ref{FOTd_der_rules}) are not in $\FO(\forallo)$ or pseudo-flat sentences either. These illustrate that our deduction system can actually derive more (sound) consequences than what is guaranteed by \Cref{completeness}.

\begin{exmp}
$\dep(x)\dashv\vdash\existso y(y=x)$.
\end{exmp}
\begin{proof}
By \existsoi, we have $x=x,\dep(x)\vdash\existso y(y=x)$. Since $\vdash x=x$ by $=$\textsf{I}, we conclude that $\dep(x)\vdash\existso y(y=x)$.
Conversely, we have $y=x,\dep(y)\vdash\dep(x)$ by $=$\textsf{Sub}. Thus, by \existsoe we conclude that $\existso y(y=x)\vdash \dep(x)$.
\end{proof}

\section{Conclusion and open problems}

In this paper, we have defined a team-based logic \FOT that characterizes exactly the elementary (or first-order) team properties. 
This expressivity result implies that the essentially first-order logic \FOT can be axiomatized completely. We then introduced a sound and complete system of natural deduction for \FOT. Syntactically \FOT is first-order logic with a  set of weaker connectives and quantifiers extended with inclusion atoms.  We have illustrated that a number of interesting first-order properties can be expressed naturally in \FOT, including different atoms of dependencies. We have given as examples formal derivations of Armstrong's Axioms, as well as certain interaction between different dependency notions in the context of Arrow's Theorem. Logics based on team semantics are usually so strong in expressive power that complete axiomatizations are not possible to obtain (e.g., dependence logic is expressively equivalent to existential second-order logic and thus not effectively axiomatizable, etc.). Finding sufficiently expressive language with good computational properties is one of the main challenges for team-based logics. It is our hope that the logic \FOT defined in this paper will lead to fruitful further developments in this direction. In a different setting that uses partly the idea of team semantics, a recent work \cite{BaltagBenthem2020} by Baltag and van Benthem also proposed a weak logic of dependence with a complete axiomatization.

We have also defined a logic that characterizes exactly the downward closed elementary (or first-order) team properties, the fragment of \FOT that we call \FOTd. We studied the axiomatization problem of $\D\oplus\FOTd$, dependence logic \D enriched with the weak connectives and quantifiers in \FOTd. On the basis of the known partial axiomatization of dependence logic, we have defined a system of natural deduction for $\D\oplus\FOTd$ that is  complete for a certain class of pseudo-flat consequences. First-order formulas belong to this class, and thus our result also implies that the system of dependence logic in \cite{Axiom_fo_d_KV} is complete for first-order consequences (over formulas), 
resolving an open problem in \cite{Axiom_fo_d_KV} in the affirmative. Generalizing the axiomatization of $\D\oplus\FOTd$ we obtained in this paper to a larger fragment is left as future work. Also, finding similar partial axiomatizations of independence logic enriched with inclusion atoms and logical constants in \FOT would be an interesting addition to this line of research.

In the proofs of the (strong) completeness theorem of the above systems, we made use of the compactness of the logics with respect to formulas (Theorem \ref{compact}). The proof we gave for compactness  has an extra assumption that the set $\Gamma$ of formulas has countably many free variables, and the argument does not directly work for $\Gamma$ having uncountably many free variables. How to prove compactness for formulas in general is unclear.

Having identified the logics \FOT and \FOTd that characterize exactly all first-order and all downward closed elementary team properties, respectively, it is natural to ask what is a logic that characterize exactly all union closed elementary team properties. There does not seem to be an obvious candidate for such a logic, since the usual disjunction $\vee$ and quantifiers $\forall,\exists$ appear to be too strong (in expressive power) to be allowed in such a logic, whereas the weak disjunction $\vvee$ and weak existential quantifier $\existso$ do not preserve union closure. It was proved in \cite{inclusion_logic_GH} that every first-order union closed team property can be defined by a so-called {\em myopic} \FO-sentence, which is further equivalent to a formula in inclusion logic. By using this result, it may be possible to identify a fragment of \FOT that characterizes union close elementary team properties. We leave this for future work.



\section{Appendix}

{Let $\psi=\forall \mathsf{x}\exists \mathsf{y}\big(\bigwedge_{i\in I}\dep(\sigma^i_{\mathsf{xy}},y_i)\wedge\alpha(\mathsf{x},\mathsf{y})\big)$ be a $\D$-sentence 
in the normal form (\ref{nf_d}) given in \cite{Van07dl}, where each $\sigma_{\mathsf{xy}}^i$ is a subsequence of $\mathsf{xy}$,  each $y_i$ is from $\mathsf{y}$, and $\alpha$ is  first-order. The game expression $\Psi$ of $\psi$ is the following infinitary sentence:
\begin{align*}
\Psi=&\forall \mathsf{x}_0\exists \mathsf{y}_0\Big(\alpha(\mathsf{x}_0,\mathsf{y}_0)
\wedge\forall \mathsf{x}_1\exists \mathsf{y}_1\big(\alpha(\mathsf{x}_1,\mathsf{y}_1)\wedge \gamma_1\\
&\quad\quad\quad\quad\quad\quad\dots ~\dots
~~~~\wedge\forall \mathsf{x}_n\exists \mathsf{y}_n\big(\alpha(\mathsf{x}_n,\mathsf{y}_n)\wedge \gamma_n\wedge~\dots
~\dots\quad\big)\dots\big)\big)\Big),
\end{align*}
where each $\mathsf{x}_\xi=\langle x_{\xi 1},\dots,x_{\xi m}\rangle$, $\mathsf{y}_\xi=\langle y_{\xi 1},\dots,y_{\xi k}\rangle$, 
and 
\[\gamma_n=\bigwedge\{\sigma^i_{\mathsf{x}_\xi\mathsf{y}_\xi}=\sigma^i_{\mathsf{x}_n\mathsf{y}_n}\to y_{\xi i}=y_{ni}\mid i\in I, ~0\leq \xi\leq n\}.\]}

{For every $n\in \mathbb{N}$, the $n$-approximation $\Psi_n$ of $\Psi$ is the finite first-order sentence:
\[
\Psi_n=\forall \mathsf{x}_0\exists \mathsf{y}_0\big(\alpha(\mathsf{x}_0,\mathsf{y}_0)
\wedge\forall \mathsf{x}_1\exists \mathsf{y}_1\big(\alpha(\mathsf{x}_1,\mathsf{y}_1)\wedge \gamma_1
\dots 
\wedge\forall \mathsf{x}_n\exists \mathsf{y}_n(\alpha(\mathsf{x}_n,\mathsf{y}_n)\wedge \gamma_n)\dots\big)\big).
\]}

\subsection*{Acknowledgements} 
The authors would like to thank Aleksi Anttila, Noora Ethol\'{e}n, Miika Hannula, Tapani Hyttinen, and Jouko V\"{a}\"{a}n\"{a}nen for interesting discussions on topics related to the paper. 

Both authors were supported by grant 308712 of the Academy of Finland, and the second author was supported also by grant 330525 of  the Academy of Finland and Research Funds of the University of Helsinki.

\bibliographystyle{asl}
\bibliography{fan}

\end{document}